\documentclass[11pt,times,letter]{article}

\usepackage{amsmath, amsfonts,amsthm,amssymb,bm}
\usepackage{dsfont}
\usepackage[english]{babel}
\usepackage[normalem]{ulem}
\usepackage{latexsym,amssymb,amsfonts,graphicx}
\usepackage{amsmath,dsfont}
\usepackage{verbatim}
\usepackage{mathrsfs}
\usepackage{bm}
\usepackage{color}
\usepackage{epsfig}
\usepackage{url}
\usepackage{bm}
\usepackage{natbib}
\usepackage[colorlinks,linkcolor=blue,citecolor=blue,anchorcolor=blue]{hyperref}

\definecolor{mycolor-green}{RGB}{199, 237, 204}
\definecolor{mycolor-yellow}{RGB}{250, 249, 222}
\definecolor{mycolor-blue}{RGB}{220, 226, 241}
\definecolor{mycolor-gray}{RGB}{234, 234, 239}

\def\esssup_#1{\underset{#1}{\mathrm{ess\,sup\, }}}
\def\essinf_#1{\underset{#1}{\mathrm{ess\,inf\, }}}
\def\argmax_#1{\underset{#1}{\mathrm{arg\,max\, }}}
\def\argmin_#1{\underset{#1}{\mathrm{arg\,min\, }}}

\newtheorem{theorem}{Theorem}[section]
\newtheorem{definition}{Definition}[section]
\newtheorem{proposition}[theorem]{Proposition}
\newtheorem{remark}[theorem]{Remark}
\newtheorem{lemma}[theorem]{Lemma}
\newtheorem{corollary}[theorem]{Corollary}

\newtheorem{assumption}[theorem]{Assumption}

\allowdisplaybreaks[4]

\title{Robust mean field control: an application to optimal execution under composite uncertainty}
\author{Huafu Liao \thanks{School of Mathematical Sciences, Dalian University of Technology. Email: hfliao@dlut.edu.cn.}\and Shuhui Liu \thanks{Department of Applied Mathematics, The Hong Kong Polytechnic University. Email: shuhuiliusdu@gmail.com.}\and Chenchen Mou \thanks{Department of Mathematics, City University of Hong Kong. Email: chencmou@cityu.edu.hk.}\and Defeng Sun\thanks{Department of Applied Mathematics, The Hong Kong Polytechnic University. Email: defeng.sun@polyu.edu.hk.}}

\begin{document}
\maketitle
\begin{abstract}
We provide a framework for robust mean field control problems that describe multi-dimensional optimal liquidation problems under uncertainty from both the underlying stochastic process and the deterministic model parameters. The verification results are established with Hamilton-Jacobi-Bellman-Isaacs (HJBI) equations where the variables are probability measures and the Hamiltonian nonlinearly involves the joint distribution of position and momentum. Using novel a priori estimates, we establish the well-posedness of the HJBI equations featuring general or quadratic Hamiltonians that are neither displacement convex nor concave in their momentum. The a priori estimates and well-posedness results are extended during their application to optimal liquidation problems, where we allow the Hamiltonian to have derivatives of linear growth and solve the constrained multi-dimensional linear quadratic optimal liquidation problem under composite uncertainty.

\vspace{0.3cm}
\noindent{\textbf{2020 AMS Mathematics subject classification}: 49N80; 49L12; 91G80.}

\vspace{0.3cm}
\noindent{\textbf{Keywords}:}\quad{robustness; mean field control; Hamilton-Jacobi-Bellman-Isaacs equation; multi-dimensional optimal execution}.
\end{abstract}

\section{Introduction}\label{intro}
\noindent In classical financial theories such as portfolio selection and option pricing, it is fundamental to assume a frictionless market and perfectly known model parameters. However, the actual trading may operate under different rules. In the real world, large investors' trading usually comes with unnegligible transaction cost and the exact model parameters could be ambiguous to investors, especially when the trading is conducted with high frequency or at a large amount. The classic Black-Scholes model for portfolio selection in a frictionless market leads to a position process with infinite variation, whereas in a frictional market such strategy could yield infinite transaction cost. Different branches of models have been developed in order to tackle the cost in liquidity. Early models dealing with transaction cost can trace back to at least \cite{Leland85} in the context of option pricing, other literature concerning models with market friction includes \cite{Cvitanic1996,Barles1998,Cvitanic1999}. In these models, the market friction comes from the transaction cost that is modeled as either fixed or proportional to the trading volume. Later on, starting from \cite{Bertsimas1998} and \cite{Almgren2001}, a different approach is developed to model the friction due to the price impact from trading. At its core, the price impact is considered to consist of temporary effect and permanent effect. The former effects the price only at the moment of execution and has no further influence on the price process, whereas the latter is linearly integrated into the price process and has effect henceforth.

The model uncertainty faced by the investor is another practical concern. Dating back to \cite{Knight21}, the model uncertainty, also known as Knightian uncertainty, has been identified different from ``risk'' and studied for decades. The application of model uncertainty in economics literature can go back to \cite{Gilboa89}. Roughly speaking, in these models, the investor is uncertain about the exact distribution of underlying stochastic processes, so she has to choose between a family of probability measures according to certain criterion. One strand of approach starts with \cite{Hansen23} in the context of portfolio selection, where the objective functional contains a penalty term that penalizes the difference between the benchmark probability measure and the chosen one. The investor has prior knowledge on the underlying stochastic processes, which constitutes the benchmark probability measure, but for fear of model misspecification, she would choose between a family of plausible alternatives according to the mentioned penalty function. We also refer to e.g. \cite{Maenhout04,Maenhout08,Bo17} as well as the reference therein.

In view of the aforementioned fundamental practical concern, we carry out an intensive study on a mean field control problem with robustness which could describe the optimal liquidation under uncertainty. Proposed independently by \cite{Lasry07} and \cite{Huang06}, the mean field games and mean field control theory has been rapidly developing for over a decade, with applications in various fields. An incomplete references on the interplay among mean field effect, model uncertainty and optimal liquidation include \cite{Nystrom14,Bauso16,Cartea2017,Huang17,Bismuth2019,Ismail19,Popier19,Pham22,Fu23,Delarue26,Delarue2026}. Our model differs from previous ones in two aspects:
\begin{itemize}
 \item In terms of the model uncertainty, we consider the composite uncertainty that comes from both the underlying stochastic processes and deterministic parameters. In control literature \cite{Sparks1997,Hara2001,Basin2009,Huang17} and so on, models with uncertainty from deterministic disturbance are studied under different circumstances. On the other hand, most existing literature on finance, such as \cite{Maenhout04,Maenhout08,Bo17,Ismail19,Pham22,Delarue26,Delarue2026} and alike, focuses on the uncertainty coming from the underlying stochastic processes. Such method neglects the possible ambiguity from the unknown parameters that are supposed to be deterministic or from an independent probability space. As a remedy, our approach models the composite uncertainty by allowing the large investor to choose between a set of plausible deterministic parameters besides underlying stochastic processes and evaluate the worst case objective functional. When dealt with dynamic programming principle, the composite uncertainty distinguishes significantly in the sense that it leads to HJBI equations with Wasserstein derivatives even when there is no mean field term in the dynamics or objective functional.
 \item With regard to‌ modeling optimal liquidation, we nonlinearly incorporate the distribution of position and transaction rate into our model. Such incorporation can better describe the price impact from the transaction of a large investor. When the large investor is selling fast, it seems natural that such action would yield more child orders of selling in the market and thus further decrease the price. It is inconvenient to work with the exact behavior of each child order, hence considering its distribution is a tractable way. In that respect, our model incorporates the distribution of the large investor's transaction rate into the price process, so that the child orders' behavior can be described. Moreover, the incorporation of distribution could be nonlinear, meaning that both the large investor and child orders could have nonlinear price impacts. We refer to \cite{Gueant17} and the references therein for the logic behind the nonlinear price impact. We also put the large investor's distribution of position into the objective function, thus enabling the consideration of the mean variance of position that corresponds to the operation uncertainty (see e.g. \cite{Bensoussan2019}).
\end{itemize}
 
Next, we explain our mathematical and technical contributions. Our framework and verification results for robust mean field control are based on dynamic programming principle together with the solution $V$ to HJBI equation in the following:
\begin{align}\label{HJB}
 \left\{\begin{aligned}\partial_tV+\frac12\int_{\mathbb R^d}{\rm tr}\big[\sigma^\top\sigma D_x\partial_\mu V(t,\mu,x)\big]\mu(dx)+\mathcal H(\tilde\mu_t)=0,\ t\in[0,T),\ \mu\in\mathcal P_2(\mathbb R^d),\\
 V(T,\mu)=U(\mu),\quad\tilde\mu_t:=(Id,\partial_\mu V(t,\mu,\cdot))\sharp\mu,\end{aligned}\right.
\end{align}
where $V$ has probability measure variables $\mu\in\mathcal P_2(\mathbb R^d)$ and under appropriate conditions,
{\small\begin{align*}
 \mathcal H\big({\rm Law}(X,Y)\big)&:=\inf_{\theta\in L^2(\Omega,\mathcal F,\mathbb P;\mathbb R^k)}\sup_{\substack{\eta\in L^2(\Omega,\mathcal F,\mathbb P;\mathbb R^{k_1})\\\alpha\in\mathbb R^{k_2}}}\mathbb E\big[a\big(\theta,\eta,\alpha,X,{\rm Law}(X,\theta,\eta)\big)\cdot Y+f\big(\alpha,{\rm Law}(X,\theta,\eta)\big)\big]\notag\\
 =&\sup_{\substack{\eta\in L^2(\Omega,\mathcal F,\mathbb P;\mathbb R^{k_1})\\\alpha\in\mathbb R^{k_2}}}\inf_{\theta\in L^2(\Omega,\mathcal F,\mathbb P;\mathbb R^k)}\mathbb E\big[a\big(\theta,\eta,\alpha,X,{\rm Law}(X,\theta,\eta)\big)\cdot Y+f\big(\alpha,{\rm Law}(X,\theta,\eta)\big)\big],\\
 X,Y\in L^2(\Omega&,\mathcal F,\mathbb P;\mathbb R^d).
 \end{align*}}
The notation and conditions will be further explained in Section \ref{robust-mfc-problem}. We note that \cite{Ismail19,Pham22} also establish verification results for robust mean field control problems, whereas the idea is based on the martingale approach and the robustness is modeled by uncertainty from the underlying stochastic processes. In \cite{Liao24}, the Hamiltonians therein nonlinearly incorporate the distribution of momentum, whereas the Hamiltonian is displacement concave in momentum, i.e., for arbitrary square integrable $\mathbb R^d\times\mathbb R^d\times\mathbb R^d$ value random variables $(X,Y_1,Y_2)$,
\begin{align}\label{displacement-concave-in-p}
 \mathcal H\big({\rm Law}(X,\gamma Y_1+(1-\gamma)Y_2)\big)\geq\gamma\mathcal H\big({\rm Law}(X,Y_1)\big)+(1-\gamma)\mathcal H\big({\rm Law}(X,Y_2)\big),\ \gamma\in[0,1],
\end{align}
which does not generally hold in our context. In \cite{Cosso16} the master Bellman-Isaacs equation with similar feature is studied in the context of zero-sum stochastic differential games of generalized McKean-Vlasov type, where the authors proved related dynamic programming principle for open-loop controls and show that the upper and lower value functions are viscosity solutions to a corresponding upper and lower Master Bellman–Isaacs equation. Recently in \cite{Delarue26}, robust mean field control problems are studied via the stochastic maximum principle, where equilibria is obtained in the form of open loop controls. To the best of our knowledge, we are the first to set up the framework and establish verification results on a robust mean field control problem via dynamic programming principle and HJBI equations, where the uncertainty can be composite and robust optimal feedback functions are derived.

Our main mathematical contribution is establishing the well-posedness of HJBI equation \eqref{HJB} and extending the results ‌in the aspect of‌ growth condition and terminal constraint. Even when there is no mean field terms in the underlying dynamics and cost functions, the composite uncertainty still make \eqref{HJB} an HJBI equations with Hamiltonians that are nonlinear with respect to the distribution of momentum. Hence \eqref{HJB} is essentially outside the scope of existing results on classical solutions to standard mean field control problems. In addition, the Hamiltonians for robust control problems could be neither concave nor convex in momentum, which can also be observed in \cite{Delarue26}. Such convexity-concavity feature is different from mean field control problems without robustness in \cite{BENSOUSSAN2019,Gangbo22,Gangbo22AOP,Liao24} and so on, thus causing difficulty. To address such difficulty, we propose a generalized displacement convexity-concavity property on Hamiltonians. Then, highlighted with novel a priori estimates, we establish the well-posedness of HJBI equations in two circumstances characterized by the Hamiltonian, namely Hamiltonians of general form and quadratic form:
\begin{itemize}
 \item For the HJBI equations with general Hamiltonians, we obtain the a priori estimates of solution $V$ and thus the well-posedness. This is done by analyzing the value functions of the corresponding $N$-particle systems $(V_N)_{N\geq1}$, which is in the spirit of \cite{Liao24}. However, here the a priori estimates are established in a significantly different way. The convexity-concavity structure of the Hamiltonian brings essential difficulty analyzing the second order derivatives of $V$.  As a result, the crucial a priori estimates on solution has to be obtained differently from the cases without robustness. From our perspective, such difference has caused essential difficulty analyzing $(\nabla^2V_N)_{N\geq1}$. Towards that end, we have to resort to a different displacement convexity structure as well as the related comparison principle for symmetric stochastic Riccati equations in Lemma \ref{lem-comparison}, which is different from \cite{Liao24}, so that  Theorem \ref{prop-eigen} can be obtained. Inspired by \cite{Cosso23}, we also develop an approximation argument that preserves the displacement convexity property in Lemma \ref{dis-con-app}. This approximation argument enables us to extend the scope of solvable cases in Theorem \ref{lqd-general-HJB-N-wp-1}, where the Hamiltonian has derivatives of linear growth.
 
 \item For the linear quadratic case where the Hamiltonian is quadratic in its variable, our a priori estimates contribute to the literature on optimal liquidation of multiple assets with strict liquidation constraint, even when the mean field effect is not considered. For single asset, the optimal liquidation with strict liquidation constraint has been studied in \cite{Ankirchner14,Kruse16,Horst19,Popier19,Graewe21} and so on. For the case with multiple assets, the volume of related literature is relatively smaller. As far as we are aware of, the related works include \cite{Kratz15,Horst19}, 
both of which belong to certain class of linear quadratic optimal control problems with terminal constraint and without robustness. However, they rely on the comparison results on stochastic Riccati equations that only work for the optimal control problems with Hamiltonians that are concave in momentum. Therefore we have difficulty applying the mathematical tools therein. In that respect, our comparison principle for stochastic Riccati equations works well with robust optimal control problems, where the Hamiltonians is not necessarily concave in momentum, so that the desired a priori estimates can be established. Our a priori estimates are novel even for the classic robust optimal liquidation problems without mean field term. Then we solve the optimal liquidation problem with singular terminal condition by considering the limit of regular ones with penalty. The above limit is utilized to obtain approximate solution numerically and illustrate how the optimal strategy and position is impacted by composite uncertainty and the variance of position.
 \end{itemize}
 We note that the above approach for a priori estimates also enables us to establish the well-posedness of HJBI equations with common noise or degenerated individual noise in the model.
 
The rest of the paper is organized as follows. Our main results are explained in Section \ref{main results}. In Section \ref{robust-mfc-problem} we set up the framework for robust mean field control problems and establish the verification results. Section \ref{solvable-cases} analyses two classes of the corresponding solvable HJBI equations. Section \ref{application} studies the optimal liquidation problems with uncertainty for a large investor, where results in Section \ref{solvable-cases} are applied and extended. In Section \ref{numerical examples} we study the numerics of a linear quadratic example. Section \ref{Proofs} contains proofs for the results in previous sections.

\section{The model and main results}\label{main results}
\subsection{The robust mean field control problem}\label{robust-mfc-problem}
To study robust optimal control in the mean field context, we will work with the feedback function in the following sense:
\begin{align}\label{feedback-theta}
 \theta:\ [0,T]\times\mathbb R^d\times\mathcal P_2(\mathbb R^d)\ \to\ \mathbb R^k,
\end{align}
 as well as the corresponding controlled process
\begin{align}\label{dynamic}
 dX^{\theta,\eta,\alpha}_s=a\big(\theta_s,\eta_s,\alpha_s,X^{\theta,\eta,\alpha}_s,\mu_{(X^{\theta,\eta,\alpha}_s,\theta_s,\eta_s)}\big)ds+\sigma dW_s,\ s\in[t,T],\quad X^{\theta,\eta,\alpha}_t=\xi\sim\mu\in\mathcal P_2(\mathbb R^d),
\end{align}
 where $a:[0,T]\times\mathbb R^k\times\mathbb R^{k_1}\times\mathbb R^{k_2}\times\mathbb R^d\times\mathcal P_2(\mathbb R^d\times\mathbb R^k\times\mathbb R^{k_1})\to\mathbb R^d$, the feedback control $\big(\theta_s:=\theta(s,X^{\theta,\eta,\alpha}_s,\mu_{X^{\theta,\eta,\alpha}_s})\big)_{s\in[t,T]}$ is generated by the feedback function $\theta$, and $(\eta_s,\alpha_s)_{s\in[0,T]}$ comes from an ambiguity set that might depend on $\theta$, where $(\eta_s)_{s\in[t,T]}$ and $(\alpha_s)_{s\in[t,T]}$ are stochastic and deterministic process respectively. Here $(\eta_s)_{s\in[t,T]}$ corresponds to the uncertainty from the underlying stochastic processes in the model. In particular, $(\eta_s)_{s\in[t,T]}$ could incorporate the drift of $(W_s)_{s\in[t,T]}$ that relates to the change of probability. Besides, $(\alpha_s)_{s\in[t,T]}$ in dynamic \eqref{dynamic} describes the ambiguity on the model parameters that are supposed to be deterministic. The process $(\alpha_s)_{s\in[t,T]}$ could also describe parameters from probability space that is independent of $\sigma(W_s,s\in[t,T])$. After taking the conditional expectation with respect to $\sigma(W_s,s\in[t,T])$, the resulting control problem is reduced to a pathwise one with deterministic $(\alpha_s)_{s\in[t,T]}$. Throughout the paper we will denote by $\mu_\Gamma$ the distribution of any random variable $\Gamma$.

Next we formulate the dynamic version of the mean field optimal control problem with robustness, starting with the admissible set. 
\begin{definition}\label{admissible-set}
 Let $t\in[0,T]$. The $t$-admissible control set $\mathcal U_t$ consists of the tuple $(\Omega,\mathcal F,\{\mathcal F_s\}_{s\in[t,T]},\mathbb P,W,\theta,\mathcal V^\theta_t)$ satisfying the following
\begin{itemize}
    \item $(\Omega,\mathcal F,\{\mathcal F_s\}_{s\in[t,T]},\mathbb P)$ is a filtered probability space satisfying the usual conditions.
    \item $(W_s)_{s\in[t,T]}$ is a standard $d$-dimensional Brownian motion defined on $(\Omega,\mathbb P,\mathcal F)$ with $W_t=0$ almost surely and $\mathcal F_s:=\sigma\big(W_u,u\in[t,s]\big)$ augmented by all the $\mathbb P$-null sets in $\mathcal F$.
    \item $\theta$ is a feedback function as in \eqref{feedback-theta}.
    \item Given feedback function $\theta$, $\mathcal V^\theta_t$ is the subset of all the processes $(\eta_s,\alpha_s)_{s\in[0,T]}$ valued in $\mathbb R^{k_1}\times\mathbb R^{k_2}$, with $(\eta_s)_{s\in[0,T]}$ $\{\mathcal F_s\}_{t\leq s\leq T}$-adapted and $(\alpha_s)_{s\in[0,T]}$deterministic, such that \eqref{dynamic} admits a solution with pathwise uniqueness on $(\Omega,\mathcal F,\{\mathcal F_s\}_{s\in[t,T]},\mathbb P)$. Here in \eqref{dynamic}, $\big(\theta_s:=\theta\big(s,X^{\theta,\eta,\alpha}_s,\mu_{X^{\theta,\eta,\alpha}_s}\big)\big)_{s\in[t,T]}$ is an $\{\mathcal F_s\}_{t\leq s\leq T}$-adapted feedback strategy on $(\Omega,\mathcal F,\mathbb P)$ that is generated by $\theta$ and $(\eta_s,\alpha_s)_{s\in[t,T]}$ is from $\mathcal V^\theta_t$.
    \item Given feedback function $\theta$, the set $\mathcal V^\theta_t$ should include those processes generated by Lipschitz feedback functions $(\eta_s,\alpha_s)_{s\in[t,T]}=\big((\hat\eta(s,X^{\theta,\eta,\alpha}_s,\mu_{X^{\theta,\eta,\alpha}_s}),\hat\alpha(s,\mu_{X^{\theta,\eta,\alpha}_s})\big)_{s\in[t,T]}$, where $\hat\eta:[0,T]\times\mathbb R^d\times\mathcal P_2(\mathbb R^d)\to\mathbb R^{k_1}$, $\hat\alpha:[0,T]\times\mathcal P_2(\mathbb R^d)\to\mathbb R^{k_2}$.
\end{itemize}
\end{definition}
\begin{remark}The following remarks are in order:
\begin{enumerate}
 \item According to standard results on mean field SDE, the admissible set $\mathcal U_t$ includes all Lipschitz feedback functions.
 \item The last item is a technical requirement for the verification theorem later.
 \item We follow the convention that using $\theta$ to represent the tuple $(\Omega,\mathcal F,\{\mathcal F_s\}_{s\in[t,T]},\mathbb P,W,\theta,\mathcal V^\theta_t)\in\mathcal U_t$.
\end{enumerate}
\end{remark}
For each $\theta\in\mathcal U_t$ we may introduce the objective functional with robustness. Roughly speaking, such functional is defined in two steps. First, define the following auxiliary functional:
\begin{align*}
 \tilde J(t,\mu,\theta,\eta,\alpha)=U\big(\mu_{X^{\theta,\eta,\alpha}_T}\big)+\int_t^Tf\big(\alpha_s,\mu_{(X^{\theta,\eta,\alpha}_s,\theta_s,\eta_s)}\big)ds.
\end{align*}
Second, introduce robustness and define the objective functional for the feedback strategies $\theta\in\mathcal U_t$ as follows:
\begin{align}\label{ob-fun-1}
 J(t,\mu,\theta)=\sup_{(\eta,\alpha)\in\mathcal V^\theta_t} \tilde J(t,\mu,\theta,\eta,\alpha).
\end{align}
In the dynamic version of the mean field control problem with robustness, we seek for a feedback functions $\theta(t,x,\mu)$ that generate (via \eqref{dynamic}) a state process at the minimum cost in \eqref{ob-fun-1}:
\begin{align}\label{robustness-vf}
 \mathcal V(t,\mu)=\inf_{\theta\in\mathcal U_t} J(t,\mu,\theta).
\end{align}
Given its mean field feature, we adopt HJBI equations with measure variables to solve \eqref{robustness-vf}. We will impose the following assumption on model paramters:
\begin{assumption}\label{assumption-saddle-points}
For $\mu_{X,Y}\in\tilde{\mathcal P}\subset\mathcal P_2(\mathbb R^d\times\mathbb R^d)$ and its lifting $(X,Y)\in L^2(\Omega,\mathcal F,\mathbb P;\mathbb R^d\times\mathbb R^d)$ in an appropriate probability space $(\Omega,\mathcal F,\mathbb P)$, there exists a unique pair of feedback functions
\begin{align*}
 \theta^*:\mathcal P_2(\mathbb R^d\times\mathbb R^d)\times\mathbb R^d\times\mathbb R^d\to\mathbb R^d,\ \eta^*:\mathcal P_2(\mathbb R^d\times\mathbb R^d)\times\mathbb R^d\times\mathbb R^d\to\mathbb R^{k_1},\ \alpha^*:\mathcal P_2(\mathbb R^d\times\mathbb R^d)\to\mathbb R^{k_2},
\end{align*}
such that
\begin{align}\label{saddlepoint-1}
&\quad\inf_{\theta\in L^2(\Omega,\mathcal F,\mathbb P;\mathbb R^k)}\mathbb E\big[a(\theta,\tilde\eta^*,\tilde\alpha^*,X,\mu_{(X,\theta,\tilde\eta^*)})\cdot Y+f(\tilde\alpha^*,\mu_{(X,\theta,\tilde\eta^*)})\big]\notag\\
 &=\sup_{\substack{\eta\in L^2(\Omega,\mathcal F,\mathbb P;\mathbb R^{k_1})\\\alpha\in\mathbb R^{k_2}}}\mathbb E\big[a(\tilde\theta^*,\eta,\alpha,X,\mu_{(X,\tilde\theta^*,\eta)})\cdot Y+f(\alpha,\mu_{(X,\tilde\theta^*,\eta)})\big]\notag\\
 &=\mathbb E\big[a(\tilde\theta^*,\tilde\eta^*,\tilde\alpha^*,X,\mu_{(X,\tilde\theta^*,\tilde\eta^*)})\cdot Y+f(\tilde\alpha^*,\mu_{(X,\tilde\theta^*,\tilde\eta^*)})\big],
 \end{align}
where
\begin{align}\label{saddlepoint}
(\tilde\theta^*,\tilde\eta^*,\tilde\alpha^*)=\big(\theta^*(\tilde\mu,X,Y),\eta^*(\tilde\mu,X,Y),\alpha^*(\tilde\mu)\big),\quad\tilde\mu={\rm Law}(X,Y).
\end{align}
\end{assumption}
It is inferred from \eqref{saddlepoint-1} and \eqref{saddlepoint} that
\begin{align}\label{saddlepoint-2}
 \mathcal H(\mu_{(X,Y)})&:=\inf_{\theta\in L^2(\Omega,\mathcal F,\mathbb P;\mathbb R^k)}\sup_{\substack{\eta\in L^2(\Omega,\mathcal F,\mathbb P;\mathbb R^{k_1})\\\alpha\in\mathbb R^{k_2}}}\mathbb E\big[a(\theta,\eta,\alpha,X,\mu_{(X,\theta,\eta)})\cdot Y+f(\alpha,\mu_{(X,\theta,\eta)})\big]\notag\\
 &=\sup_{\substack{\eta\in L^2(\Omega,\mathcal F,\mathbb P;\mathbb R^{k_1})\\\alpha\in\mathbb R^{k_2}}}\inf_{\theta\in L^2(\Omega,\mathcal F,\mathbb P;\mathbb R^k)}\mathbb E\big[a(\theta,\eta,\alpha,X,\mu_{(X,\theta,\eta)})\cdot Y+f(\alpha,\mu_{(X,\theta,\eta)})\big],
 \end{align}
which can be understood as Isaacs condition in the context of robust mean field control. Before delving into the HJBI equation, let us first introduce some preliminaries on Wasserstein derivatives. One may also find similar contents in \cite{Carmona2018-I}. 

Let $U:\mathcal P_2(\mathbb R^d)\to\mathbb R$ be a $\mathcal W_2$-continuous function. Its Wasserstein gradient is a mapping that takes the form
\begin{align*}
 \partial_\mu U(\mu,x):\ (\mu,x)\in\mathcal P_2(\mathbb R^d)\times\mathbb R^d\ \to\ \mathbb R^d,
\end{align*}
and can be characterized by
\begin{align}\label{W-derivative}
 U(\mathcal L_{\xi+\eta})-U(\mathcal L_\xi)=\mathbb E\big[\partial_\mu U(\mathcal L_\xi,\xi)\cdot\eta\big]+o\big(\mathbb E[|\eta|^2]^\frac12\big),
\end{align}
for any square integrable random variables $\xi,\eta$.

Let $\mathcal C^0\big(\mathcal P_2(\mathbb R^d)\big)$ denote the set of $\mathcal W_2$-continuous functions $U:\mathcal P_2(\mathbb R^d)\to\mathbb R$. By $\mathcal C^1\big(\mathcal P_2(\mathbb R^d)\big)$ we mean the space of functions $U\in\mathcal C^0\big(\mathcal P_2(\mathbb R^d)\big)$ such that $\partial_\mu U$ exists and is continuous on $\mathcal P_2(\mathbb R^d)\times\mathbb R^d$. Similarly, $\mathcal C^2\big(\mathcal P_2(\mathbb R^d)\big)$ is the space of functions $U\in\mathcal C^1\big(\mathcal P_2(\mathbb R^d)\big)$ such that the following maps exist and are all jointly continuous:
\begin{align*}
 &\mathcal P_2(\mathbb R^d)\times\mathbb R^d\ni(\mu,x)\mapsto D_x\partial_\mu U(\mu,x);\\
 &\mathcal P_2(\mathbb R^d)\times\mathbb R^d\times\mathbb R^d\ni(\mu,x,\tilde x)\mapsto\partial^2_{\mu\mu}U(\mu,x,\tilde x).
\end{align*}
Inductively, $\mathcal C^3\big(\mathcal P_2(\mathbb R^d)\big)$ is the space of functions $U\in\mathcal C^2\big(\mathcal P_2(\mathbb R^d)\big)$ such that the third order derivatives, i.e., the following maps exist and are all jointly continuous:
\begin{align*}
 &\mathcal P_2(\mathbb R^d)\times\mathbb R^d\ni(\mu,x)\mapsto D^2_{xx}\partial_\mu U(\mu,x);\\
 &\mathcal P_2(\mathbb R^d)\times\mathbb R^d\times\mathbb R^d\ni(\mu,x,\tilde x)\mapsto\big(D_x\partial^2_{\mu\mu}U(\mu,x,\tilde x),D_{\tilde x}\partial^2_{\mu\mu}U(\mu,x,\tilde x)\big);\\
 &\mathcal P_2(\mathbb R^d)\times\mathbb R^d\times\mathbb R^d\times\mathbb R^d\ni(\mu,x,\tilde x,\hat x)\mapsto\partial^3_{\mu\mu\mu}U(\mu,x,\tilde x,\hat x).
\end{align*}
In the same way as above, we further inductively define $\mathcal C^k\big(\mathcal P_2(\mathbb R^d)\big)$, $k=4,5,6$. For $x=(x_1,x_2)\in\mathbb R^{d_1}\times\mathbb R^{d_2}$, $\mu\in\mathcal P_2(\mathbb R^d)$, $U\in\mathcal C^2\big(\mathcal P_2(\mathbb R^d)\big)$, we use the following notation to denote different components of $\partial_\mu U$:
\begin{align*}
    \partial_\mu U(\mu,x)=\big(\partial_\mu U^{(x_1)}(\mu,x),\partial_\mu U^{(x_2)}(\mu,x)\big)^\top\in\mathbb R^{d_1}\times\mathbb R^{d_2}.
\end{align*}
Inductively, for $i=1,2$, $\partial_\mu(\partial_\mu U^{(x_i)})=\big(\partial^2_{\mu\mu}U^{(x_i)(x_1)},\partial^2_{\mu\mu}U^{(x_i)(x_2)}\big)\in\mathbb R^{d_i\times d_1}\times\mathbb R^{d_i\times d_2}$.

With the above preparation, we may now consider HJBI equation \eqref{HJB} and verification results.
\begin{theorem}\label{verification}
 Suppose that \eqref{HJB} admits a classical solution $V(t,\cdot)\in C^2\big(\mathcal P_2(\mathbb R^d)\big)$ satisfying for any compact subset $\mathcal K\subset\mathcal P_2(\mathbb R^d)$,
 \begin{align*}
  \sup_{\mu\in\mathcal K}\bigg[\int_{\mathbb R^d}|D_x\partial_\mu V(t,\mu,x)|\mu(dx)\bigg]<+\infty.
 \end{align*}
 Suppose further that Assumption \ref{assumption-saddle-points} holds for $(Id,\partial_\mu V(t,\mu,\cdot))\sharp\mu$, $\mu\in\mathcal P_2(\mathbb R^d)$ with $a$ and $f$ continuously differentiable. Moreover, let $(\theta,\eta,\alpha)$ be the feedback function defined according to \eqref{saddlepoint} with 
\begin{align*}
&\tilde\mu_t=\big(Id,\partial_\mu V(t,\mu,\cdot)\big)\sharp\mu,\ (\theta,\eta)(t,x,\mu)=(\theta^*,\eta^*)\big(t,\tilde\mu_t,x,\partial_\mu V(t,\mu,x)\big),\ \alpha(t,\mu)=\alpha^*(t,\tilde\mu_t),
\end{align*}
and suppose that $(\theta,\eta,\alpha)$ is Lipschitz. Then 
 $$V(t,\mu)=\mathcal V(t,\mu),\quad(t,\mu)\in[0,T]\times\mathcal P_2(\mathbb R^d),$$
 and $(\theta,\eta,\alpha)$ is the optimal feedback function satisfying
 \begin{align}\label{feedback func}
 \left\{\begin{aligned}
 &\int_{\mathbb R^d}\partial_{\hat\mu} a^{(\theta)}\big((\theta,\eta)(t,\tilde x,\mu),\alpha(t,\mu),x,\hat\mu_t,(\theta,\eta)(t,x,\mu),x\big)\partial_\mu V(t,\mu,\tilde x)\mu(d\tilde x)\\
 &\qquad+D_\theta a\big((\theta,\eta)(t,x,\mu),\alpha(t,\mu),x\big)\partial_\mu V(t,\mu,x)+\partial_{\hat\mu}f^{(\theta)}\big(\alpha(t,\mu),\hat\mu_t,x,(\theta,\eta)(t,x,\mu)\big)=0,\\
 &\int_{\mathbb R^d}\partial_{\hat\mu} a^{(\eta)}\big((\theta,\eta)(t,\tilde x,\mu),\alpha(t,\mu),x,\hat\mu_t,(\theta,\eta)(t,x,\mu),x\big)\partial_\mu V(t,\mu,\tilde x)\mu(d\tilde x)\\
&\qquad+D_\eta a\big((\theta,\eta)(t,x,\mu),\alpha(t,\mu),x\big)\partial_\mu V(t,\mu,x)+\partial_{\hat\mu}f^{(\eta)}\big(\alpha(t,\mu),\hat\mu_t,x,(\theta,\eta)(t,x,\mu)\big)=0,\\
&\int_{\mathbb R^d}D_\alpha a\big((\theta,\eta)(t,x,\mu),\alpha(t,\mu),x,\hat\mu\big)\partial_\mu V(t,\mu,x)\mu(dx)+D_\alpha f\big(\alpha(t,\mu),\hat\mu_t\big)=0,\\
&(t,x,\mu)\in[0,T]\times\mathbb R^d\times\mathcal P_2(\mathbb R^d).
 \end{aligned}\right.
\end{align}
where
\begin{align*}
 \hat\mu_t:=\big(\cdot,\theta(t,\cdot,\mu),\eta(t,\cdot,\mu)\big)\sharp\mu,\ \partial_{\hat\mu} a=:\big(\partial_{\hat\mu} a^{(x)},\partial_{\hat\mu} a^{(\theta)},\partial_{\hat\mu} a^{(\eta)}\big)^\top\in\mathbb R^d\times\mathbb R^k\times\mathbb R^{k_1}.
\end{align*}
\end{theorem}
\begin{remark} As will be shown later, \eqref{feedback func} could serve as the characterization of the optimal robust strategy for some particular cases.
\end{remark}

\subsection{Well-posedness of the HJBI equations}\label{solvable-cases}
In view of the verification results in Theorem \ref{verification}, solving \eqref{robustness-vf} boils down to the well-posedness of \eqref{HJB}. Towards that end, we further analyse the solvability under different conditions, namely the case where parameters have bounded derivatives, as well as the linear quadratic case. Both mentioned cases admit the feature that the Hamiltonian $\mathcal H$ satisfies Assumption \ref{assumption-displacement-convex}, which is on regularity and certain kind of displacement convex/concave property. The current Hamiltonian in \eqref{HJB} has different features from the ones in \cite{Liao24}, where the focus is on mean field control and related potential mean field games without robustness. For the mean field control problems in \cite{Liao24}, the Hamiltonian is naturally displacement concave in momentum in the sense of \eqref{displacement-concave-in-p}, which is not true for \eqref{saddlepoint-1} in general. For the {\color{blue}current case} where the Hamiltonian is not necessarily displacement concave in momentum, we propose two conditions so that the well-posedness of \eqref{HJB} is established.
\begin{assumption}\label{assumption-displacement-convex}
\begin{enumerate}
 \item $\mathcal H\in\mathcal C^6\big(\mathcal P_2(\mathbb R^{d_1+d_2}\times\mathbb R^{d_1+d_2})\big)$, all derivatives of $\mathcal H$ are bounded;
 \item There exists a constant $C$ such that for $\mu\in\mathcal P_2(\mathbb R^d)$ and arbitrary measurable test mappings with compact support
\begin{align*}
 \phi_i,\ \psi_i:\ \mathbb R^d\times\mathbb R^d\mapsto\mathbb R^{d_i},\quad\tilde\phi_i:\ \mathbb R^d\mapsto\mathbb R^{d_i},\quad i=1,2,\ d_1+d_2=d,
\end{align*}
the following holds:
{\small\begin{align}\label{general-displacement}
  &\quad C\int_{\mathbb R^d}\big(|\phi_1(\xi)|^2+|\psi_1(\xi)|^2\big)\tilde\mu(d\xi)\geq\notag\\
  &\sum_{i=1}^2\bigg(\int_{\mathbb R^d}\int_{\mathbb R^d}\phi_i(\xi_1)^\top\partial^2_{\tilde\mu\tilde\mu}\mathcal H^{(x_i)(x_i)}(\tilde\mu,\xi_1,\xi_2)\phi_i(\xi_2)\tilde\mu(d\xi_1)\tilde\mu(d\xi_2)+\int_{\mathbb R^d}\phi_i(\xi)^\top D_{x_i}\partial_{\tilde\mu}\mathcal H^{(x_i)}(\tilde\mu,\xi)\phi_i(\xi)\tilde\mu(d\xi)\bigg)\notag\\
  &-\sum_{i=1}^2\bigg(\int_{\mathbb R^d}\int_{\mathbb R^d}\psi_i(\xi_1)^\top\partial^2_{\tilde\mu\tilde\mu}\mathcal H^{(p_i)(p_i)}(\tilde\mu,\xi_1,\xi_2)\psi_i(\xi_2)\tilde\mu(d\xi_1)\tilde\mu(d\xi_2)+\int_{\mathbb R^d}\psi_2(\xi)^\top D_{p_i}\partial_{\tilde\mu}\mathcal H^{(p_i)}(\tilde\mu,\xi)\psi_2(\xi)\tilde\mu(d\xi)\bigg)\notag\\
  &+\int_{\mathbb R^d}\int_{\mathbb R^d}\phi_1(\xi_1)^\top\partial^2_{\tilde\mu\tilde\mu}\mathcal H^{(x_1)(x_2)}(\tilde\mu,\xi_1,\xi_2)\phi_2(\xi_2)\tilde\mu(d\xi_1)\tilde\mu(d\xi_2)+2\int_{\mathbb R^d}\phi_1(\xi)^\top D_{x_2}\partial_{\tilde\mu}\mathcal H^{(x_1)}(\tilde\mu,\xi)\phi_2(\xi)\tilde\mu(d\xi)\notag\\
  &+\int_{\mathbb R^d}\int_{\mathbb R^d}\phi_1(\xi_1)^\top\partial^2_{\tilde\mu\tilde\mu}\mathcal H^{(x_1)(p_2)}(\tilde\mu,\xi_1,\xi_2)\psi_2(\xi_2)\tilde\mu(d\xi_1)\tilde\mu(d\xi_2)+2\int_{\mathbb R^d}\phi_1(\xi)^\top D_{p_2}\partial_{\tilde\mu}\mathcal H^{(x_1)}(\tilde\mu,\xi)\psi_2(\xi)\tilde\mu(d\xi)\notag\\
  &+\int_{\mathbb R^d}\int_{\mathbb R^d}\phi_2(\xi_1)^\top\partial^2_{\tilde\mu\tilde\mu}\mathcal H^{(x_2)(p_1)}(\tilde\mu,\xi_1,\xi_2)\psi_1(\xi_2)\tilde\mu(d\xi_1)\tilde\mu(d\xi_2)+2\int_{\mathbb R^d}\phi_2(\xi)^\top D_{p_1}\partial_{\tilde\mu}\mathcal H^{(x_2)}(\tilde\mu,\xi)\psi_1(\xi)\tilde\mu(d\xi)\notag\\
  &+\int_{\mathbb R^d}\int_{\mathbb R^d}\psi_1(\xi_1)^\top\partial^2_{\tilde\mu\tilde\mu}\mathcal H^{(p_1)(p_2)}(\tilde\mu,\xi_1,\xi_2)\psi_2(\xi_2)\tilde\mu(d\xi_1)\tilde\mu(d\xi_2)+2\int_{\mathbb R^d}\psi_1(\xi)^\top D_{p_2}\partial_{\tilde\mu}\mathcal H^{(p_1)}(\tilde\mu,\xi)\psi_2(\xi)\tilde\mu(d\xi)\notag\\
  &\quad\geq-C\int_{\mathbb R^d}\big(|\phi_2(\xi)|^2+|\psi_2(\xi)|^2\big)\tilde\mu(d\xi),
\end{align}}
and
{\small\begin{align}\label{general-displacement-1}
  &\quad C\int_{\mathbb R^d}|\tilde\phi_1(x)|^2\mu(dx)\geq\notag\\
  &\sum_{i=1}^2\bigg(\int_{\mathbb R^d}\int_{\mathbb R^d}\tilde\phi_i(x)^\top\partial^2_{\mu\mu}U^{(x_i)(x_i)}(\mu,x,\hat x)\tilde\phi_i(\hat x)\tilde\mu(dx)\mu(d\hat x)+\int_{\mathbb R^d}\tilde\phi_i(x)^\top D_{x_i}\partial_\mu U^{(x_i)}(\mu,x)\tilde\phi_i(x)\mu(dx)\bigg)\notag\\
  &+\int_{\mathbb R^d}\int_{\mathbb R^d}\tilde\phi_1(x)^\top\partial^2_{\tilde\mu\tilde\mu}U^{(x_1)(x_2)}(\mu,x,\hat x)\tilde\phi_2(x)\mu(dx)\mu(d\hat x)+2\int_{\mathbb R^d}\tilde\phi_1(x)^\top D_{x_2}\partial_\mu U^{(x_1)}(\mu,x)\tilde\phi_2(x)\mu(dx)\notag\\
  &\quad\geq-C\int_{\mathbb R^d}|\tilde\phi_2(x)|^2\mu(dx),
\end{align}
where for $(x_1,x_2,p_1,p_2)\in\mathbb R^{d_1}\times\mathbb R^{d_2}\times\mathbb R^{d_1}\times\mathbb R^{d_2},\ \tilde\mu\in\mathcal P_2(\mathbb R^d\times\mathbb R^d)$,
\begin{align*}
    \partial_{\tilde\mu}\mathcal H(\tilde\mu,x_1,x_2,p_1,p_2)=(\partial_{\tilde\mu}\mathcal H^{(x_1)},\partial_{\tilde\mu}\mathcal H^{(x_2)},\partial_{\tilde\mu}\mathcal H^{(p_1)},\partial_{\tilde\mu}\mathcal H^{(p_2)})(\tilde\mu,x_1,x_2,p_1,p_2)\in\mathbb R^{d_1}\times\mathbb R^{d_2}\times\mathbb R^{d_1}\times\mathbb R^{d_2}.
\end{align*}}
\end{enumerate}
\end{assumption}
We note here that, in the context of optimal liquidation with uncertainty, the first item in Assumption \ref{assumption-displacement-convex} will be relaxed via the approximation argument in Lemma \ref{dis-con-app} below. The second item in Assumption \ref{assumption-displacement-convex} is a property that describes a generalized version of displacement convexity. Now we are ready to state our first main result.
\begin{theorem}\label{existence-of-decoupling-field}
Suppose Assumption \ref{assumption-displacement-convex} and $\sigma\geq0$. Then HJBI equation \eqref{HJB} admits a unique global in time classical solution satisfying $V(t,\cdot)\in\mathcal C^4\big(\mathcal P_2(\mathbb R^d)\big)$, $\partial_tV(t,\cdot)\in\mathcal C^2\big(\mathcal P_2(\mathbb R^d)\big)$ and $|\partial^{i_1}_{x_1}\cdots\partial^{i_j}_{x_j}\partial^j_\mu V|_\infty$ ($0\leq j\leq 4,\ 0\leq i_1,\ldots,i_j\leq4,\ 0\leq i_1+\cdots+i_j+j\leq 4$) depends only on $\mathcal H$ and $U$.
\end{theorem}
Our second main result focuses on another particular solvable case of \eqref{HJB} where the Hamiltonian is quadratic in its variables. More specifically, we consider the following Hamiltonian and terminal condition
\begin{align}\label{LQ-ham}
 \left\{\begin{aligned}\mathcal H(\tilde\mu)=\int_{\mathbb R^d}\int_{\mathbb R^d}\bigg[(x,p)^\top Q^{(1)}\left(\begin{matrix}x\\p\end{matrix}\right)+(x,p)^\top\mathbf c\bigg]\tilde\mu(dxdp)+(\bar x,\bar p)^\top Q^{(2)}\left(\begin{matrix}\bar x\\\bar p\end{matrix}\right),\\
 \tilde\mu\in\mathcal P_2(\mathbb R^d\times\mathbb R^d),\\
 U(\mu)=\int_{\mathbb R^d}\bigg[x^\top Q^{(1)}_Tx+x^\top\mathbf c_T\bigg]\tilde\mu(dxdp)+\bar x^\top Q^{(2)}_T\bar x,\end{aligned}\right.
\end{align}
where $Q^{(i)}\in\mathbb R^{2d\times 2d}$, $i=1,2$, $\mathbf c\in\mathbb R^{2d}$, $\bar x=\int_{\mathbb R^d}\int_{\mathbb R^d}x\tilde\mu(dxdp)$, $\bar p=\int_{\mathbb R^d}\int_{\mathbb R^d}p\tilde\mu(dxdp)$. Such type of HJBI equations can be linked to the mean field control problems with robustness where the dynamics and cost are in linear and quadratic form respectively. It turns out that the HJBI equation can be reduced to a ordinary Riccati system. We first interpret the second item of Assumption \ref{assumption-displacement-convex} into the conditions on matrix data.
\begin{lemma}\label{LQ-condition}
 Suppose Assumption \ref{assumption-displacement-convex}. Then
\begin{align}\label{generalized-displacement-convexity-LQ}
  \left\{\begin{aligned}-C\big(|x_2|^2+|p_1|^2\big)&\leq\mathbf x^\top Q^{(1)}_{11}\mathbf x+2(x_1,0)^\top Q^{(1)}_{12}\left(\begin{matrix}0\\p_2\end{matrix}\right)\\
  &\quad+2(0,x_2)^\top Q^{(1)}_{12}\left(\begin{matrix}p_1\\0\end{matrix}\right)+\mathbf p^\top Q^{(1)}_{22}\mathbf p\leq C\big(|x_1|^2+|p_2|^2\big),\\
  -C\big(|x_2|^2+|p_1|^2\big)&\leq\mathbf x^\top\big(Q^{(1)}_{11}+Q^{(2)}_{11}\big)\mathbf x+2(x_1,0)^\top\big(Q^{(1)}_{12}+Q^{(2)}_{12}\big)\left(\begin{matrix}0\\p_2\end{matrix}\right)\\
  \quad+2(0,x_2)^\top&\big(Q^{(1)}_{12}+Q^{(2)}_{12}\big)\left(\begin{matrix}p_1\\0\end{matrix}\right)+\mathbf p^\top\big(Q^{(1)}_{22}+Q^{(2)}_{22}\big)\mathbf p\leq C\big(|x_1|^2+|p_2|^2\big),\\
  -C|x_2|^2&\leq\mathbf x^\top Q^{(1)}_T\mathbf x\leq C|x_1|^2,\quad-C|x_2|^2\leq\mathbf x^\top\big(Q^{(1)}_T+Q^{(2)}_T\big)\mathbf x\leq C|x_1|^2,\\
  &\mathbf x=(x_1,x_2)^\top,\ \mathbf p=(p_1,p_2)^\top\quad x_i,\ p_i\in\mathbb R^{d_i},\ i=1,2.\end{aligned}\right.
\end{align}
\end{lemma}
We can prove the following result on well-posedness:
\begin{theorem}\label{LQ-global-wp}
 Suppose that \eqref{generalized-displacement-convexity-LQ} holds. Then HJBI equation \eqref{HJB} with $(\mathcal H,U)$ from \eqref{LQ-ham} admits a solution:
 \begin{align}\label{LQ-solution}
 V(t,\mu)=\int_{\mathbb R^d}\big[x^\top a^{(1)}(t)x+a^{(2)}(t)^\top x\big]\mu(dx)+\int_{\mathbb R^d}\int_{\mathbb R^d}x^\top a^{(3)}(t)\hat x\mu(d\hat x)\mu(dx)+a^{(4)}(t),\notag\\
 (t,\mu)\in[0,T]\times\mathcal P_2(\mathbb R^d),
\end{align}
where for $t\in[0,T]$, $a^{(1)}(t),a^{(3)}(t)\in\mathbb R^{d\times d}$, $a^{(2)}(t)\in\mathbb R^d$, $a^{(4)}(t)\in\mathbb R$ are deterministic and satisfy
\begin{align}\label{LQ-Riccati}
 \left\{\begin{aligned}
 \frac d{dt}a^{(1)}(t)+Q^{(1)}_{11}+4a^{(1)}(t)^\top Q^{(1)}_{22}a^{(1)}(t)+2a^{(1)}(t)^\top Q^{(1)}_{21}+2Q^{(1)}_{12}a^{(1)}(t)=0,\\
 \frac d{dt}a^{(2)}(t)+2Q^{(1)}_{12}a^{(2)}(t)+4a^{(1)}(t)Q^{(1)}_{22}a^{(2)}(t)+2a^{(3)}(t)Q^{(1)}_{22}a^{(2)}(t)+2Q^{(2)}_{12}a^{(2)}(t)\\
 +4\big(a^{(1)}(t)+a^{(3)}(t)\big)Q^{(2)}_{22}a^{(2)}(t)+c_1+2\big(a^{(1)}(t)+a^{(3)}(t)\big)c_2=0,\\
 \frac d{dt}a^{(3)}(t)+2Q^{(1)}_{12}a^{(3)}(t)+2a^{(3)}(t)^\top Q^{(1)}_{21}+4a^{(1)}(t)^\top Q^{(1)}_{22}a^{(3)}(t)+4a^{(3)}(t)^\top Q^{(1)}_{22}a^{(3)}(t)\\
 +4a^{(3)}(t)^\top Q^{(1)}_{22}a^{(1)}(t)+Q^{(2)}_{11}+2Q^{(2)}_{12}\big(a^{(1)}(t)+a^{(3)}(t)\big)+2\big(a^{(1)}(t)+a^{(3)}(t)\big)^\top Q^{(2)}_{21}\\
 +4\big(a^{(1)}(t)+a^{(3)}(t)\big)^\top Q^{(2)}_{22}\big(a^{(1)}(t)+a^{(3)}(t)\big)=0,\\
 \frac d{dt}a^{(4)}(t)+a^{(2)}(t)^\top Q^{(1)}_{22}a^{(2)}(t)+a^{(2)}(t)^\top Q^{(2)}_{22}a^{(2)}(t)+a^{(2)}(t)^\top c_2+{\rm tr}\big[\sigma^\top\sigma a^{(1)}(t)\big]=0,\\
 a^{(1)}(T)=Q^{(1)}_T,\ a^{(2)}(T)=\mathbf c_T,\ a^{(3)}(T)=Q^{(2)}_T,\ a^{(4)}(T)=0.
 \end{aligned}\right.
\end{align}
\end{theorem}
Among the linear quadratic case, we will further investigate those with penalized terminal data:
\begin{align}\label{LQ-ham-penalized}
 \left\{\begin{aligned}&\quad\mathcal H^\varepsilon(\tilde\mu)=\int_{\mathbb R^d\times\mathbb R^d}\bigg[(x,p)^\top Q^{(1)}\left(\begin{matrix}x\\p\end{matrix}\right)-\varepsilon p^\top_1p_1\bigg]\tilde\mu(dxdp)+(\bar x,\bar p)^\top Q^{(2)}\left(\begin{matrix}\bar x\\\bar p\end{matrix}\right)\\
 &=\int_{\mathbb R^d\times\mathbb R^d}\bigg[(x,p)^\top\tilde Q^{(1)}\left(\begin{matrix}x\\p\end{matrix}\right)\bigg]\tilde\mu(dxdp)+(\bar x,\bar p)^\top Q^{(2)}\left(\begin{matrix}\bar x\\\bar p\end{matrix}\right),\quad 
 \tilde\mu\in\mathcal P_2(\mathbb R^d\times\mathbb R^d),\\
 &U^\lambda(\mu)=\int_{\mathbb R^d}(\lambda x^\top_1x_1+q^T_{12}x^\top_1x_2+x^\top_2q^T_{22}x_2)\mu(dx),\quad\mu\in\mathcal P_2(\mathbb R^d),\end{aligned}\right.
\end{align}
where $\varepsilon>0$, the integrator $x=(x_1,x_2)\in\mathbb R^{d_1}\times\mathbb R^{d_2}$ and the penalty relating to the constraint are placed on the first component. Let $V^\lambda$ solve \eqref{HJB} with data $(\mathcal H^\varepsilon,U^\lambda)$. Our interests is in the limit of $V^\lambda$ as $\lambda\to+\infty$. Intuitively speaking, as the penalty parameter goes in infinity, the limit $V^{+\infty}$ of $(V^\lambda)_{\lambda>0}$ describes the value function for robust mean field control problem \eqref{ob-fun-1}$\sim$\eqref{robustness-vf} with an additional constraint $X^{(1)}_T=0$. Formally, $V^{+\infty}$ satisfies the following HJBI equation with singular terminal value:
\begin{align}\label{constraint-HJB}
    \left\{\begin{aligned}
        \partial_tV+\frac12\int_{\mathbb R^d}{\rm tr}\big[\sigma^\top\sigma D_x\partial_\mu V(t,\mu,x)\big]\mu(dx)+\mathcal H^\varepsilon(\tilde\mu_t)=0,\ t\in[0,T),\ \mu\in\mathcal P_2(\mathbb R^d),\\
 V(T,\mu)=+\infty.
    \end{aligned}\right.
\end{align}
For the case where $\mathcal H^\varepsilon$ is linear, i.e., $\mathcal H^\varepsilon(\tilde\mu)=\int_{\mathbb R^d\times\mathbb R^d}H(x_1,x_2,p)\tilde\mu(dxdp)$, we can plug $V(t,\mu)=\int_{\mathbb R^d}u(t,x)\mu(dx)$ into \eqref{constraint-HJB} and reduce it to a finite dimensional equation on $u$ with domain $[0,T]\times\mathbb R^d$ and singular terminal condition. Such finite dimensional counterpart of \eqref{constraint-HJB} has been studied in the optimal liquidation literature. Recall the discussion in Introduction for an incomplete reference. Our previous results have ensured the well-posedness of \eqref{HJB} with $(\mathcal H^\varepsilon,U^\lambda)$ from \eqref{LQ-ham-penalized}. Furthermore, the convergence of $(V^\lambda)_{\lambda>0}$ to $V^{+\infty}$ is made rigorous in the following:

\begin{theorem}\label{limit-V-lambda}
    Suppose \eqref{generalized-displacement-convexity-LQ} and let $V^\lambda$ solve \eqref{HJB} with data $(\mathcal H^\varepsilon,U^\lambda)$ from \eqref{LQ-ham-penalized}. Then $V^\lambda$ admits the representation
\begin{align}\label{lambda-value}
V^\lambda(t,\mu)=\int_{\mathbb R^d}x^\top a^\lambda_1(t)x\mu(dx)+\int_{\mathbb R^d}\int_{\mathbb R^d}x^\top a^\lambda_3(t)\hat x\mu(d\hat x)\mu(dx),\quad
 (t,\mu)\in[0,T]\times\mathcal P_2(\mathbb R^d),
   \end{align} 
    where $\big(a^\lambda_i(t)\big)_{1\leq i\leq 4,t\in[0,T)}$ satisfies \eqref{LQ-Riccati} with $a^\lambda_2(t)=a^\lambda_4(t)=0$, and there exists a constant $C$ independent of $\lambda$ such that for $\varphi=a^\lambda_1,\ a^\lambda_1+a^\lambda_3$,
\begin{align}\label{uniform-estimates}
 \left\{\begin{aligned}
 &\varphi(t)\geq\left(\begin{matrix}\frac1{C(4\varepsilon+C)(T-t)}I_{d_1}&0\\
 0&e^{q_{22}(T-t)}q^T_{22}e^{q_{22}(T-t)}-C(T-t)I_{d_2}
 \end{matrix}\right),\\
 &\varphi(t)\leq\left(\begin{matrix}\bigg[\frac{\sqrt C}{2\sqrt{\varepsilon}}\big(e^{4\sqrt{\varepsilon}(T-t)}-1\big)^{-1}+\frac14\sqrt{\frac C\varepsilon}\bigg]I_{d_1}&0\\
 0&e^{q_{22}(T-t)}q^T_{22}e^{q_{22}(T-t)}
 \end{matrix}\right),\ t\in[0,T].
 \end{aligned}\right.
    \end{align}
As a result, $\big(a^\lambda_i(t)\big)_{t\in[0,T-\kappa]}$ converges uniformly to $\big(a^{+\infty}_i(t)\big)_{t\in[0,T-\kappa]}$ as $\lambda\to+\infty$ for $i=1,3,$ and arbitrary $\kappa>0$, while $(V^\lambda(t,\mu)\big)_{t\in[0,T-\kappa],\mu\in\mathcal P_2(\mathbb R^d)}$ converges to the following
\begin{align}\label{constraint-value}
 V^{+\infty}(t,\mu)=\int_{\mathbb R^d}x^\top a^{+\infty}_1(t)x\mu(dx)+\int_{\mathbb R^d}\int_{\mathbb R^d}x^\top a^{+\infty}_3(t)\hat x\mu(d\hat x)\mu(dx),\notag\\
 (t,\mu)\in[0,T-\kappa]\times\mathcal P_2(\mathbb R^d).
\end{align}
\end{theorem}
\subsection{Solving the optimal liquidation problem with uncertainty}\label{application}
Having established the well-posedness of HJBI as well as the verification results in the previous section, we consider the optimal liquidation problems with composite uncertainty.

Let $T\in(0,+\infty)$, $(\Omega,\mathcal F,\mathbb P,\mathbb F)$ be a complete filtered probability space with $\mathbb F=(\mathcal F_t)_{t\in[0,T]}$ generated by a Brownian motion $(W_t)_{t\in[0,T]}$. Consider a large investor holding $M_t$ shares in a risky stock:
\begin{align}\label{model-1}
 M_t=m-\int_0^t\xi_sds,\quad t\in[0,T].
\end{align}
In \eqref{model-1}, $(\xi_t)_{t\in[0,T]}$ is an $\mathbb F$-adapted process in $\mathbb R^{d_1}$ which describes the transaction rate of the investor. The aim of the large investor is to liquidate her assets by a fixed time $T$. As is explained in Section \ref{intro}, the trading of the large investor has an impact on the price process of the stock. More specifically, the large investor\rq{}s trading rate and its distribution have an impact on the price process in $\mathbb R^{d_2}$ via $Q_t$ in the following:
\begin{align}\label{lqd-dyn-0}
 dQ_t=\big[-\rho Q_t+\kappa_1\alpha_t-\phi\xi_t-\tilde\phi\mathbb E\xi_t\big]dt+dW^\eta_t,\quad t\in[0,T],\ Q_0=q,
 \end{align}
where $\rho>0$ and naturally $\phi,\ \tilde\phi\geq0$ since more selling in general decreases the asset price. The position and trading of the large investor could incur child orders. Since working with each child order process is inconvenient, it seems reasonable to work with its distribution, which boils down to the $\mathbb E\xi_t$ term in \eqref{lqd-dyn-0}. 

The investor suffers from two kinds of uncertainty. The first kind of uncertainty comes from the unknown of probability space, which is expressed via the drifted Brownian motion.
\begin{align}\label{lqd-dyn-3}
 W^\eta_t=W_t+\kappa_2\int_0^t\eta_sds,\quad\kappa_2\geq0.
\end{align}
Here, $W$ is a Brownian motion under the reference probability measure and $(\eta_t)_{t\in[0,T]}$ can be understood as the deviation from the reference probability due to ambiguity. This deviation can be expressed through the Radon-Nikodym derivative $\mathcal E\big(\kappa_2\int_0^\cdot\eta^\top_sdW_s\big)_t,\ t\in[0,T].$
The second kind of uncertainty is due to the unknown model parameters that is expressed via the deterministic process $(\alpha_t)_{t\in[0,T]}$ in \eqref{lqd-dyn-0}. The above two types of uncertainty are different in nature. Mathematically speaking, $(\eta_t)_{t\in[0,T]}$ is an $\mathbb F$-adapted stochastic process that arises from the change of probability, while $(\alpha_t)_{t\in[0,T]}$ is a deterministic that describes the uncertainty of model parameters that are supposed to be deterministic. In view of the intuition of $(\eta_t)_{t\in[0,T]}$, additional constraints are required so that $\mathcal E\big(\kappa_2\int_0^\cdot\eta^\top_sdW_s\big)_t$ is well defined. Towards that end, we further assume in this section that the admissible set $\tilde{\mathcal U_t}$ consists of $(\xi_s)_{s\in[t,T]}$ that satisfies Definition \ref{admissible-set} and the additional requirement that for any $(\eta_s,\alpha_s)_{s\in[t,T]}\in\mathcal V^\xi_t$,
\begin{align*}
 \bigg(\kappa_2\int_t^s\eta^\top_rdW_r\bigg)_{s\in[t,T]}
\end{align*}
is a BMO martingale, i.e.,
\begin{align}\label{BMO}
 \kappa_2\sup_{t\leq\tau\leq T}\mathbb E\bigg[\int_\tau^T|\eta_t|^2dt\bigg|\mathcal F_\tau\bigg]<+\infty,
\end{align}
where $\tau$ ranges over all stopping times of filtration $\{\mathcal F_s\}_{s\in[t,T]}$.
\subsubsection{The optimal liquidation problem with general execution cost}
In the general case, the transaction of the investor incurs a nonlinear transient impact on the fundamental price. Combining the transient and permanent impact from the transaction, the transaction price of the asset $\tilde S_t$ is postulated to deviate from fundamental price process $S_t$ in the following way:
\begin{align*}
 \tilde S_t=S_t-g(\xi_t,Q_t,\mu_{\xi_t}).
\end{align*}
In \cite{Horst19}, the function $g$ is taken to be a linear separated form $g_0(\xi,Q,\mu)=\Lambda\xi+Q$. One may understand $g_0$ as the combination of transaction cost per share $\Lambda\xi$ and the difference in liquidation proceeds per share $Q$. The mean field generalization of $g_0$ will be further analyzed in the next section. Here we are interested in more general $g$ that could be non-separable and depend on $\mu_{\xi_t}$. The investor chooses to minimize the following utility:
\begin{align}\label{general-cost}
 \sup_{(\eta,\alpha)\in\mathcal V^\xi_0}\mathbb E\bigg[U_T(M_T,Q_T,\mu_{M_T})+\int_0^T\big(U(\xi_t,Q_t,\mu_{\xi_t})+f(\mu_{M_t})-\lambda|\eta_t|^2-\Phi(\alpha_t)\big)dt\bigg].
\end{align}
Here in the objective functional, $U(\xi_t,Q_t,\mu_{\xi_t})=\xi_t\cdot g(\xi_t,Q_t,\mu_{\xi_t})$ stands for the complex of transaction cost and liquidation proceeds. The $-\lambda|\eta_t|^2$ term corresponds to the relative entropy which penalizes the deviation from the reference probability measure. The $\Phi(\alpha_t)$ in the running cost describes the ambiguity set of the parameters. The term $f(\mu_{M_t})$ can be used to model the operating risk. A typical choice of $f(\mu_{M_t})$ is the variance in the capital stock ${\rm Var}[M_t]$ (see also \cite{Bensoussan2019}). The liquidity at time $T$ is done via a bulk transaction with the the complex of transaction cost and liquidation proceeds $U_T(M_T,S_T,\mu_{M_T})$.

We make technical assumptions on growth and convexity in the following:
\begin{align}\label{ham-convexity-1}\left\{\begin{aligned}
&\text{$f,\ \Phi,\ U$ have continuous derivatives till sixth order,}\\
&\big|\partial^{l_1}_{m_1}\cdots\partial^{l_k}_{m_k}\partial^k_\mu f(\mu_M,m_1,\ldots,m_k)\big|\leq C,\ 0\leq k\leq 6,\ 0\leq l_1,\ldots,l_k\leq6,\ \\
&\qquad2\leq l_1+\cdots+l_k+k\leq 6,\ m_1,\ldots,m_k\in\mathbb R^{d_1},\quad D^2_\alpha\Phi\geq\delta I_{d_2},\\
 &f(\mu_{\gamma M_1+(1-\gamma)M_2})\leq\gamma f(\mu_{M_1})+(1-\gamma)f(\mu_{M_2}),\ \gamma\in[0,1],\ M_1,M_2\in L^2(\Omega,\mathcal F,\mathbb P;\mathbb R^{d_1}),\\
 &\big|\partial^{l_0}_\xi\partial^{\tilde l_0}_Q\partial^{l_1}_{\xi_1}\cdots\partial^{l_k}_{\xi_k}\partial^k_\mu U(\xi,Q,\mu,\xi_1,\ldots,\xi_k)\big|\leq C,\ D^2_{QQ}U(\xi,Q,\mu)\leq0,\ (\xi,Q,\mu)\in\mathbb R^{d_1}\times\mathbb R^{d_2}\times\mathcal P_2(\mathbb R^d),\\
 &\qquad0\leq k\leq 6,\ 0\leq l_1,\ldots,l_k\leq6,\ 2\leq l_0+\tilde l_0+l_1+\cdots+l_k+k\leq 6,\ \xi_1,\ldots,\xi_k\in\mathbb R^{d_1},\\
 &\mathbb E[\langle D_\xi U(\xi,Q,\mu_\xi)-D_\xi U(\eta,Q,\mu_\eta),\xi-\eta\rangle]+\mathbb E\mathbb{\breve E}[\langle\partial_\mu U(\xi,Q,\mu_\xi,\breve\xi)-\partial_\mu U(\eta,Q,\mu_\eta,\breve\eta),\xi-\eta\rangle]\\
 &\qquad\geq\delta\mathbb E|\xi-\eta|^2,\quad\xi,\eta\in L^2(\Omega,\mathcal F,\mathbb P;\mathbb R^{d_1}).
 \end{aligned}\right.
\end{align}
In order to solve the above optimal liquidation problem with robustness, we take the joint distribution of $(M_t,Q_t)$ as the state variable. In view of \eqref{model-1}$\sim$\eqref{general-cost}, the dynamic programming principle yields the associated HJB equation as follows
\begin{align}\label{lqd-general-HJB}
\left\{\begin{aligned}& \partial_tV+\int_{\mathbb R^d}\frac12{\rm tr}\big[D_q\partial_\mu V^{(q)}(t,\mu,m,q)\mu(dmdq)\big]+\mathcal H\big((Id,\partial_\mu V)\sharp\mu\big)=0,\\ &\qquad(t,\mu)\in[0,T)\times\mathcal P_2(\mathbb R^{d_1}\times\mathbb R^{d_2}),\\
&V(T,\mu)=\mathcal U(\mu):=\int_{\mathbb R\times\mathbb R}U_T(m,q,\pi_1\sharp\mu)\mu(dmdq),\quad\ \mu\in\mathcal P_2(\mathbb R^{d_1}\times\mathbb R^{d_2}),\end{aligned}\right.
 \end{align}
where $\partial_\mu V=(\partial_\mu V^{(m)},\partial_\mu V^{(q)})^\top\in\mathbb R^{d_1}\times\mathbb R^{d_2}$, $\pi_1:(m,q)\mapsto m$, for random variable $(M,Q,\tilde p_1,\tilde p_2)$ taking values in $\mathbb R^{d_1}\times\mathbb R^{d_2}\times\mathbb R^{d_1}\times\mathbb R^{d_2}$,
\begin{align}\label{lqd-general-ham}
&\quad\mathcal H\big(\text{Law}(M,Q,\tilde p_1,\tilde p_2)\big)\notag\\
&=\inf_{\xi\in L^2(\Omega;\mathbb R^+)}\sup_{\substack{\eta\in L^2(\Omega;\mathbb R)\\\alpha\in\mathbb R}}\bigg\{\mathbb E\big[-\tilde p_1\cdot\xi+\tilde p_2\cdot\big(\gamma_1\xi+\gamma_2\mathbb E\xi+\kappa_2\eta+\kappa_1\alpha\big)+U(\xi,Q,\mu_\xi)\notag\\
&\quad\qquad\qquad\qquad\qquad\qquad-\lambda|\eta|^2\big]-\Phi(\alpha)+f\big(\mu_M\big)\bigg\}\notag\\
&=\inf_{\xi\in L^2(\Omega;\mathbb R^+)}\mathbb E\big[(-\tilde p_1+\gamma^\top_1\tilde p_2+\gamma^\top_2\mathbb E\tilde p_2)\cdot\xi+U(\xi,Q,\mu_\xi)\big]+f(\mu_M)+\frac{\kappa_2^2\mathbb E|\tilde p_2|^2}{4\lambda}+\Phi^*(\mathbb E\tilde p_2),
\end{align}
where $\Phi^*(\alpha):=\inf_{\tilde\alpha\in\mathbb R^{d_2}}\{\tilde\alpha\cdot\alpha-\Phi(\tilde\alpha)\}$,\ $\alpha\in\mathbb R^{d_2}$. The Hamiltonian in \eqref{lqd-general-ham} does not satisfy Assumption \ref{assumption-displacement-convex} because $U$ and $f$ could admit unbounded derivatives. Still, with a local convergence argument, our theory in Section \ref{solvable-cases} can be utilized to prove the well-posedness \eqref{lqd-general-HJB}. Towards that end, we first modify the idea in \cite{Cosso23} and introduce an approximation procedure to $\mathcal H$ in \eqref{lqd-general-ham} where each approximation has bounded derivatives and each preserves similar convexity property to $\mathcal H$. The specific definition is as follows:

For $n\geq1$ and $\mu\in\mathcal P_2(\mathbb R^4)$, define
\begin{align}\label{N-poly}
 \mathcal H^{R,n}(\tilde\mu):=\int_{\mathbb R}\cdots\int_{\mathbb R}\mathcal H^R\bigg(\frac1n\sum_{i=1}^n\delta_{z_i}\bigg)\tilde\mu(dz_1)\cdots\tilde\mu(dz_n)=\mathbb E\bigg[\mathcal H^R\bigg(\frac1n\sum_{i=1}^n\delta_{\xi_i}\bigg)\bigg].
\end{align}
Here $\xi_1,\ldots,\xi_q\in\mathbb R$ are i.i.d. with law $\tilde\mu$ in certain probability basis, and $f^R$ is defined in such a way that
\begin{align}\label{convexity}
&\mathcal H^R\bigg(\frac1n\sum_{i=1}^n\delta_{z_i}\bigg)=\mathcal H\bigg(\frac1n\sum_{i=1}^n\delta_{z_i}\bigg),\quad z_i\in\mathbb R,\ \frac1n\sum_{i=1}^n|z_i|^2\leq R,\\
 &D^2_{zz}\mathcal H^R\bigg(\frac1n\sum_{i=1}^n\delta_{z_i}\bigg)=\rho_R\bigg(\frac1n\sum_{i=1}^n|z_i|^2\bigg)D^2_{zz}\mathcal H\bigg(\frac1n\sum_{i=1}^n\delta_{z_i}\bigg),\quad z_i\in\mathbb R,
\end{align}
where $\rho:[0,+\infty)\to[0,1]$ is smooth and satisfies
\begin{align*}
 \rho_R(x)=\left\{\begin{aligned}1,\quad0\leq x\leq R,\\
 0,\quad x\geq R+1.\end{aligned}\right.
\end{align*}
Denote by
\begin{align}\label{truncation}
 \mathcal H_n(\tilde\mu)=\mathcal H^{n,n}(\tilde\mu),\quad\tilde\mu\in\mathcal P_2\big((\mathbb R^d)^4\big),\ n=1,2,\ldots.
\end{align}
Then it is easy to see that $\mathcal H_n$ satisfies Assumption \ref{assumption-displacement-convex} if $\mathcal H$ does. Moreover the following approximation holds:
\begin{lemma}\label{dis-con-app}
 There exists $(\mathcal H_n)_{n\geq1}\subset\mathcal C^6\big(\mathcal P_2(\mathbb R)\big)$ with bounded derivatives such that
 \begin{align*}
  \lim_{n\to+\infty}\mathcal H_n(\mu)=\mathcal H(\mu),\quad\mu\in\mathcal P_2(\mathbb R).
 \end{align*}
\end{lemma}
Let us modify $(\mathcal H,\mathcal U)$ from \eqref{lqd-general-HJB} via $(\mathcal H_n,\mathcal U_n)_{n\geq1}$ in the way as in Lemma \ref{dis-con-app}, and consider the associated HJB equations in the following, where the solution is denoted by $V_n$:
\begin{align}\label{lqd-general-HJB-m}
\left\{\begin{aligned}& \partial_tV_n+\int D_q\partial_\mu V^{(q)}_n(t,\mu,m,q)\mu(dmdq)+\mathcal H_n\big((Id,\partial_\mu V_n)\sharp\mu\big)=0,\\
&\qquad(t,\mu)\in[0,T)\times\mathcal P_2(\mathbb R^d),\\
&V_n(T,\mu)=\mathcal U_n(\mu),\ \mu\in\mathcal P_2(\mathbb R^d).\end{aligned}\right.
 \end{align}
\begin{theorem}\label{lqd-general-HJB-N-wp-1}
 The solutions $(V_n)_{n\geq1}$ converge to the unique classical solution $V$ to \eqref{lqd-general-HJB} on $[0,T]\times\mathcal P_2(\mathbb R^d)$.
\end{theorem}
With the classical solution $V$ to \eqref{lqd-general-HJB}, we construct the optimal robust strategy as follows:
\begin{proposition}\label{xi-implicit}
 Let $V$ be the solution to \eqref{lqd-general-HJB}. The robust optimal liquidation has an admissible optimal strategy
 which is given implicitly by 
 {\small\begin{align}\label{optimal-strategy-1}
  \left\{\begin{aligned}&-\tilde\phi^\top\int_{\mathbb R^d}\partial_\mu V^{(q)}(t,\mu,m,q)\mu(dmdq)+D_\xi U\big(\xi(t,\mu,m,q),m,\xi(t,\mu,\cdot)\sharp\mu\big)\\
  &\quad+\int_{\mathbb R^d}\partial_{\tilde\mu}U\big(\xi(t,\mu,\tilde m,\tilde q),\tilde m,\xi(t,\mu,\cdot)\sharp\mu,\xi(t,\mu,m,q)\big)\mu(d\tilde md\tilde q)-\phi^\top\partial_\mu V^{(q)}(t,\mu,m,q)-\partial_\mu V^{(m)}(t,\mu,m,q)=0,\\
  &\eta^*=\frac1{2\lambda}\partial_\mu V^{(q)}(t,\mu,m,q),\quad\alpha^*=(D\Phi)^{-1}\big(\partial_\mu V^{(q)}(t,\mu,m,q)\big),\ (t,\mu,m,q)\in[0,T]\times\mathcal P_2(\mathbb R^d)\times\mathbb R^{d_1}\times\mathbb R^{d_2}.
  \end{aligned}\right.
 \end{align}}
For $t\in[0,T]$, $(\xi^*,\eta^*,\alpha^*)(t,\cdot)$ from \eqref{optimal-strategy-1} is Lipschitz.
\end{proposition}
\subsubsection{The linear quadratic case with strict liquidation constraint}
We now turn to the application of our well-posedness results on the linear quadratic case. The dynamic programming method for linear quadratic models enables us to obtain and analysis explicit solutions for optimal liquidation with composite uncertainty. Consider the large investor holding $M_t$ shares of assets as in \eqref{model-1}. For the linear quadratic case, the investor's position still follows the dynamic in \eqref{model-1} while it is postulated that\begin{align}\label{lqd-dyn-2}
 dQ_t=(-\rho Q_t+\kappa_1\alpha_t-\gamma_1\xi_t-\gamma_2\mathbb E\xi_t+\gamma_3M_t+\gamma_4\mathbb EM_t)dt+dW^\eta_t,\quad t\in[0,T],\ Q_0=q,
\end{align}
where the price dynamics is impacted by $(\xi_t,M_t)_{t\in[0,T]}$ and its distribution. Given a feedback strategy $\xi_t$ with initial wealth $x$, we are interested in the following cost functional:
\begin{align}\label{lqd-dyn-4-0}
 \left\{\begin{aligned}&J_0(m,q,\xi)=\sup_{(\eta,\alpha)\in\Gamma^\xi_0}\tilde J_0(m,\xi,\eta,\alpha),\\
&\tilde J_0(m,q,\xi,\eta,\alpha)=\mathbb E\bigg[\lambda_1\int_0^T|\xi_t|^2dt+\lambda_{21}\int_0^T|M_t|^2dt+\lambda_{22}\int_0^T|M_t-\mathbb EM_t|^2dt\\
 &\qquad\qquad\qquad-\int_0^T\big(\lambda_3|\eta_t|^2+\lambda_4|\alpha_t|^2\big)dt-\lambda_5\int_0^T\xi^\top_t\gamma_5Q_tdt+\lambda|M_T|^2-\hat\lambda_5M^\top_T\gamma_5Q_T\bigg].\end{aligned}\right.
\end{align}
In the right hand side of \eqref{lqd-dyn-4}, the first term corresponds to the transaction cost caused by the temporary market impact. The second term corresponds to the market risk. The third term penalize the deviation from the reference model, where the integral of $\lambda_3|\eta_t|^2$ describes the relative entropy between the current probability space and the reference probability space and the integral of $\lambda_4|\alpha_t|^2$ is the distance from reference parameter $\bar\alpha=0$. The forth term is the expected liquidation proceeds. The purpose of the last term is two fold: when $\lambda<+\infty$, it serves as the transaction cost from the bulk transaction at time $T$, while the bulk transaction generates liquidation proceeds $\hat\lambda_5M^\top_T\gamma_5Q_T$; when $\lambda=+\infty$, it formally puts a constraint on the liquidation that $X_T=0$, which will be thoroughly examined later. We will denote the value function by $V^\lambda$ to indicate the terminal condition.

In order to refer to previous results and solve the above optimal liquidation problem with composite uncertainty, we consider the controlled distribution flow of  $(M_t,Q_t)_{t\in[0,T]}$ from \eqref{model-1} and \eqref{lqd-dyn-2}, then we perform transformation and simplification to the objective functional. Applying It\^o's formula to $Ke^{C(T-t)}Q^\top_tQ_t$ and plugging the result into \eqref{lqd-dyn-4-0}, we obtain
\begin{align*}
&\tilde J_0(m,q,\xi,\eta,\alpha)=\mathbb E\bigg[\lambda_1\int_0^T|\xi_t|^2dt+\lambda_{21}\int_0^T|M_t|^2dt+\lambda_{22}\int_0^T|M_t-\mathbb EM_t|^2dt\\
 &\qquad\qquad\qquad-\int_0^T\big(\lambda_3|\eta_t|^2+\lambda_4|\alpha_t|^2\big)dt-\lambda_5\int_0^T\xi^\top_t\gamma_5Q_tdt-C\int_0^TK_tQ^\top_tQ_tdt\\
 &\qquad\qquad\qquad+2\int_0^TK_tQ^\top_t\big(-\rho Q_t+\kappa_1\alpha_t-\gamma_1\xi_t-\gamma_2\mathbb E\xi_t+\gamma_3M_t+\gamma_4\mathbb EM_t+\kappa_2\eta_t\big)dt\\
 &\qquad\qquad\qquad+\lambda|M_T|^2-\hat\lambda_5M^\top_T\gamma_5Q_T-K_TQ^\top_TQ_T\bigg]+K_0\mathbb E[q^\top q],
\end{align*}
where $K_t:=Ke^{C(T-t)}$ with $K$ and $C$ to be determined. Therefore, solving \eqref{lqd-dyn-4-0} is equivalent to solving the minima of $J(0,m,q,\xi)$, where
\begin{align}\label{lqd-dyn-4}
 \left\{\begin{aligned}&J(t,m,q,\xi)=\sup_{(\eta,\alpha)\in\Gamma^\xi_t}\tilde J(t,m,\xi,\eta,\alpha),\\
&\tilde J(t,m,q,\xi,\eta,\alpha)=\mathbb E\bigg[\lambda_1\int_t^T|\xi_s|^2ds+\lambda_{21}\int_t^T|M_s|^2ds+\lambda_{22}\int_t^T|M_s-\mathbb EM_s|^2ds\\
 &\qquad\qquad\qquad-\int_t^T\big(\lambda_3|\eta_s|^2+\lambda_4|\alpha_s|^2\big)ds-\lambda_5\int_t^T\xi^\top_s\gamma_5Q_sds-C\int_t^TK_sQ^\top_sQ_sds\\
 &\qquad\qquad\qquad+2\int_t^TK_sQ^\top_s\big(-\rho Q_s+\kappa_1\alpha_s-\gamma_1\xi_s-\gamma_2\mathbb E\xi_s+\gamma_3M_s+\gamma_4\mathbb EM_s+\kappa_2\eta_s\big)ds\\
 &\qquad\qquad\qquad+\lambda|M_T|^2-\hat\lambda_5M^\top_T\gamma_5Q_T-K_TQ^\top_TQ_T\bigg].
\end{aligned}\right.
\end{align}
Define the value function
\begin{align*}
 V^\lambda(t,\mu):=\inf_{\xi\in\mathcal U_t} J(t,m,q,\xi),\quad t\in[0,T],\ {\rm Law}(m,q)=\mu.
\end{align*}
Using the dynamic programming principle, we formally obtain the HJBI equation \eqref{HJB} for $V^\lambda$, where the Hamiltonian is the following:
\begin{align}\label{lqd-ham-1}
 \mathcal H_t(\tilde\mu)&=\inf_{\xi\in L^2(\Omega;\mathbb R^{d_1})}\sup_{\substack{\eta\in L^2(\Omega;\mathbb R^{d_2})\\\alpha\in\mathbb R^{d_2}}}\mathbb E\bigg[\lambda_1|\xi|^2+(\lambda_{21}+\lambda_{22})|M|^2-\lambda_{22}|\mathbb EM|^2-\lambda_3|\eta|^2-\lambda_4|\alpha|^2\notag\\
 &-(\tilde p^\top_1+\lambda_5 Q^\top\gamma^\top_5 )\xi+(\tilde p_2^\top+2K_tQ^\top)\big(-\rho Q+\kappa_1\alpha-\gamma_1\xi-\gamma_2\mathbb E\xi+\gamma_3M+\gamma_4\mathbb EM+\kappa_2\eta\big)-CK_tQ^\top Q\bigg]\notag\\
 &=-\frac1{4\lambda_1}\mathbb E\big[|-\tilde p_1-\lambda_5\gamma_5 Q-\gamma^\top_1(\tilde p_2+2K_tQ)-\gamma^\top_2\mathbb E(\tilde p_2+2K_tQ)|^2\big]+\frac{\kappa^2_1}{4\lambda_4}|\mathbb E(\tilde p_2+2K_tQ)|^2\notag\\
 &\quad+\frac{\kappa^2_2}{4\lambda_3}\mathbb E\big[|\tilde p_2+2K_tQ|^2\big]+(\lambda_{21}+\lambda_{22})\mathbb E\big[|M|^2\big]-\lambda_{22}|\mathbb EM|^2+\mathbb E\big[(\tilde p^\top_2+2K_tQ^\top)\gamma_3M\big]\notag\\
 &\quad-(C+2\rho)K_t\mathbb E[Q^\top Q]+\mathbb E\big[\tilde p^\top_2+2K_tQ^\top\big]\gamma_4\mathbb EM\notag\\
 &=:\mathbb E\big[\Lambda^\top Q^{(1)}_t\Lambda\big]-\varepsilon\mathbb E[|\tilde p_1|^2]+\mathbb E[\Lambda]^\top Q^{(2)}_t\mathbb E[\Lambda],
\end{align}
and the terminal condition is
\begin{align*}
 V^\lambda(T,\mu)=\int_{\mathbb R^{d_2}}\int_{\mathbb R^{d_1}}(\lambda|m|^2-\hat\lambda_5 m^\top q-K|q|^2)\mu(dmdq),\quad\mu\in\mathcal P_2(\mathbb R^d).
\end{align*}
Here $\Lambda=(M,Q,\tilde p_1,\tilde p_2)^\top\in L^2(\Omega,\mathcal F,\mathbb P;\mathbb R^d\times\mathbb R^d)$ is the lifting of $\tilde\mu\in\mathcal P_2(\mathbb R^d\times\mathbb R^d)$. The specific expression of $Q^{(1)},\ Q^{(2)}$ is given in \eqref{app-Q}. After directly adapting previous verification theorem and well-posedness results to the time-inhomogeneous case, we can show the existence and uniqueness of $V^\lambda$.
\begin{theorem}\label{application-LQ-1}
Suppose that \eqref{generalized-displacement-convexity-LQ} holds for some $\varepsilon>0$ on the following $\big(Q^{(1)}_t,Q^{(2)}_t\big)_{t\in[0,T]}$:
\begin{align}\label{app-Q}
 \left\{\begin{aligned}Q^{(1)}_t=\left(\begin{matrix}\lambda_{21}+\lambda_{22}&K_t\gamma^\top_3&0&\frac{\gamma^\top_3}2\\ \ \\
K_t\gamma_3&{\small\substack{\big[-(C+2\rho)K_t+\frac{\kappa^2_2K^2_t}{\lambda_3}\big]I_{d_2}\\-\frac{(\lambda_5\gamma_5+2K_t\gamma^\top_1)^\top(\lambda_5\gamma_5+2K_t\gamma^\top_1)}{4\lambda_1}}}&-\frac{\lambda_5\gamma^\top_5+2K_t\gamma_1}{4\lambda_1}&\substack{-\frac1{4\lambda_1}(\lambda_5\gamma^\top_5+2K_t\gamma_1)\gamma^\top_1\\+\frac{\kappa^2_2K_t}{2\lambda_3}}\\ \ \\
0&-\frac{\lambda_5\gamma_5+2K_t\gamma^\top_1}{4\lambda_1}&-\frac1{4\lambda_1}+\varepsilon&-\frac{\gamma^\top_1}{4\lambda_1}\\\ \\
\frac{\gamma_3}2&\substack{-\frac1{4\lambda_1}\gamma_1(\lambda_5\gamma_5+2K_t\gamma^\top_1)\\+\frac{\kappa^2_2K_t}{2\lambda_3}}&-\frac{\gamma_1}{4\lambda_1}&-\frac{\gamma_1\gamma^\top_1}{4\lambda_1}+\frac{\kappa^2_2}{4\lambda_3}
 \end{matrix}\right),\\ \ \\
 Q^{(2)}_t=\left(\begin{matrix}-\lambda_{22}&K_t\gamma^\top_4&0&\frac{\gamma^\top_4}2\ \\ \\ 
 K_t\gamma_4&{\small\substack{-\frac1{\lambda_1}(\lambda_5\gamma^\top_5+2K_t\gamma_1)K_t\gamma^\top_2\\-\frac{K^2_t}{\lambda_1}\gamma_2\gamma^\top_2+\frac{K^2_t\kappa^2_1}{\lambda_4}}}&-\frac{K_t\gamma_2}{2\lambda_1}&-\frac{K_t\gamma_2\gamma^\top_1}{2\lambda_1}+\frac{\kappa^2_1K_t}{2\lambda_4}\\ \ \\
0&-\frac{K_t\gamma^\top_2}{2\lambda_1}&0&-\frac{\gamma^\top_2}{4\lambda_1}\\ \ \\
\frac{\gamma_4}2&-\frac{K_t\gamma_1\gamma^\top_2}{2\lambda_1}+\frac{\kappa^2_1K_t}{2\lambda_4}&-\frac{\gamma_2}{4\lambda_1}&-\frac{\gamma_1\gamma^\top_2}{2\lambda_1}-\frac{\gamma_2\gamma^\top_2}{4\lambda_1}+\frac{\kappa^2_1}{4\lambda_4}
 \end{matrix}\right).\end{aligned}\right.
\end{align}
Then the value function $V^\lambda$ is given by \eqref{lambda-value}, where $(Q^{(1)},Q^{(2)})$ in \eqref{LQ-ham-penalized} is replcaced with $\big(Q^{(1)}_t,Q^{(2)}_t\big)_{t\in[0,T]}$. The optimal feedback strategy is given by
\begin{align}\label{lambda-feedback}
 \left\{\begin{aligned}
  &\xi^{*,\lambda}(t,M,Q,\mu)=\frac1{2\lambda_1}\bigg(p^\lambda_1(t,M,Q,\mu)+\gamma^\top_1\big(p^\lambda_2(t,M,Q,\mu)+2K_tQ\big)\\
  &\qquad\qquad\qquad\qquad+\gamma_2\tilde{\mathbb E}[p^\lambda_2(t,\tilde M,\tilde Q,\mu)+2K_t\tilde Q]+\lambda_5\gamma_5Q\bigg),\\
  &\eta^{*,\lambda}(t,M,Q,\mu)=\frac{\kappa_2}{2\lambda_3}\big(p^\lambda_2(t,M,Q,\mu)+2K_tQ\big),\\
  &\alpha^{*,\lambda}(t,M,Q,\mu)=\frac{\kappa_1}{2\lambda_4}\tilde{\mathbb E}[p^\lambda_2(t,\tilde M,\tilde Q,\mu)+2K_tQ].
 \end{aligned}\right.
\end{align}
Here for $i=1,2,\ s\in[0,T]$, we denote by
\begin{align*}
 &\big(p^\lambda_1,p^\lambda_2\big)^\top(t,M,Q,\mu):=\partial_\mu V^\lambda(t,\mu,M,Q)=2a^{(1)}(t)\left(\begin{matrix}M\\
 Q
 \end{matrix}\right)+2a^{(3)}(t)\left(\begin{matrix}\int_{\mathbb R^d}m\mu(dmds)\\
 \int_{\mathbb R^d}s\mu(dmds)
 \end{matrix}\right)+a^{(2)}(t),\\
 &(t,\mu,M,Q)\in[0,T]\times\mathcal P_2(\mathbb R^d)\times\mathbb R^{d_1}\times\mathbb R^{d_2}.
\end{align*}
\end{theorem}
A more involved situation is the multi-asset liquidation problems with strict liquidation constraint. In order to make the strict liquidation constraint rigorous for the liquidation problem with uncertainty, we will work with the admissible set $\tilde{\mathcal U}^{Con}_t$ consisting of $\xi\in\tilde{\mathcal U}_t$ such that $M_T=0$ for all $(\eta_s,\alpha_s)_{s\in[t,T]}\in\mathcal V^\xi_t$. Although the candidate value function has been constructed in Theorem \ref{limit-V-lambda}, our verification is more involved than the usual one since we are now dealing with a mean field control problem with both robustness and terminal constraint. The candidate optimal strategy and the corresponding controlled processes are constructed with the candidate value function in Theorem \ref{limit-V-lambda} as follows:
\begin{align}\label{MS-path}
  \left\{\begin{aligned}dM^*_s&=-\xi^*_sds,\\
  dQ^*_s&=\big[\alpha^*_s+\eta^*_s-\gamma_1\xi^*_s-\gamma_2\mathbb E\xi^*_s+\gamma_3M^*_s+\gamma_4\mathbb EM^*_s)\big]ds+dW_s,\ s\in[0,T],\\
  (M^*_0,&Q^*_0)=(m,q).
  \end{aligned}\right.
\end{align}
For $\varphi=\xi^*,\ \alpha^*,\ \eta^*$, $(\varphi_s)_{s\in[t,T]}$ is defined via a feedback strategy in the same way as in \eqref{lambda-feedback} with $\lambda=+\infty$, where $\varphi_s:=\varphi\big(s,M^*_s,Q^*_s,\mu_{(M^*_s,Q^*_s)}\big)$ and
\begin{align*}
 &\big(p^{+\infty}_1,p^{+\infty}_2\big)^\top(t,M,Q,\mu):=\partial_\mu V^{+\infty}(t,\mu,M,Q)=2a^{+\infty}_1(t)\left(\begin{matrix}M\\
 Q
 \end{matrix}\right)+2a^{+\infty}_3(t)\left(\begin{matrix}\int_{\mathbb R^d}m\mu(dmds)\\
 \int_{\mathbb R^d}s\mu(dmds)
 \end{matrix}\right),\\
 &(t,\mu,M,Q)\in[0,T]\times\mathcal P_2(\mathbb R^d)\times\mathbb R^{d_1}\times\mathbb R^{d_2}.
\end{align*}
The next theorem reveals that the candidate in \eqref{MS-path} is indeed optimal.
\begin{theorem}\label{verification-LQ-constrained}
Suppose that \eqref{generalized-displacement-convexity-LQ} holds for some $\varepsilon>0$ on the $\big(Q^{(1)}_t,Q^{(2)}_t\big)_{t\in[0,T]}$ in \eqref{app-Q}. The value function and the optimal strategy for the constrained optimal liquidation problem are given in \eqref{constraint-value} and \eqref{MS-path}, where $(Q^{(1)},Q^{(2)})$ in \eqref{LQ-ham-penalized} is replcaced with $\big(Q^{(1)}_t,Q^{(2)}_t\big)_{t\in[0,T]}$.
\end{theorem}

\section{Numerical Examples}\label{numerical examples}
In this section we get insights on the optimal liquidation problem with composite uncertainty via numerical examples. The investor liquidates single asset with the liquidation constraint $M_T=0$. The default parameters in \eqref{model-1},\ \eqref{lqd-dyn-2},\ \eqref{lqd-dyn-4} are set as follows:
\begin{align*}
 &T=1,\ \kappa_1=\kappa_2=0.1,\ \lambda_1=0.5,\ \lambda_{21}=\lambda_{22}=1,\ \lambda_3=\lambda_4=1,\ \lambda_5=\hat\lambda_5=2,\\ 
 &\gamma_1=\gamma_2=0.2,\ \gamma_3=\gamma_4=0,\ \gamma_5=1,\ C=1,\ K=e^{-1},\ \rho=1.
\end{align*}
In the numerical examples that follow, the parameters would by default take the above values unless specified otherwise. In view of Theorem \ref{limit-V-lambda} and Theorem \ref{verification-LQ-constrained}, we take $\lambda=10^5$ in \eqref{lambda-value} so that the optimal liquidation strategy with the constraint $M_T=0$ is approximately obtained.

\begin{figure}[htbp]
    \centering
    \begin{tabular}{ccc} 
        \includegraphics[width=0.3\textwidth]{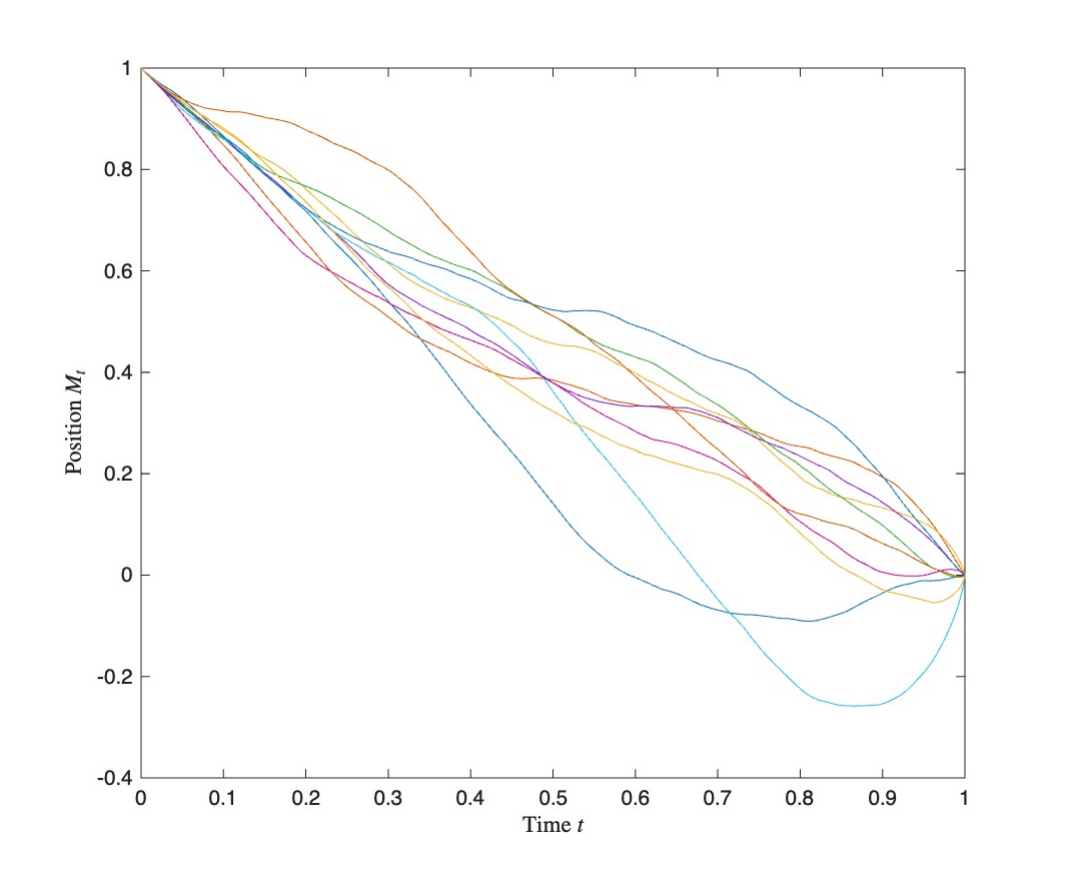} & 
        \includegraphics[width=0.3\textwidth]{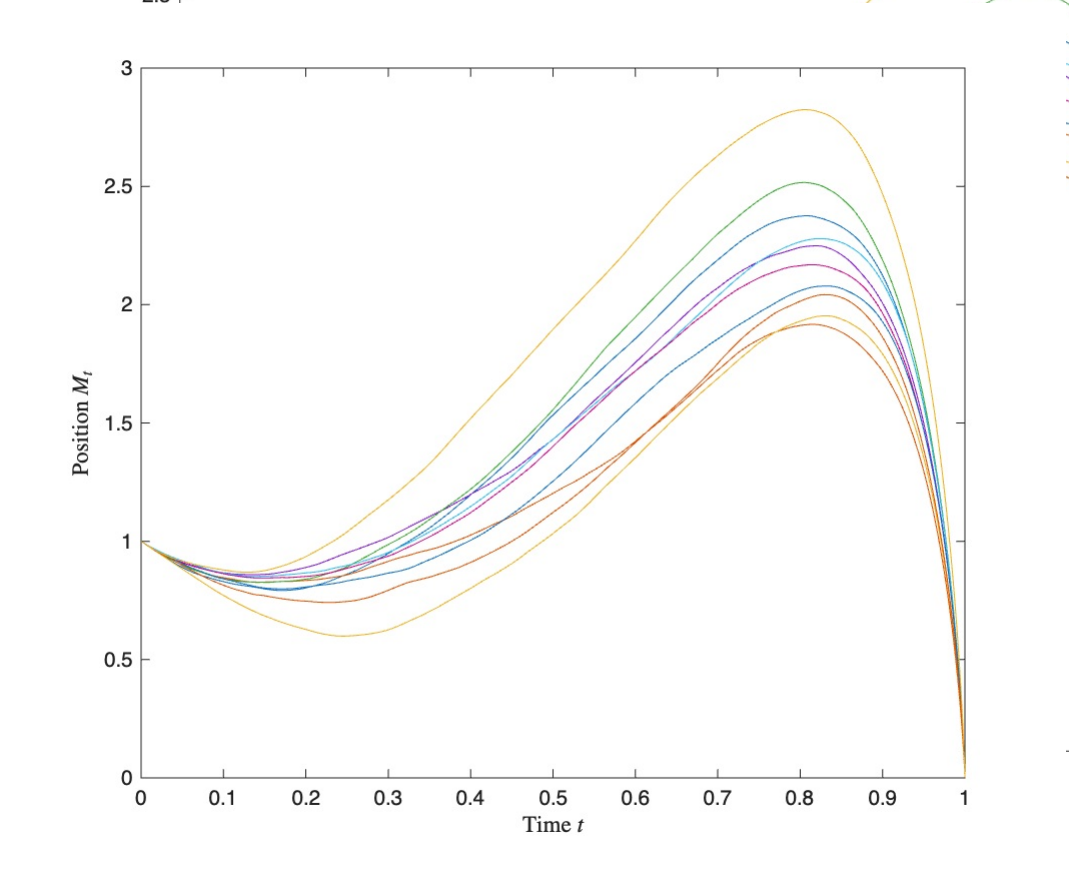} & \includegraphics[width=0.3\textwidth]{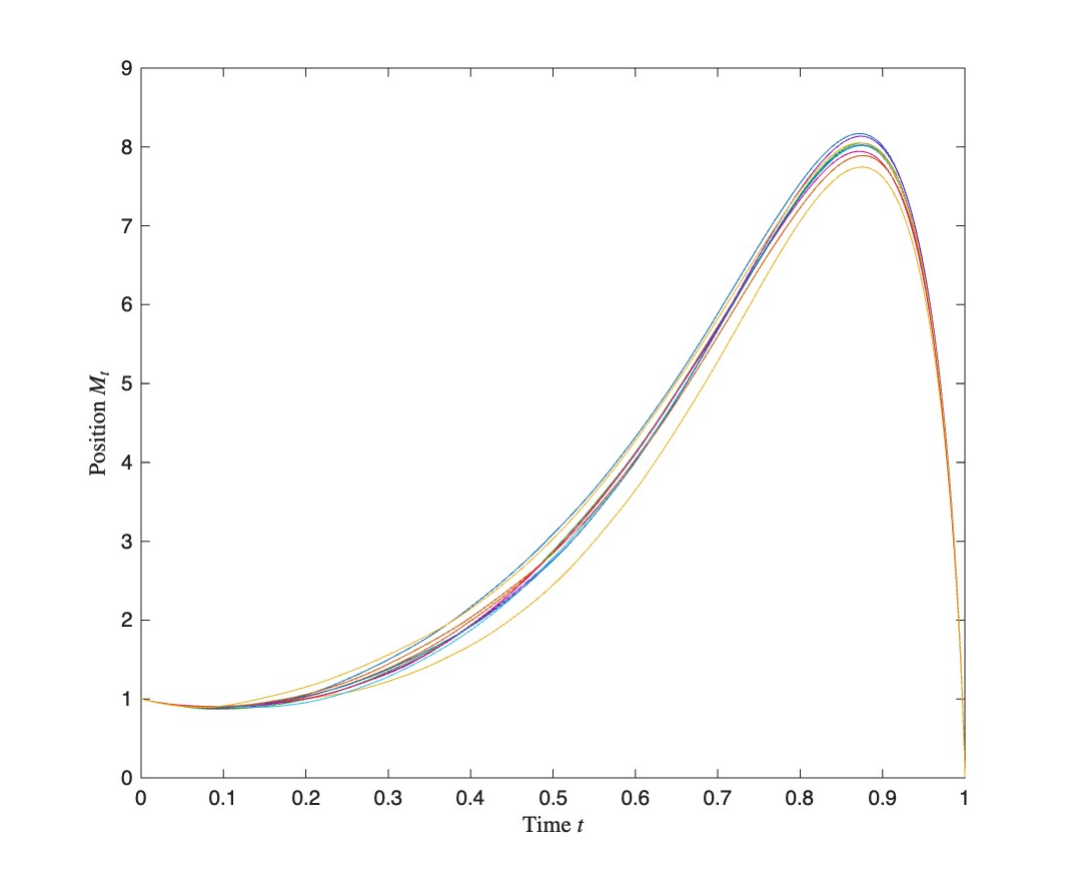}\\
        (a) $\kappa_1=0.1$ & (b) $\kappa_1=3$ & (c) $\kappa_1=4$\\        \includegraphics[width=0.3\textwidth]{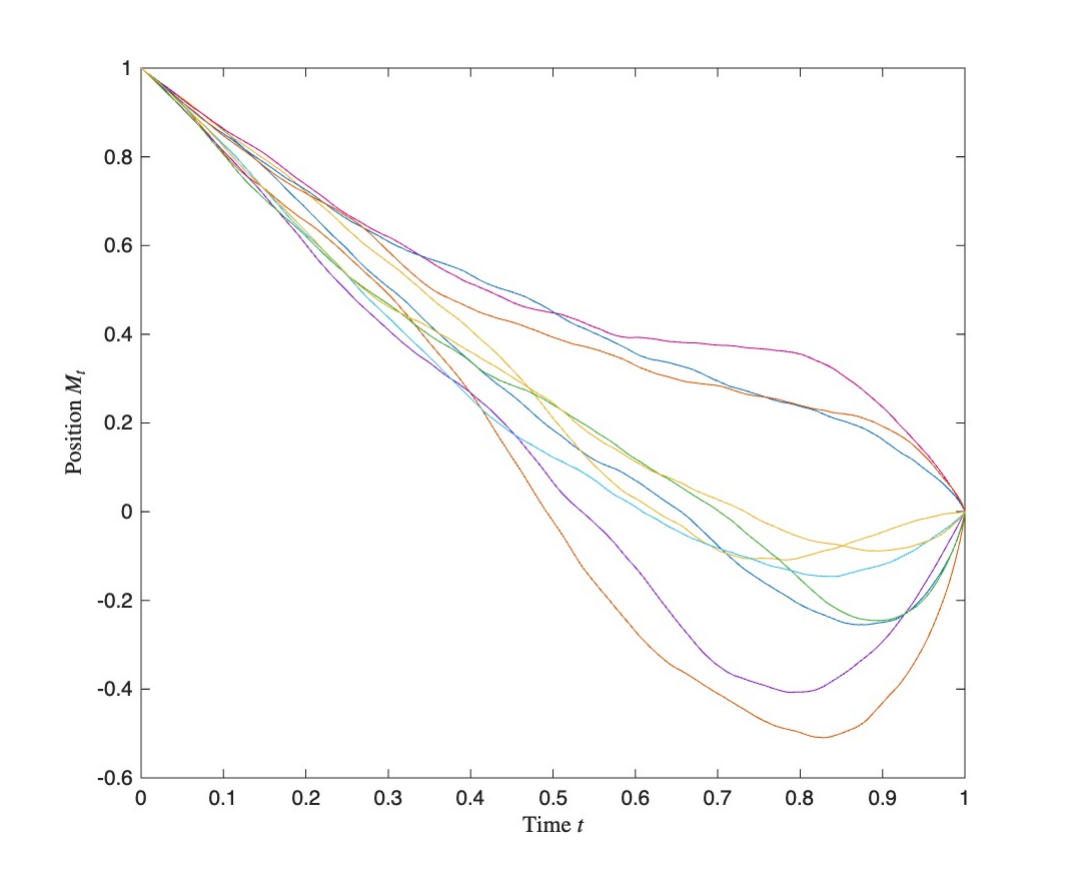} & 
        \includegraphics[width=0.3\textwidth]{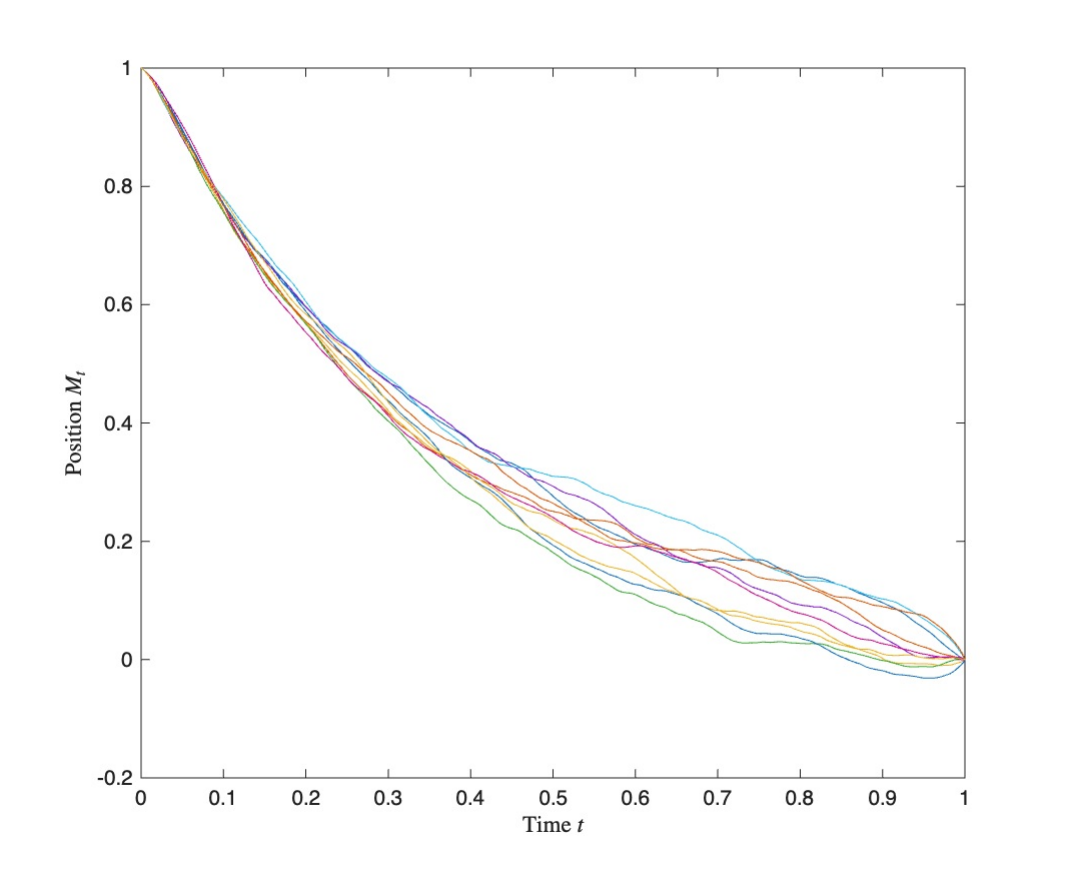} & \includegraphics[width=0.3\textwidth]{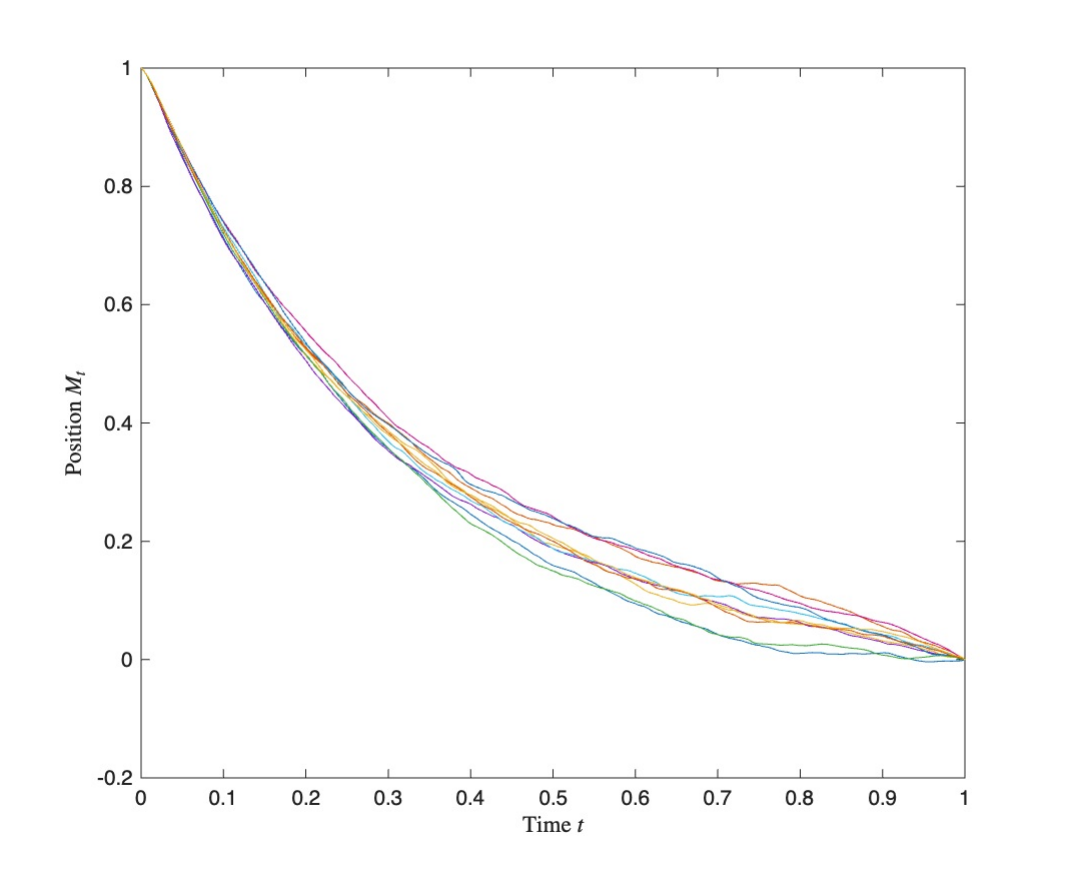}\\
        (d) $\kappa_2=0.1$ & (e) $\kappa_2=7$ & (f) $\kappa_2=10$\\    \end{tabular}
    \caption{Simulation of optimal trajectories with different uncertainty}\label{simulation-uncertainty-1}
\end{figure}
Figure \ref{simulation-uncertainty-1} shows the simulation of multiple trajectories of the optimal position under different $(\kappa_1,\kappa_2)$. It is indicated that the consideration of both kinds of uncertainty admit a regularization effect so that each sample path has similar patterns. The intuition is the following: with larger $\kappa_1,\ \kappa_2$ comes the less penalty for the same deviation from preset parameters, meaning that the investor is less certain about the model parameters or the underlying probability measure. Given the enlarged uncertainty, the investor tends to liquidate her assets along a trajectory close to certain benchmark trajectory despite the randomness from the Brownian motion in $Q_t$.

\begin{figure}[htbp]
    \centering
    \begin{tabular}{ccc} 
        \includegraphics[width=0.3\textwidth]{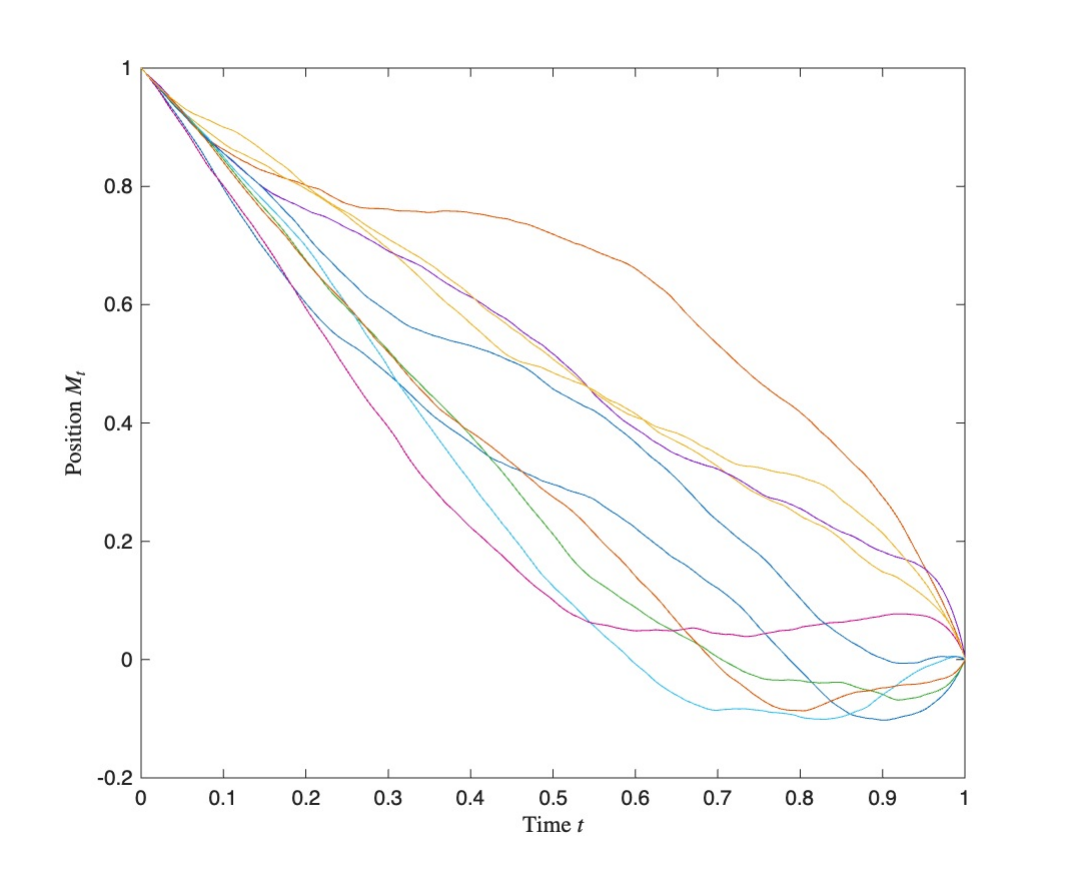} & 
        \includegraphics[width=0.3\textwidth]{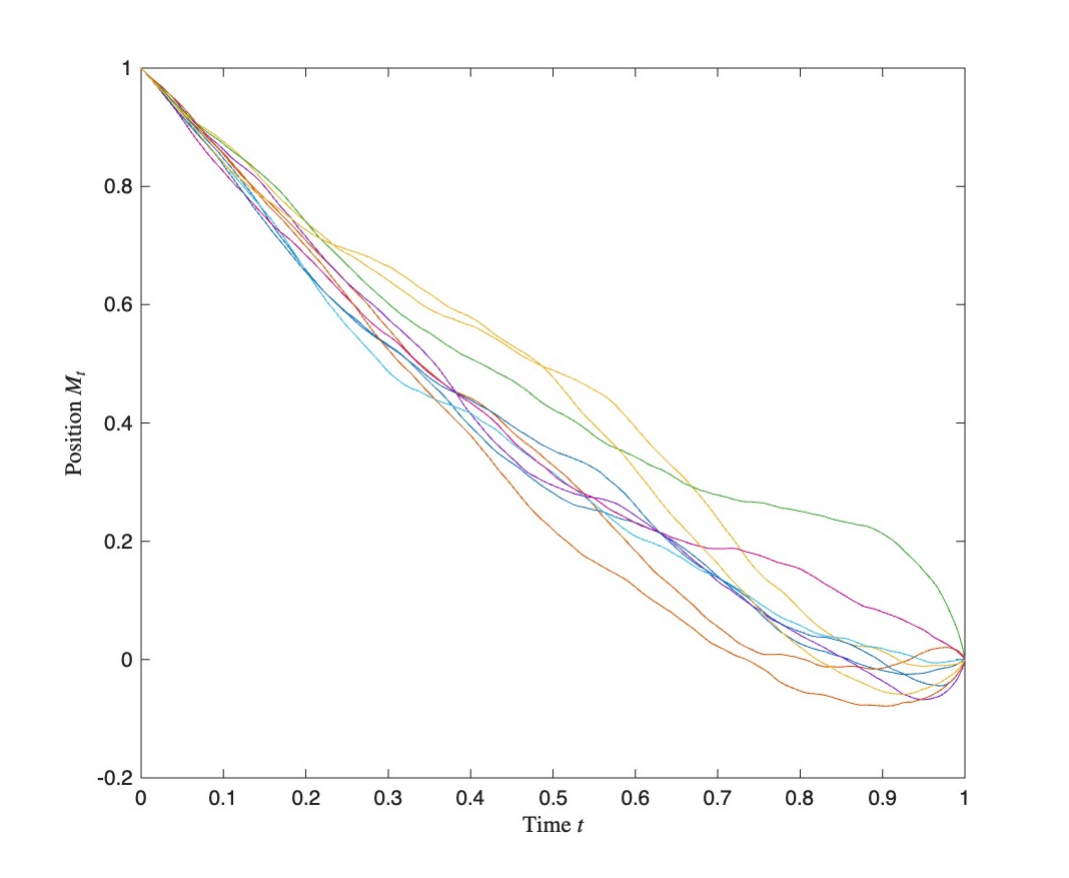} & \includegraphics[width=0.3\textwidth]{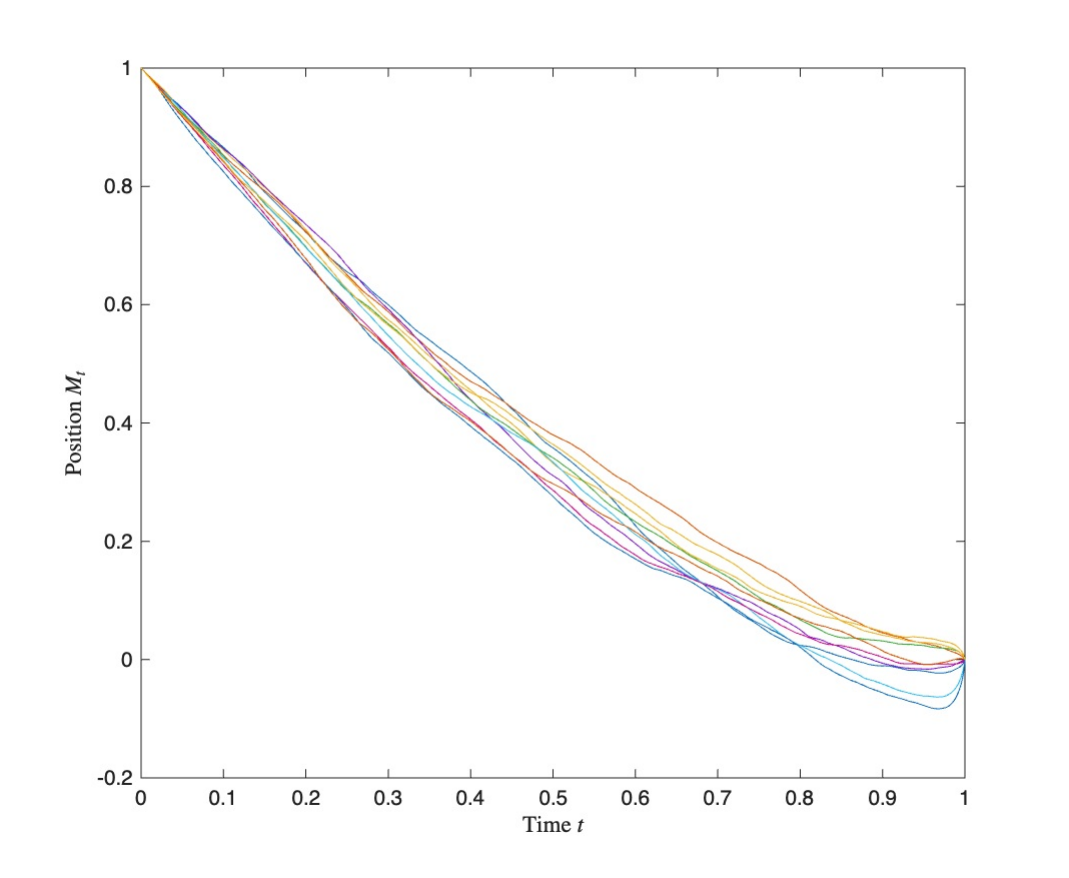}\\
        (a) $\lambda_{22}=1$ & (b) $\lambda_{22}=50$ & (c) $\lambda_{22}=500$\\
    \end{tabular}
    \caption{Simulation of optimal trajectories with different management risk}\label{simulation-management-risk}
\end{figure}
The regularization effect also happens when $\lambda_{22}$ is tuned larger, which is shown in Figure \ref{simulation-management-risk}. For the investor with the objective functional containing larger $\lambda_{22}$, more penalization is put on the variance of position $M_t$, which gives the intuition that the investor is more averse to the oscillation in their position. As a result, for the case with large $\lambda_{22}$, each trajectory of the position tends to be close to certain benchmark.
\begin{figure}[htbp]
    \centering
    \begin{tabular}{cc} 
        \includegraphics[width=0.4\textwidth]{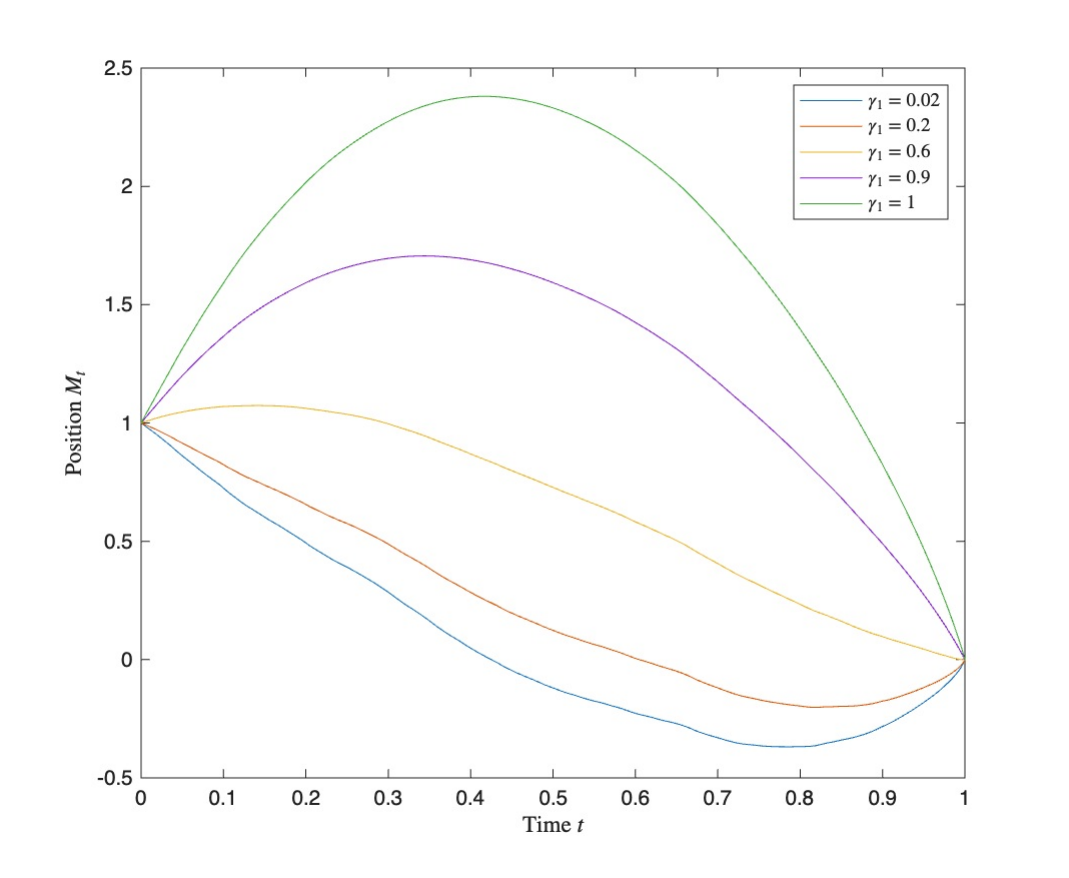} & 
        \includegraphics[width=0.4\textwidth]{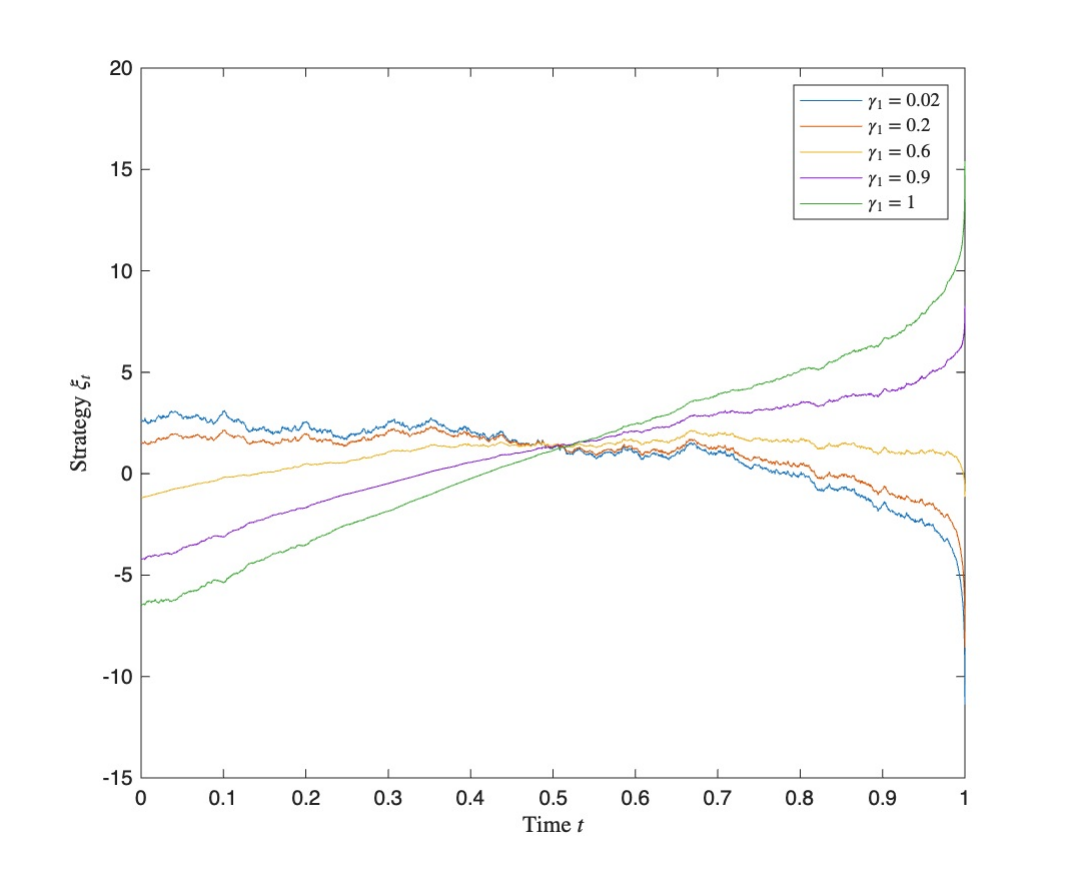} \\
         \includegraphics[width=0.4\textwidth]{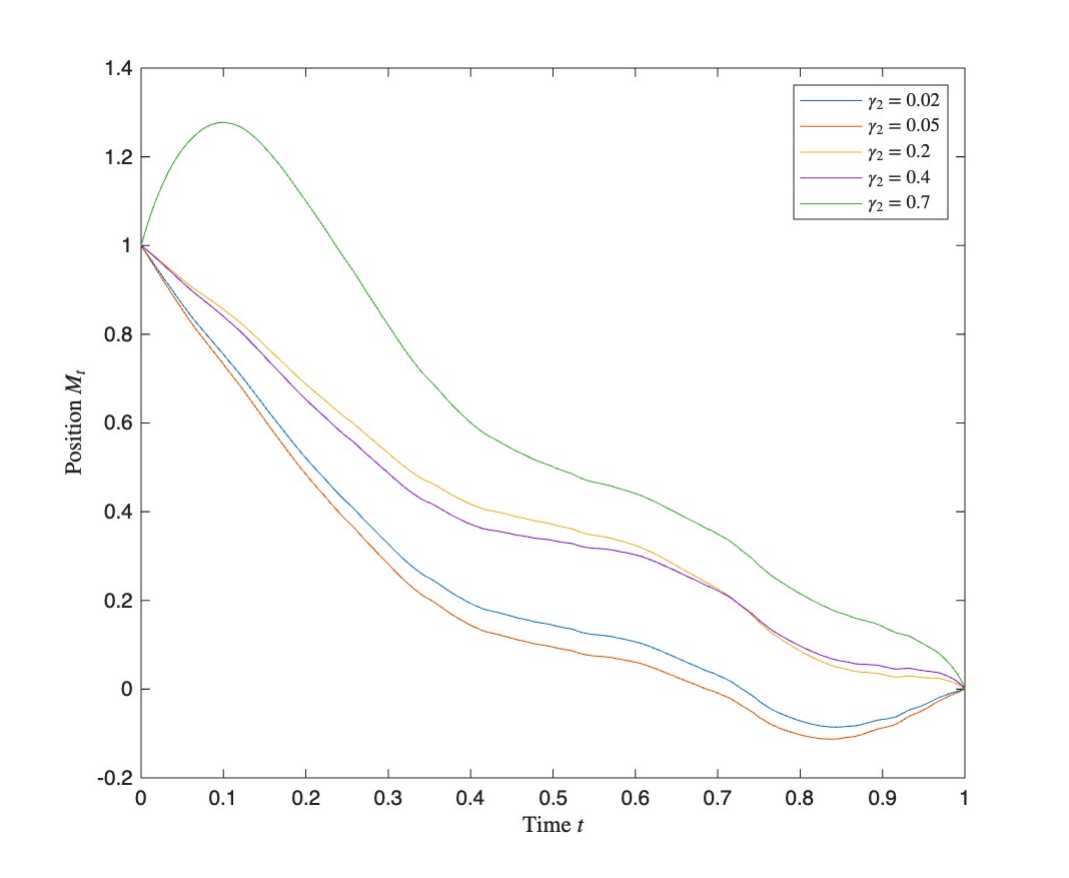} & 
        \includegraphics[width=0.4\textwidth]{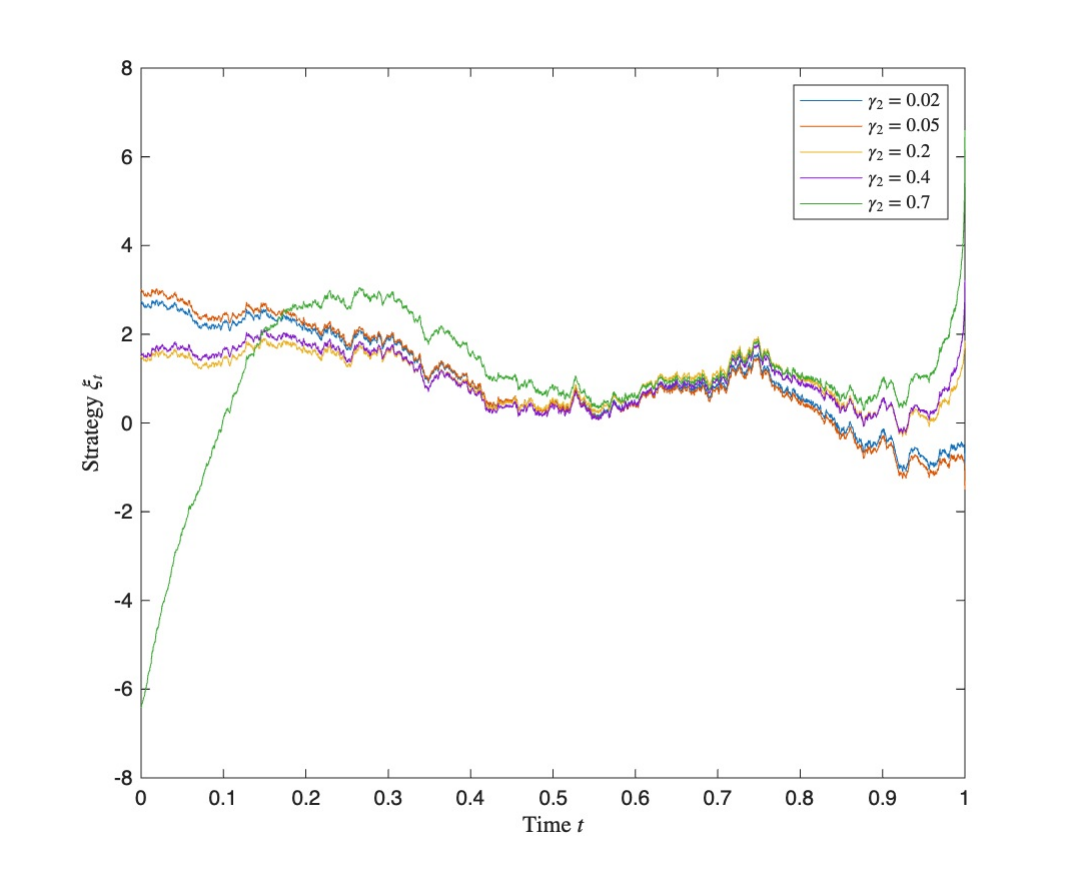} \\
    \end{tabular}
    \caption{Simulation of optimal positions and strategies with different persistent impact factor}\label{simulation-gamma}
\end{figure}
In Figure \ref{simulation-gamma} we report the simulation results on the optimal strategy versus the persistent impact factor $\gamma_1,\ \gamma_2$, where in each plot we fix the sample path of Brownian motion so that the trajectories are comparable. According to the simulation, given the consideration of child orders and uncertainty, the optimal strategy is no longer convex in time, which is different from \cite{Horst19}. On the other hand, it is still observed that the investor tends to do more transaction, either long or short, near the initial time or the maturity, especially when the transaction cost is high.

\section{Proofs for main results}\label{Proofs}
\subsection{Proofs for Theorem \ref{verification}}
\begin{proof}[Proof of Theorem \ref{verification}]
Let $\xi\in L^2(\Omega,\mathcal F,\mathbb P;\mathbb R^d)$ be the lifting of $\mu$ in an appropriate probability space $(\Omega,\mathcal F,\mathbb P)$ that is rich enough to contain a Brownian motion $(W_s)_{s\in[t,T]}$ independent of $\xi$. Consider the mean field SDE
\begin{align}\label{optimal-path}
dX^*_s=a\big(\theta^*_s,\eta^*_s,\alpha^*_s,X^*_s,\mu_{(X^*_s,\theta^*_s,\eta^*_s)}\big)ds+\sigma dW_s,\quad X^*_t=\xi,
 \end{align}
where
\begin{align*}
 (\theta^*_s,\eta^*_s):=(\theta,\eta)(s,X^*_s,\mu_{X^*_s}),\ \alpha^*_s:=\alpha(s,\mu_{X^*_s}).
\end{align*}
According assumptions on $(\theta,\eta,\alpha)$, \eqref{optimal-path} is well posed. Further consider any other admissible feedback function $\tilde\theta\in\mathcal U_t$ and the corresponding state process
\begin{align*}
dX^\theta_s=a\big(\theta_s,\eta_s,\alpha_s,X^\theta_s,\mu_{(X^\theta_s,\theta_s,\eta_s)}\big)ds+\sigma dW_s,\quad X^\theta_t=\xi,
\end{align*}
where 
\begin{align*}
 (\theta_s,\eta_s):=(\tilde\theta,\eta)\big(s,X^\theta_s,\mu_{X^\theta_s}\big),\ \alpha_s:=\alpha\big(s,\mu_{X^\theta_s}\big).
\end{align*}
We note that the above is well-posed due to the admissibility of $\tilde\theta$. According to \eqref{saddlepoint-1} and \eqref{saddlepoint-2} with $(X,Y)=\big(X^\theta_s,\partial_\mu V(s,\mu_{X^\theta_s},X^\theta_s)\big)$, an application of It\^o's formula yields
{\small\begin{align}\label{veri-0}
&\quad V(T,\mu_{X^\theta_T})-V(t,\mu_{X^\theta_t})\notag\\
&=\int_t^T\bigg\{\partial_tV(s,\mu_{X^\theta_s})+\mathbb E\big[\frac12{\rm tr}[\sigma^\top\sigma D_x\partial_\mu V(s,\mu_{X^\theta_s},X^\theta_s)]\big]+a\big(\theta_s,\eta_s,\alpha_s,X^\theta_s,\mu_{(X^\theta_s,\theta_s,\eta_s)}\big)\cdot\partial_\mu V(s,\mu_{X^\theta_s},X^\theta_s)\bigg\}ds\notag\\
 &\geq\int_t^T\bigg\{\partial_tV(s,\mu_{X^\theta_s})+\mathbb E\big[\frac12{\rm tr}[\sigma^\top\sigma D_x\partial_\mu V(s,\mu_{X^\theta_s},X^\theta_s)]\big]+\mathcal H(\tilde\mu^\theta_s)\bigg\}ds-\int_t^Tf\big(\alpha_s,\mu_{(X^{\theta}_s,\theta_s,\eta_s)}\big)ds,
\end{align}}
where $\tilde\mu^\theta_s:=\big(Id,\partial_\mu V(s,\mu_{X^\theta_s},\cdot)\big)\sharp\mu_{X^\theta_s},\ s\in[0,T].$
Plugging the boundary conditions into the above, we obtain
\begin{align*}
 U(\mu_{X^\theta_T})+\int_t^Tf\big(\mu_{(X^{\theta}_s,\theta_s,\eta_s)}\big)ds\geq V(t,\mu_{X^\theta_t}).
\end{align*}
Hence
\begin{align*}
 V(t,\mu_{X^\theta_t})\leq\tilde J(t,\mu,\tilde\theta,\hat\alpha,\hat\eta)\leq\sup_{(\eta,\alpha)\in\mathcal V^{\tilde\theta}_t}\tilde J(t,\mu,\tilde\theta,\alpha,\eta),\quad\text{where}\quad(\hat\eta_s,\hat\alpha_s)=\big(\eta(s,X^\theta_s,\mu^\theta_s),\alpha(s,\mu^\theta_s)\big).
\end{align*}
Since $\tilde\theta\in\mathcal U_t$ is arbitrary,
\begin{align*}
 V(t,\mu)\leq\inf_{\theta\in\mathcal U_t}\sup_{(\eta,\alpha)\in\mathcal V^\theta_t}\tilde J(t,\mu,\theta,\eta).
\end{align*}
The inequalities above become equalities when $\tilde\theta=\theta^*$, i.e.,
\begin{align*}
 V(t,\mu)=\tilde J(t,\mu,\theta^*,\eta^*,\alpha^*).
\end{align*}
It remains to show that
\begin{align*}
 \tilde J(t,\mu,\theta^*,\eta^*,\alpha^*)=J(t,\mu,\theta^*),
\end{align*}
in other words,
\begin{align}\label{eta-objective}
 \tilde J(t,\mu,\theta^*,\eta^*,\alpha^*)=\sup_{(\eta,\alpha)\in\mathcal V^{\theta^*}_t}\tilde J(t,\mu,\theta^*,\eta,\alpha),
\end{align}
where for any $(\eta,\alpha)\in\mathcal V^{\theta^*}_t$, $\tilde J(t,\mu,\theta^*,\eta,\alpha)$ corresponds to the state process generated in \eqref{dynamic}. In fact, an application of It\^o's formula together with the assumption on saddle point \eqref{saddlepoint-1} and \eqref{saddlepoint-2} yields
{\small\begin{align}\label{veri-0-1}
&\quad V(T,\mu_{X^{\theta^*}_T})-V(t,\mu_{X^{\theta^*}_t})\notag\\
&=\int_t^T\bigg\{\mathbb E\big[\frac12{\rm tr}[\sigma^\top\sigma D_x\partial_\mu V(s,\mu_{X^{\theta^*}_s},X^{\theta^*}_s)]+a\big(\theta^*(s,X^{\theta^*}_s,\mu_{X^{\theta^*}_s}),\eta_s,\alpha_s,X^{\theta^*}_s,\mu_{(X^{\theta^*}_s,\theta^*_s,\eta_s)}\big)\partial_\mu V(s,\mu_{X^{\theta^*}_s},X^{\theta^*}_s)\big]\notag\\
&\qquad+\partial_tV(s,\mu_{X^{\theta^*}_s})\bigg\}ds\notag\\
 &\leq\int_t^T\bigg\{\partial_tV(s,\mu_{X^{\theta^*}_s})+\mathbb E\big[\frac12{\rm tr}[\sigma^\top\sigma D_x\partial_\mu V(s,\mu_{X^{\theta^*}_s},X^{\theta^*}_s)]\big]+\mathcal H(\tilde\mu^{\theta^*}_s)\bigg\}ds-\int_t^Tf\big(\alpha_s,\mu_{(X^{\theta^*}_s,\theta^*_s,\eta_s)}\big)ds.
\end{align}}
Here ``='' holds when $(\eta,\alpha)=(\eta^*,\alpha^*)$ as in \eqref{optimal-path}, which leads to \eqref{eta-objective}.
\end{proof}

\subsection{Proofs for Theorem \ref{existence-of-decoupling-field}}
The blueprint for the proof is in the spirit of \cite{Liao24}, i.e., the particle system approach. In a nutshell, we infer the a priori estimates on \eqref{HJB} via the corresponding particle systems. However, our method is by no means the same as in \cite{Liao24}. The major difference is that the current generalized Hamiltonian is no longer dispalcement concave in momentum. Instead, both position and momentum contain displacement convex and concave components at the same time. Such difference causes essential difficulties when dealing with the a priori estimates on both the gradients and Hessians of solutions to the corresponding particle systems. Due to the difficulties coming with the aforementioned differences, some crucial lemmas and a priori estimates in \cite{Liao24} can't be applied here. Therefore we need to establish their counterparts in the current situation so that the general framework still functions well. The aforementioned crucial lemmas and a priori estimates are mainly concerning with the following HJBI equation for controlled particle systems 
\begin{align}\label{HJB-N}
  \left\{\begin{aligned}
    &\partial_tV_N(t,x)+\frac12\sum_{i=1}^N{\rm tr}\big[\sigma^\top\sigma D^2_{x_ix_i}V_N(t,x)\big]+H_N\big(x,\nabla_x V_N(t,x)\big)=0,\\
    &V_N(T,x)=U_N(x),\quad(t,x)\in[0,T)\times\mathbb R^{dN},
  \end{aligned}\right.
\end{align}
where for $x=(x_1,\ldots,x_N)\in\mathbb R^{dN}$, $p=(p_1,\ldots,p_N)\in\mathbb R^{dN}$,
\begin{align}\label{HJB-N-1}
 H_N(x,p):=\mathcal H\bigg(\frac1N\sum_{i=1}^N\delta_{(x_i,Np_i)}\bigg),\quad U_N(x):=U\bigg(\frac1N\sum_{i=1}^N\delta_{x_i}\bigg).
\end{align}
We note that the functions defined in \eqref{HJB-N-1} are special due to the following property of Wasserstein derivatives
\begin{align*}
    &D_{x_i}H_N(x,p)=\frac1N\partial_\mu\mathcal H\bigg(\frac1N\sum_{i=1}^N\delta_{(x_i,Np_i)},x_i,Np_i\bigg)^{(x)},\quad D_{x_i}U_N(x)=\frac1N\partial_\mu U\bigg(\frac1N\sum_{i=1}^N\delta_{x_i},x_i\bigg),\\
    &x=(x_1,\ldots,x_N)\in(\mathbb R^d)^N,\ p=(p_1,\ldots,p_N)\in(\mathbb R^d)^N.
\end{align*}
To see how \eqref{HJB-N} is connected to \eqref{HJB}, we take the following result from \cite{Liao24} which describes the propagation of chaos under appropriate conditions. Its proof can be directly adapted to the current case and is hence omitted.
\begin{proposition}\label{prop-propagation-1}
  Suppose that
  \begin{enumerate}
  \item Assumption \ref{assumption-displacement-convex} holds;
  \item $V_N\in C^{1,2}\big([0,T)\times(\mathbb R^d)^N\big)\cap C\big([0,T]\times(\mathbb R^d)^N\big)$ is the classical solution to \eqref{HJB-N} with bounded derivatives;
  \item $V$ is the classical solution to \eqref{HJB}, $\partial_tV(\cdot)\in\mathcal C\big([0,T]\times\mathcal P_2(\mathbb R^d)\big)$, $V(t,\cdot)\in\mathcal C^2\big(\mathcal P_2(\mathbb R^d)\big)$ with jointly continuous and bounded derivatives.
  \end{enumerate}
 Then there exists a constant $C$ that depends only on $|\partial^2_{\mu\mu}V|_\infty$ and $T$ such that
  \begin{align}\label{propagation-1}
    \big|V(t,\mu_x)-V_N(t,x_1,\ldots,x_N)\big|\leq\frac CN,\quad\mu_x:=\frac1N\sum_{i=1}^N\delta_{x_i}.
  \end{align}
\end{proposition}
In the spirit of the particle system approach, the aforementioned difficulties concerning the well-posedness of \eqref{HJB} can be addressed by studying \eqref{HJB-N}. To continue, we first prepare key lemmas and a priori estimates on \eqref{HJB-N}, which are significantly different from those in \cite{Liao24}. Then we will prove Theorem \ref{existence-of-decoupling-field}.

As has been brought up, the first difficulty we are met with is that the Hamiltonian are no longer displacement concave in the momentum, thus the representation of the Hamiltonian through Legendre transform is no longer available for \eqref{HJB-N}. As a result, we need to show the following Lemma \ref{differentiability} and Proposition \ref{uniform-1-1} from \cite{Liao24} without the mentioned representation of the Hamiltonian.
\begin{lemma}\label{differentiability}
 Suppose that Assumption \ref{assumption-displacement-convex} holds and $\sigma>0$. Then \eqref{HJB-N} admits a unique classical solution satisfying $V_N$, $(D_{x_i}V_N)_{1\leq i\leq N}$, $(D^2_{x_ix_j}V_N)_{1\leq i,j\leq N}\in C^{1,2+\gamma}\big([0,T)\times\mathbb R^{dN}\big)\cap C\big([0,T]\times\mathbb R^{dN}\big)$ for some $\gamma>0$ with linear growth and bounded gradients.
\end{lemma}
\begin{proof}
 We begin with the existence. Define the PDE with a cutoff function 
 \begin{align}\label{cut-HJB-N}
  \left\{\begin{aligned}
    \partial_tV^R_N(t,x)+\frac12\sum_{i=1}^N{\rm tr}\big[\sigma^\top\sigma D^2_{x_ix_i}V^R_N(t,x)\big]+H^R_N\big(x,\nabla_x V^R_N(t,x)\big)=0,\\
    V^R_N(T,x)=U^R_N(x),\quad(t,x)\in[0,T)\times(\mathbb R^d)^N.
  \end{aligned}\right.
\end{align}
Here
\begin{align*}
 H^R_N(x,p)=H_N\big(\Phi_R(x),p\big),\ U^R_N(x)=U_N\big(\Phi_R(x)\big),\\ \Phi_R(x)=\big(\rho_R(x_1),\ldots,\rho_R(x_N)\big),
\quad x=(x_1,\ldots,x_N)\in(\mathbb R^d)^N,
\end{align*}
where $\rho_R(x)$ is a smooth cutoff function for $x\in\mathbb R^d$ that vanishes when $|x|>R$. In view of Theorem 5.8.1 in \cite{Ladyzhenskaya1968}, \eqref{cut-HJB-N} admits a unique classical solution. It suffices to show that the aforementioned solution is locally uniformly bounded in $R$. Towards that end, define
 \begin{align*}
  X^{t,x,i}_s=x+\sigma W^i_s,\quad Y^{t,x}_s=V^R_N(s,X^{t,x}_s).
 \end{align*}
 Then $Y^{t,x}_t=V^R_N(t,x)$ and satisfies BSDE
 \begin{align}\label{cut-BSDE}
 \left\{\begin{aligned}
 dY^{t,x}_s&=-H^R_N\big(X^{t,x}_s,\sigma^{-1}Z^{t,x}_s\big)ds+Z^{t,x}_sdW_s,\ s\in[t,T],\\
 Y^{t,x}_T&=U^R\big(\mu_{X^{t,x}_T}\big).
 \end{aligned}\right.
 \end{align}
 The data satisfy
 \begin{align*}
  \big|H^R_N\big(X^{t,x}_s,\sigma^{-1}Z^{t,x}_s\big)\big|\leq C\big(1+|x|+|W_s|+|Z^{t,x}_s|\big),\ \big|U^R\big(\mu_{X^{t,x}_T}\big)\big|\leq C\big(1+|x|+|W_T|\big),
 \end{align*}
where the constant $C$ depends only on $H_N$ and is independent of $R$. Consider
\begin{align*}
 \left\{\begin{aligned}
 dY^{t,x,+}_s&=-C\big(1+|x|+|W_s|+|Z^{t,x,+}_s|\big)ds+Z^{t,x,+}_sdW_s,\ s\in[t,T],\\
 Y^{t,x,+}_T&=C\big(1+|x|+|W_T|\big),
 \end{aligned}\right.
 \end{align*}
 and
 \begin{align*}
 \left\{\begin{aligned}
 dY^{t,x,-}_s&=C\big(1+|x|+|W_s|+|Z^{t,x,-}_s|\big)ds+Z^{t,x,-}_sdW_s,\ s\in[t,T],\\
 Y^{t,x,-}_T&=-C\big(1+|x|+|W_T|\big).
 \end{aligned}\right.
 \end{align*}
 According to standard comparison theorems on BSDE,
 \begin{align*}
  Y^{t,x,-}_t\leq Y^{t,x}_t\leq Y^{t,x,+}_t.
 \end{align*}
Hence $V^R_N(t,x)$ is locally bounded uniform in $R$. In order to show such uniform boundedness for higher order derivatives of $V^R_N$, we first refer to Theorem 5.3.1 in \cite{Ladyzhenskaya1968} and get that $\nabla_xV^R_N$ is locally H$\ddot{\rm o}$lder continuous where the local H$\ddot{\rm o}$lder constant is independent of $R$, which further implies the local H$\ddot{\rm o}$lder continuity of $\nabla^2_xV^R_N$. Let us turn to the uniqueness. For any classical solution to \eqref{HJB-N} with bounded derivatives, define
\begin{align*}
 Y^{t,x}_s=V_N\big(s,X^{t,x}_s\big),\quad s\in[t,T].
\end{align*}
Then $(Y^{t,x}_s,Z^{t,x}_s)_{s\in[t,T]}$ uniquely solves the following BSDE with stochastic factor $(X^{t,x}_s)_{s\in[t,T]}$
 \begin{align*}
 \left\{\begin{aligned}
 dY^{t,x}_s&=-H_N\big(X^{t,x}_s,\sigma^{-1}Z^{t,x}_s\big)ds+Z^{t,x}_sdW_s,\ s\in[t,T],\\
 Y^{t,x}_T&=U^R\big(\mu_{X^{t,x}_T}\big).
 \end{aligned}\right.
 \end{align*}
Hence the uniqueness.
\end{proof}
The a priori estimate below has a counterpart from \cite{Liao24} where the original proof relies on the existence of optimal control and the corresponding verification results. Therefore here we need to show it in the current circumstance.
\begin{proposition}\label{uniform-1-1}
  Suppose that Assumption \ref{assumption-displacement-convex} holds and $\sigma>0$. Then there exists a constant $C$ depending only on $|\partial_{\tilde\mu}\mathcal H|_\infty+|\partial_\mu U|_\infty$ (independent of $N$) such that
  \begin{align}\label{uniform-1}
    \big|ND_{x_i}V_N(t,x)\big|\leq C,\quad i=1,\ldots,N,\ (t,x)\in[0,T]\times\mathbb R^N.
  \end{align}
  \end{proposition}
  \begin{proof}
    We lose nothing by only considering $t=0$. As mentioned in Lemma \ref{differentiability}, equation \eqref{HJB-N} has a classical solution with bounded derivatives. Hence we may consider $(X_t)_{t\in[0,T]}=\big(X^1_t,\ldots,X^N_t\big)^\top_{t\in[0,T]}$, satisfying
 \begin{align*}
 \left\{\begin{aligned}
 dX^i_t=D_{p_i}H_N\big(X_t,\nabla_x V_N(t,X_t\big)dt+\sigma dW^i_t,\quad t\in[0,T],\ i=1,\ldots,N,\\
 X_0=x=(x_1,\ldots,x_N)^\top,\ (x_1,\ldots,x_N)\in(\mathbb R^d)^N.
 \end{aligned}\right.
 \end{align*}
Since $D_{p_i}H_N(x,p)$ is bounded function for $(x,p)\in(\mathbb R^d)^N\times(\mathbb R^d)^N$, the above SDE is well posed in weak sense. After taking $D_{x_i}$ in \eqref{cut-HJB-N}, we obtain the following from Feynman-Kac representation
\begin{align*}
 D_{x_i}V_N\big(0,x\big)=\mathbb E\bigg[D_{x_i}U_N(\mu_{X_T})+\int_0^TD_{x_i}H_N\big(X_t,\nabla_x V_N(t,X_t)\big)dt\bigg].
\end{align*}
In view of \eqref{HJB-N-1},
\begin{align*}
 |D_{x_i}U_N(x)|+|D_{x_i}H_N(x,p)|\leq \frac CN,\quad(x,p)\in(\mathbb R^d)^N\times(\mathbb R^d)^N,\ 1\leq i\leq N.
\end{align*}
The above yields
\begin{align*}
 \big|ND_{x_i}V_N(t,x)\big|\leq C,\quad i=1,\ldots,N,\ (t,x)\in[0,T]\times\mathbb R^N,
\end{align*}
where the constant depends only on $|\partial_{\tilde\mu}\mathcal H|_\infty+|\partial_\mu U|_\infty$. 
\end{proof}
Proposition \ref{prop-propagation-1} and Proposition \ref{uniform-1-1} imply the following.
\begin{corollary}
  Suppose that
\begin{enumerate}
  \item Assumption \ref{assumption-displacement-convex} holds and $\sigma>0$;
  \item $V$ is the classical solution to \eqref{HJB}. For some $\tilde\delta>0$, $\partial_tV(t,\cdot)\in\mathcal C\big([T-\tilde\delta,T]\times\mathcal P_2(\mathbb R^d)\big)$, $V(t,\cdot)\in\mathcal C^2\big(\mathcal P_2(\mathbb R^d)\big)$ with jointly continuous and bounded derivatives on $t\in[T-\tilde\delta,T]$.
  \end{enumerate}
  Then it holds for $\mu,\ \nu\in\mathcal P_2(\mathbb R^d)$ that
  \begin{align*}
    |V(t,\mu)-V(t,\nu)|\leq C\mathcal W_1(\mu,\nu),\quad t\in[T-\tilde\delta,T],
  \end{align*}
  where the constant $C$ is the same as in \eqref{uniform-1} and thus depends only on $|\partial_{\tilde\mu}\mathcal H|_\infty+|\partial_\mu U|_\infty$.
\end{corollary}

The next lemma also has a counterpart in \cite{Liao24} whose proof relies on the representation of Hamiltonians through Legendre transform. We adapt the proof to current setting. Before going further, for $1\leq k,l\leq N$, $1\leq m,n\leq d$ we take $D^2_{x_kx_l}$ in \eqref{HJB-N} and obtain the system of equations on $\big(V^{ij}_N\big)_{1\leq i,j\leq N}$ where $V^{kl}_N:=D^2_{x_kx_l}V_N\in\mathbb R^{d\times d}$. For $k\neq l$:
{\small\begin{align}\label{Fey-Kac}
  \left\{\begin{aligned}
    &\qquad\partial_t V^{kl}_N+\frac12\sum_{i=1}^ND^2_{x_ix_i}{\rm tr}\big[\sigma^\top\sigma V^{kl}_N\big]+\sum_{i=1}^N\partial_{\tilde\mu}\mathcal H^{(p)}(\tilde\mu_t,x_i,NV^i_N)D_{x_i}V^{kl}_N+D_x\partial_{\tilde\mu}\mathcal H^{(p)}(\tilde\mu_t,x_l,NV^l_N)V^{lk}_N\\
    &\quad+\frac1N\sum_{i=1}^NV^{ki}_N\partial^2_{\tilde\mu\tilde\mu}\mathcal H^{(p)(x)}(\tilde\mu_t,x_i,NV^i_N,x_l,NV^l_N)+\frac1N\sum_{i=1}^N\partial^2_{\tilde\mu\tilde\mu}\mathcal H^{(x)(p)}(\tilde\mu_t,x_k,NV^k_N,x_i,NV^i_N)V^{il}_N\\
    &\quad+\sum_{i,j=1}^NV^{ki}_N\partial^2_{\tilde\mu\tilde\mu}\mathcal H^{(p)(p)}(\tilde\mu_t,x_i,NV^i_N,x_j,NV^j_N)V^{jl}_N+N\sum_{i=1}^NV^{ki}_N D_p\partial_{\tilde\mu}\mathcal H^{(p)}(\tilde\mu_t,x_i,NV^i_N)V^{il}_N\\
    &\quad+\frac1{N^2}\partial^2_{\tilde\mu\tilde\mu}\mathcal H^{(x)(x)}(\tilde\mu_t,x_k,NV^k_N,x_l,NV^l_N)+V^{lk}_N D_p\partial_{\tilde\mu}\mathcal H^{(x)}(\tilde\mu_t,x_k,NV^k_N)=0,\\
    &\quad V^{kl}_N(T,x)=\frac1{N^2}\partial^2_{\mu\mu}U\bigg(\frac1N\sum_{i=1}^N\delta_{x_i},x_k,x_l\bigg),\quad t\in[0,T],\ x=(x_1,\ldots,x_N)\in(\mathbb R^d)^N.
  \end{aligned}\right.
\end{align}}
For $k=l$:
{\small\begin{align}\label{Fey-Kac-1}
  \left\{\begin{aligned}
&\qquad\partial_tV^{kk}_N+\frac12\sum_{i=1}^ND^2_{x_ix_i}{\rm tr}\big[\sigma^\top\sigma V^{kk}_N\big]+\sum_{i=1}^N\partial_{\tilde\mu}\mathcal H^{(p)}(\tilde\mu_t,x_i,NV^i_N)D_{x_i}V^{kk}_N+D_p\partial_{\tilde\mu}\mathcal H^{(x)}(\tilde\mu_t,x_k,NV^k_N)V^{kk}_N\\
    &\quad+D_x\partial_{\tilde\mu}\mathcal H^{(p)}(\tilde\mu_t,x_k,NV^k_N)V^{kk}_N+\frac1N\sum_{i=1}^N\partial^2_{\tilde\mu\tilde\mu}\mathcal H^{(p)(x)}(\tilde\mu_t,x_i,NV^i_N,x_k,NV^k_N)V^{ik}_N\\
    &\quad+\frac1N\sum_{i=1}^N\partial^2_{\tilde\mu\tilde\mu}\mathcal H^{(x)(p)}(\tilde\mu_t,x_k,NV^k_N,x_j,NV^j_N)V^{jk}_N\\
    &\quad+\sum_{i,j=1}^NV^{ki}_N\partial^2_{\tilde\mu\tilde\mu}\mathcal H^{(p)(p)}(\tilde\mu_t,x_i,NV^i_N,x_j,NV^j_N)V^{jk}_N+N\sum_{i=1}^NV^{ki}_ND_p\partial_{\tilde\mu}\mathcal H^{(p)}(\tilde\mu_t,x_i,NV^i_N)V^{ik}_N\\
    &\quad+\frac1{N^2}\partial^2_{\tilde\mu\tilde\mu}\mathcal H^{(x)(x)}(\tilde\mu_t,x_k,NV^k_N,x_k,NV^k_N)+\frac1ND_x\partial_{\tilde\mu}\mathcal H^{(x)}(\tilde\mu_t,x_k,NV^k_N)=0,\\
    &\quad V^{kk}_N(T,x)=\frac1{N^2}\partial^2_{\mu\mu}U\bigg(\frac1N\sum_{i=1}^N\delta_{x_i},x_k,x_k\bigg)+\frac1ND_x\partial_{\mu}U\bigg(\frac1N\sum_{i=1}^N\delta_{x_i},x_k\bigg),\\
   &\quad t\in[0,T],\ x=(x_1,\ldots,x_N)\in(\mathbb R^d)^N.
  \end{aligned}\right.
\end{align}}

\begin{lemma}\label{continuation-1}
 Suppose that 
 \begin{enumerate}
  \item Assumption \ref{assumption-displacement-convex} holds and $\sigma>0$;
  \item The system \eqref{Fey-Kac}$\sim$\eqref{Fey-Kac-1} admits bounded solutions on $t\in[T_0,T]$, in particular,
 \begin{align}\label{continuation-1-assumption}
  |V^{ij}_N(T_0,x)|\leq\frac CN,\quad x\in(\mathbb R^d)^N.
 \end{align}
 \end{enumerate}
Then there exist constants $\tilde\delta_0>0$ and $K$ depending only on above $C$ such that
 \begin{align}\label{continuation-1-result}
  |V^{ij}_N(t,x)|\leq\frac KN,\quad (t,x)\in[T_0-\tilde\delta_0,T_0]\times(\mathbb R^d)^N.
 \end{align}
\end{lemma}
\begin{proof}
Take $\partial^2_{x_ix_j}$ in \eqref{cut-HJB-N} and obtain \eqref{Fey-Kac}$\sim$\eqref{Fey-Kac-1} with $H_N$ replaced by $H^R_N$. Then, in view of Lemma \ref{differentiability}, for $1\leq i,j\leq N$ and $x\in(\mathbb R^d)^N$, the following SDE has a weak solution
 \begin{align*}
 \left\{\begin{aligned}&dX^{R,i}_t=D_{p_i}H^R_N\big(X^R_t,\nabla_x V^R_N(t,X^R_t)\big)dt+\sigma dW^i_t,\ t\in[T_0-\tilde\delta_0,T_0],\\
 &X^R_{T-\tilde\delta_0}=x=(x_1,\ldots,x_N)^\top,\end{aligned}\right.
\end{align*} 
where $\tilde\delta_0$ is a positive constant to be determined. We may also consider
\begin{align*}
 Y^{R,ij}_t=D^2_{x_ix_j}V^R_N(t,X^R_t),\quad t\in[T_0-\tilde\delta_0,T_0].
\end{align*}
Then $\big(Y^{R,ij}_t\big)_{t\in[T-\tilde\delta_0,T],1\leq i,j\leq N}$ satisfies
\begin{align}\label{Fey-Kac-3-var}
    Y^R_t&=\mathbb E\bigg[\int_t^{T_0}\bigg(H^{R,xx}_N(s)+A^R_N(s)Y^R_s+Y^R_sA^R_N(s)^\top+Y^R_sH^{R,pp}_N(s)Y^R_s\bigg)ds\notag\\
   &\qquad\quad+V^{R,xx}_N(T_0)\bigg|\mathcal F^{\mathbf W}_t\bigg],\quad t\in[T_0-\tilde\delta_0,T_0],
  \end{align}
  where ${\mathbf W}_t=(W^1_t,\ldots,W^N_t)^\top,\ t\in[T_0-\tilde\delta_0,T_0]$,
 \begin{align*}
  &V^{R,xx}_N(T_0)=\nabla^2_{xx}V^R_N\big(T_0,X^R_N(T_0)\big),\quad H^{R,xx}_N(s)=\nabla^2_{xx}H^R_N\big(X^R_N(s),\nabla_x V^R_N(s,X^R_N(s))\big),\\
  &A^R_N(s)=\nabla^2_{xp}H^R_N\big(X^R_N(s),\nabla_x V^R_N(s,X^R_N(s))\big),\quad H^{R,pp}_N(s)=\nabla^2_{pp}H^R_N\big(X^R_N(s),\nabla_x V^R_N(s,X^R_N(s))\big).
 \end{align*}
In view of Assumption \ref{assumption-displacement-convex} and the definition, each entry of $V^{R,xx}_N(T_0),\ H^{R,xx}_N(s),\ A^R_N(s),\ H^{R,pp}_N(s)$ is bounded by a constant independent of $R$. Using a similar contraction method to that in Proposition 3.10 of \cite{Huafu24}, we may show the existence of positive constants $\tilde\delta_0,\ K$ such that \eqref{continuation-1-result} holds true.
\end{proof}
The next a priori estimate is crucial to our analysis. The corresponding method and assumption here are notably different from their counterparts in \cite{Liao24}.
\begin{lemma}\label{prop-eigen}
  Suppose that
  \begin{enumerate}
   \item Assumption \ref{assumption-displacement-convex} holds and $\sigma>0$;
   \item For some $\tilde\delta>0$, the system \eqref{Fey-Kac}$\sim$\eqref{Fey-Kac-1} admits bounded classical solutions $V^{ij}_N\in C^{1,2}([0,T]\times(\mathbb R^d)^N)$, $1\leq i,j\leq N$;
 \end{enumerate}
  Then there exists a constant $C$ depending only on
 \begin{align}\label{C-depend}
  |\partial^2_{\tilde\mu\tilde\mu}\mathcal H|_\infty+|\partial_x\partial_{\tilde\mu}\mathcal H^{(x)}|_\infty+|\partial_p\partial_{\tilde\mu}\mathcal H^{(p)}|_\infty+|\partial^2_{\mu\mu}U|_\infty+|\partial_x\partial_\mu U|_\infty,
 \end{align}
independent of $\tilde\delta$ such that
  \begin{align}\label{eigen}
    -\frac CN\sum_{i=1}^N\xi^2_i\leq\sum_{i,j=1}^N\xi^\top_iV^{ij}_N(t,x)\xi_j\leq\frac CN\sum_{i=1}^N\xi^2_i,\ \xi\in\mathbb R^N,\ (t,x)\in[T-\tilde\delta,T]\times(\mathbb R^d)^N.
  \end{align}
\end{lemma}
Before proving Theorem \ref{prop-eigen}, we need a comparison result on the stochastic Riccati equation. The symmetric Riccati equations have a striking property that they preserves their ordering. We now extend such property to stochastic Riccati equations. For $i=1,2$, let us consider the following two stochastic Riccati equations
\begin{align}\label{comparison}
\left\{\begin{aligned}    dY^i_t&=\big(Q^i_t+A^i_tY^i_t+Y^i_t(A^i_t)^\top-Y^i_tH^i_tY^i_t\big)dt+Z^i_td\breve W_t,\quad t\in[0,T],\\
    Y^i_T&=U^i.\end{aligned}\right.
\end{align}
And denote by
\begin{align*}
 JH^i_s=\left(\begin{matrix}Q^i_s&(A^i_s)^\top\\
 A^i_s&-H^i_s
 \end{matrix}\right),\quad s\in[0,T],\ i=1,2.
\end{align*}
With the above notations we present the following comparison result. Its deterministic counterpart can be found in \cite{Abou-Kandil2003}.
\begin{lemma}\label{lem-comparison}
 Suppose that for $i=1,2$, the BSDE in \eqref{comparison} admit solutions $(Y^i,Z^i)$ with bounded $Y^i$, and
\begin{align}\label{assumption-comparison}
 U^1\geq U^2,\ JH^1_s\leq JH^2_s,\quad s\in[0,T].
\end{align}
Then $Y^1_s\geq Y^2_s$, $s\in[0,T]$.
\end{lemma}
\begin{proof}
 Denote by
\begin{align*}
  \tilde Y_s:=Y^1_s-Y^2_s,\quad s\in[0,T].
\end{align*}
According to \eqref{comparison}, direct calculation yields
\begin{align*}
\left\{\begin{aligned}
d\tilde Y_s&=\big(\tilde A_s\tilde Y_s+\tilde Y_s(\tilde A_s)^\top+\tilde Q_s\big)ds+\tilde Z_sd\breve W_s,\quad s\in[0,T],\\
\tilde Y_T&=U^1-U^2,
\end{aligned}\right.
\end{align*}
where for $s\in[0,T]$,
\begin{align*}
 \tilde A_s&=-\frac12\tilde Y_sH^1_s-Y^2_sH^1_s+A^1_s,\\
 \tilde Q_s&=(A^1_s-A^2_s)Y^2_s+Y^2_s(A^1_s-A^2_s)+Y^2_s(H^2_s-H^1_s)Y^2_s+Q^1_s-Q^2_s.
\end{align*}
In view of \eqref{assumption-comparison}, $\tilde Y_T\geq0$,
\begin{align*}
 \tilde Q_s=\left(\begin{matrix}Y^2_s&I_d\end{matrix}\right)^\top(JH^1_s-JH^2_s)\left(\begin{matrix}Y^2_s\\I_d\end{matrix}\right)\leq0\ \text{for}\ s\in[0,T].
\end{align*}
Define the matrix valued process satisfying
\begin{align*}
 d\Phi_s=-\tilde A_s\Phi_sds,\quad s\in[0,T],\quad \Phi_0=I.
\end{align*}
Then
\begin{align*}
 d\Phi_s\tilde Y_s\Phi^\top_s=\Phi_s\tilde Q_s\Phi^\top_sds+\Phi_s\tilde Z_s\Phi^\top_sd\breve W_s,\quad s\in[0,T].
\end{align*}
The above and $\tilde Y_T\geq0$, $\tilde Q_s\leq0$ implies that $\tilde Y_s\geq0$ for $s\in[0,T]$.
\end{proof}
Given the comparison result in Lemma \ref{lem-comparison}, we turn to the proof of Lemma \ref{prop-eigen}.
\begin{proof}[Proof of Lemma \ref{prop-eigen}]
 Consider
\begin{align*}
\left\{\begin{aligned}
&dX^i_t=D_{p_i}H_N\big(X_t,\nabla_x V_N(t,X_t)\big)dt+\sigma dW^i_t,\ t\in[0,T],\\
&X_0=x=(x_1,\ldots,x_N)^\top,
\end{aligned}\right.
\end{align*} 
and for $s\in[0,T]$, $Y_s=(Y^{ij}_s)_{1\leq i,j\leq N}$ where
\begin{align}\label{prop-eigen-0}
 Y^{ij}_s=D^2_{x_ix_j}V_N(s,X_s)\in\mathbb R^{d\times d}.
\end{align}
According to \eqref{Fey-Kac}$\sim$\eqref{Fey-Kac-1}, we obtain
\begin{align*}
 \left\{\begin{aligned}
 dY_s&=-\big(H^{xx}_N(s)+A_N(s)Y_s+Y_sA_N(s)^\top+Y_sH^{pp}_N(s)Y_s\big)ds+\sum_{i=1}^NZ^{(i)}_sdW^i_s,\quad s\in[0,T],\\
 Y_T&=V^{xx}_N(T).
 \end{aligned}\right.
\end{align*}
Here for $s\in[0,T]$, we use $H^{xx}_N(s),\ A_N(s),\ H^{pp}_N(s)$ and $V^{xx}_N(T)\in\mathbb R^{dN\times dN}$ to denote
\begin{align*}
 &\big(H^{xx}_N(s)\big)_{ij}:=D^2_{x_ix_j}H_N\big(X_s,\nabla_x V_N(t,X_s)\big),\ \big(A_N(s)\big)_{ij}:=D^2_{x_ip_j}H_N\big(X_s,\nabla_x V_N(t,X_s)\big),\notag\\
 &V^{xx}_N(T)=D^2_{x_ix_j}V_N(T,X_T),\ 1\leq i,j\leq N.
\end{align*}
Let us consider
\begin{align*}
 \tilde\mu^N_t=\frac1N\sum_{i=1}^N\delta_{(X^i_t,\xi^i_t)},
\end{align*}
where
\begin{align*}
 \xi^i_t:=D_{x_i}V_N(t,X_t),\quad1\leq i\leq N,\ t\in[0,T].
\end{align*}
For arbitrary $(\alpha_i)_{1\leq i\leq N},\ (\beta_i)_{1\leq i\leq N}\in(\mathbb R^d)^N$, denote by
\begin{align*}
 \alpha_i=(\hat\alpha_i,\breve\alpha_i)^\top,\ \beta_i=(\hat\beta_i,\breve\beta_i)^\top,\quad1\leq i\leq N,\ t\in[0,T],
\end{align*}
where $\hat\alpha_i,\ \hat\beta_i\in\mathbb R^{d_1}$, $\breve\alpha_i,\ \breve\beta_i\in\mathbb R^{d_2}$. It is easy to find smooth test mapping $\phi=(\phi_1,\phi_2)^\top,\ \psi=(\psi_1,\psi_2)^\top$ where
\begin{align*}
 \phi_i,\ \psi_i:\quad\mathbb R^d\times\mathbb R^d\mapsto\mathbb R^{d_i},\quad i=1,2,
\end{align*}
such that
\begin{align*}
 \phi(X^i_t,\xi^i_t)=\alpha_i,\ \psi(X^i_t,\xi^i_t)=\beta_i,\ 1\leq i\leq N.
 \end{align*}
Plug $\tilde\mu^N_t,\ \phi,\ \psi$ into \eqref{general-displacement}. According to \eqref{HJB-N-1}, since $(\alpha_i)_{1\leq i\leq N},\ (\beta_i)_{1\leq i\leq N}\in(\mathbb R^d)^N$ are arbitrary, we obtain
\begin{align}\label{prop-eigen-1}
 \left\{\begin{aligned}&\left(\begin{matrix}-C\cdot D_{2,N}&D_{3,N}(s)\\
 D_{3,N}(s)^\top&-C\cdot D_{1,N}
 \end{matrix}\right)\leq\left(\begin{matrix}-H^{xx}_N(s)&-A_N(s)\\
 -A_N(s)^\top&-H^{pp}_N(s)
 \end{matrix}\right)\leq\left(\begin{matrix}C\cdot D_{1,N}&D_{3,N}(s)\\
 D_{3,N}(s)^\top&C\cdot D_{2,N}
 \end{matrix}\right),\\
 &-CD_{1,N}\leq V^{xx}_N(T)\leq CD_{2,N},\end{aligned}\right.
\end{align}
where
\begin{align}\label{D1D2}
\left\{\begin{aligned}
&D_{2,N}=N^{-1}{\rm diag}(D_2,\ldots,D_2)\in\mathbb R^{Nd\times Nd},\quad D_{1,N}=N^{-1}{\rm diag}(D_1,\ldots,D_1)\in\mathbb R^{Nd\times Nd},\\
&D_1=\left(\begin{matrix}0_{d_1}&0\\
 0&I_{d_2}
 \end{matrix}\right)\in\mathbb R^{d\times d},\quad
 D_2=\left(\begin{matrix}I_{d_1}&0\\
 0&0_{d_2}
 \end{matrix}\right)\in\mathbb R^{d\times d},\\
 &D_{3,N}(s)=\big((D_{3,N}(s))_{ij}\big)_{1\leq i,j\leq N}\in\mathbb R^{Nd\times Nd},\quad \big(D_{3,N}(s)\big)_{ij}=\left(\begin{matrix}-a^{ij}_{11}(s)&0\\
 0&-a^{ij}_{22}(s)
 \end{matrix}\right)\in\mathbb R^{d\times d},\\
 &\quad s\in[0,T],\ 1\leq i,j\leq N,
\end{aligned}\right.
\end{align}
where $a^{ij}_{11}(s)\in\mathbb R^{d_1\times d_1},\ a^{ij}_{22}(s)\in\mathbb R^{d_2\times d_2}$, denote the entrances of $\big(A_N(s)\big)_{ij}$ as follows:
\begin{align*}
 \big(A_N(s)\big)_{ij}=\left(\begin{matrix}a^{ij}_{11}(s)&a^{ij}_{12}(s)\\
 a^{ij}_{21}(s)&a^{ij}_{22}(s)
 \end{matrix}\right)\in\mathbb R^{d\times d}.
\end{align*}
For $s\in[0,T]$ and $k=1,2$, let $D^{(k)}_{3,N}(s)=\big((D^{(k)}_{3,N}(s))_{ij}\big)_{1\leq i,j\leq N}\in\mathbb R^{Nd_k\times Nd_k}$ and
\begin{align*}
 \big(D^{(1)}_{3,N}(s)\big)_{ij}=-a^{ij}_{11}(s)\in\mathbb R^{d_1\times d_1},\ \big(D^{(2)}_{3,N}(s)\big)_{ij}=-a^{ij}_{22}(s)\in\mathbb R^{d_2\times d_2}\quad1\leq i,j\leq N.
\end{align*}
For $k=1,2$, let $\big(Y^{+,(k)}_s\big)_{s\in[0,T]}$ solve
\begin{align}\label{sup-solution}
 \left\{\begin{aligned}dY^{+,(1)}_s&=-\frac CNI_{Nd_1}+D^{(1)}_{3,N}(s)Y^{+,(1)}_s+Y^{+,(1)}_sD^{(1)}_{3,N}(s)^\top+\sum_{i=1}^NZ^{+,(1),i}_sdW^i_s,\quad s\in[0,T),\\
 dY^{+,(2)}_s&=D^{(2)}_{3,N}(s)Y^{+,(2)}_s+Y^{+,(2)}_sD^{(2)}_{3,N}(s)^\top-\frac CNY^{+,(2)}_sI_{Nd_2}Y^{+,(2)}_s+\sum_{i=1}^NZ^{+,(2),i}_sdW^i_s,\\
 Y^{+,(1)}_T&=\frac CNI_{Nd_1},\quad Y^{+,(2)}_T=0.\end{aligned}\right.
\end{align}
And further use $\big((Y^+_s)_{ij}\big)_{s\in[0,T],1\leq i,j\leq N}$ to denote
\begin{align}\label{sup-solution-1}
 \big(Y^+_s\big)_{ij}:=\left(\begin{matrix}\big(Y^{+,(1)}_s\big)_{ij}&0\\
 0&\big(Y^{+,(2)}_s\big)_{ij}
 \end{matrix}\right)\in\mathbb R^{d\times d},\quad1\leq i,j\leq N.
\end{align}
Simple algebra yields
\begin{align*}
 \left\{\begin{aligned}dY^+_s&=-CD_{2,N}+D_{3,N}(s)Y^+_s+Y^+_sD_{3,N}(s)^\top-CY^+_sD_{1,N}Y^+_s+\sum_{i=1}^NZ^{+,i}_sdW^i_s,\quad s\in[0,T),\\
 Y^+_T&=CD_{2,N}.\end{aligned}\right.
\end{align*}
According to Lemma \ref{lem-comparison}, $Y_s\leq Y^+_s,\ s\in[0,T).$ It is straightforward to see that $Y^{+,(2)}_s=0,\ s\in[0,T).$ Moreover, according to its definition, it is easy to verify that
\begin{align*}
 |(D_{1,N})_{ij}|+|(D_{2,N})_{ij}|\leq CN^{-1}(\delta_{ij}+N^{-1}),\quad|(D_{3,N}(s))_{ij}|\leq CN^{-1},\quad s\in[0,T),\ 1\leq i,j\leq N.
\end{align*}
Then we may argue in a similar way to \cite{Liao24} and show
\begin{align*}
 \big|(Y^+_s)_{ij}\big|\leq\frac CN,\quad s\in[0,T),\ 1\leq i,j\leq N.
\end{align*}
Therefore $ Y_s\leq\frac CNI_{Nd},\ s\in[0,T]$. Considering
\begin{align*}
 \left\{\begin{aligned}dY^{+,(1)}_s&=-\frac CNI_{Nd_1}+D^{(1)}_{3,N}(s)Y^{+,(1)}_s+Y^{+,(1)}_sD^{(1)}_{3,N}(s)^\top+\sum_{i=1}^NZ^{+,(1),i}_sdW^i_s,\quad s\in[0,T),\\
 dY^{+,(2)}_s&=D^{(2)}_{3,N}(s)Y^{+,(2)}_s+Y^{+,(2)}_sD^{(2)}_{3,N}(s)^\top-\frac CNY^{+,(2)}_sI_{Nd_2}Y^{+,(2)}_s+\sum_{i=1}^NZ^{+,(2),i}_sdW^i_s,\\
 Y^{+,(1)}_T&=\frac CNI_{Nd_1},\quad Y^{+,(2)}_T=0,\end{aligned}\right.
\end{align*}
and constructing $\big((Y^-_s)_{ij}\big)_{s\in[0,T],1\leq i,j\leq N}$ in the same way as in \eqref{sup-solution}$\sim$\eqref{sup-solution-1}, we can show
\begin{align*}
 -\frac CNI_{Nd}\leq Y_s\leq\frac CNI_{Nd},\ s\in[0,T],
\end{align*}
which implies \eqref{eigen}.
\end{proof}
With the above preparation, we may go on to the global in time well-posedness of \eqref{HJB}. Let us now summarize the results here and sketch the proof of Theorem \ref{existence-of-decoupling-field}.
 \begin{proof}[Sketch of the proof of Theorem \ref{existence-of-decoupling-field}]Following the procedure in Lemma 4.2 and Lemma 4.4$\sim$4.6 in \cite{Liao24}, we may generate a local in time solution $V(t,\mu)$ on $(t,\mu)\in[T-\delta,T]\times\mathcal P_2(\mathbb R^d)$ to \eqref{HJB} via the decoupling field $\big(V^\mu_s=V(s,\mathcal L_{X^\mu_s})\big)_{s\in[T-\delta,T]}$ of the following mean field FBSDE
\begin{align}\label{1st-step}
  \left\{\begin{aligned}
    dX^\mu_s&=\partial_{\tilde\mu}\mathcal H^{(p)}\big(\mathbb P_{(X^\mu_s,Y^\mu_s)},X^\mu_s,Y^\mu_s\big)ds+\sigma dW_s,\ X^\mu_t=\xi,\\
    dY^\mu_s&=-\partial_{\tilde\mu}\mathcal H^{(x)}\big(\mathbb P_{(X^\mu_s,Y^\mu_s)},X^\mu_s,Y^\mu_s\big)ds+Z^\mu_sdW_s,\ Y^\mu_T=\partial_\mu U(\mathbb P_{X^\mu_T},X^\mu_T),\\
    V^\mu_s&=\mathbb E\bigg[U(\mathbb P_{X^\mu_T})+\int_s^T\bigg(\mathcal H(\mathbb P_{(X^\mu_u,Y^\mu_u)})-\partial_{\tilde\mu}\mathcal H^{(p)}\big(\mathbb P_{(X^\mu_u,Y^\mu_u)},X^\mu_u,Y^\mu_u\big)\cdot Y^\mu_u\bigg)du\bigg|\mathcal F^W_s\bigg],
\end{aligned}\right.
\end{align}
where the duration $\delta>0$ depends on $\mathcal H$, $U$, and $\mathbb P_{(X^\mu_u,Y^\mu_u)}:=\text{Law}(X^\mu_u,Y^\mu_u)$.

Next let us consider \eqref{HJB-N} and obtain the a priori estimates on $(V_N)_{N\geq1}$ uniform in $N$. More specifically, the uniform a priori estimates on the first order derivatives has been established here in Proposition \ref{uniform-1-1}, while the uniform a priori estimates on the Hessians of $(V_N)_{N\geq1}$ is Lemma \ref{prop-eigen}.

Given the a priori estimates on $(V_N)_{N\geq1}$, we use them to infer the a priori estimates on $V$. Similar to Proposition 5.3 in \cite{Liao24}, we can infer from Proposition \ref{uniform-1-1} that the local in time solution $V$ admits the a priori estimate that for $\mu,\ \nu\in\mathcal P_2(\mathbb R)$ that
  \begin{align}\label{1st-regularity-1}
    |V(t,\mu)-V(t,\nu)|\leq C\mathcal W_1(\mu,\nu),\quad t\in[T-\tilde\delta,T],
  \end{align}
  where the constant $C$ is the same as in \eqref{uniform-1} and depends only on $|\partial_{\tilde\mu}\mathcal H|_\infty+|\partial_\mu U|_\infty$. As for the second order derivatives of $V$, we adopt the idea that proves Theorem 5.12 in \cite{Liao24}, where the key ingredients Lemma 5.4, Lemma 5.5 therein now replaced by Lemma \ref{continuation-1}, Lemma \ref{prop-eigen} here. Then we obtain the existence of a constant $C$ depending only on \eqref{C-depend} such that
\begin{align}\label{Hessian-bound}
    0\leq|D_x\partial_\mu V(t,\cdot)|_\infty+|\partial^2_{\mu\mu}V(t,\cdot)|_\infty\leq C.
\end{align}
Having obtain the a priori estimates on the first and the second order derivatives, we may formally take derivatives in \eqref{HJB} and use the results in \cite{Buckdahn2017} on Feynman-Kac representation and obtain the following a priori estimates on the higher order derivatives of $V$:
\begin{align}\label{higher-order-goal}
 \sum_{\substack{0\leq k\leq 4,0\leq l_1,\ldots,l_k\leq4\\0\leq l_1+\cdots+l_k+k\leq 4}}\big|D^{l_1}_{x_1}\cdots D^{l_k}_{x_k}\partial^k_\mu V(t,\cdot)\big|_\infty<C_2,
\end{align}
where $C_2$ depends only on $\mathcal H$ and $U$. The corresponding details of deriving \eqref{higher-order-goal} are the same as in Proposition 5.16$\sim$5.17 in \cite{Liao24}. We may now inductively replace the terminal condition with $V(T-k\delta,\cdot)$, $k=1,2,\ldots$, apply Lemma \ref{continuation-1} together with the above reasoning, and obtain the analogy of \eqref{1st-regularity-1}$\sim$\eqref{higher-order-goal} on $[T-k\delta,T-(k+1)\delta],\ k\geq1.$ Then we can obtain the global well-posedness after finite repetition.
 \end{proof}

\subsection{Proofs for Theorem \ref{LQ-global-wp} and \ref{limit-V-lambda}}\label{lq-solvable}
We also impose Assumption \ref{assumption-displacement-convex} in this occasion, which has an implication presented in Lemma \ref{LQ-condition}.

\begin{proof}[Proof of Lemma \ref{LQ-condition}]
 Suppose \eqref{general-displacement}. For arbitrary $\varepsilon>0$ and $\mathbf x=(x_1,x_2)^\top,\ \mathbf p=(p_1,p_2)^\top,\ x_i,\ p_i\in\mathbb R^{d_i},\ i=1,2.$ Consider
 \begin{align*}
  \tilde\mu\big(\{(\mathbf0,\mathbf0)\}\big)=\varepsilon,\ \tilde\mu\big(\{(\mathbf0,\mathbf0)\}^c\big)=1-\varepsilon.\\
  \phi(\mathbf0)=\mathbf x,\ \psi(\mathbf0)=\mathbf p.
 \end{align*}
Plugging the above into \eqref{general-displacement}, we obtain
\begin{align*}
 -C\big(|x_2|^2+|p_1|^2\big)&\leq\mathbf x^\top Q^{(1)}_{11}\mathbf x+2(x_1,0)^\top Q^{(1)}_{12}\left(\begin{matrix}0\\p_2\end{matrix}\right)\\
  &\quad+2(0,x_2)^\top Q^{(1)}_{12}\left(\begin{matrix}p_1\\0\end{matrix}\right)-\mathbf p^\top Q^{(1)}_{22}\mathbf p+O(\varepsilon)\leq C\big(|x_1|^2+|p_2|^2\big).
\end{align*}
We then get the first inequality in \eqref{generalized-displacement-convexity-LQ} by sending $\varepsilon$ to 0. Consider instead
 \begin{align*}
  \tilde\mu\big(\{(\mathbf0,\mathbf0)\}\big)=1,\ \tilde\mu\big(\{(\mathbf0,\mathbf0)\}^c\big)=0,
 \end{align*}
 and use \eqref{general-displacement} again, we have the second inequality in \eqref{generalized-displacement-convexity-LQ}.
\end{proof}
In general, the second item in Assumption \ref{assumption-displacement-convex} is about infinite dimensional mapping $\mathcal H$. For the linear quadratic case, the previous assumption is reduced to the one on the matrix $\big(Q^{(i)}\big)_{i\in\{1,2\}}$ by Lemma \ref{LQ-condition}. Next we utilize such property to constitute the proof of Theorem \ref{LQ-global-wp}.

\begin{proof}[Proof of Theorem \ref{LQ-global-wp}]
 We begin by showing the well-posedness of \eqref{LQ-Riccati}. It is easy to see that the well-posedness of \eqref{LQ-Riccati} follows by the solvability of $\big(a_1(t),a_3(t)\big)_{t\in[0,T]}$. Towards that end, we first show the well-posedness of $\big(a_1(t)\big)_{t\in[0,T]}$ which is decoupled from other parts of the system. According to \eqref{generalized-displacement-convexity-LQ},
 \begin{align*}
  \left(\begin{matrix}-C\cdot D_2&D_3\\
 D^\top_3&-C\cdot D_1
 \end{matrix}\right)\leq\left(\begin{matrix}-Q^{(1)}_{11}&-2Q^{(1)}_{12}\\
 -2Q^{(1)}_{21}&-4Q^{(1)}_{22}
 \end{matrix}\right)\leq\left(\begin{matrix}C\cdot D_1&D_3\\
 D^\top_3&C\cdot D_2
 \end{matrix}\right),
 \end{align*}
where $D_1$, $D_2$ is the same as in \eqref{D1D2},
\begin{align*}
 D_1=\left(\begin{matrix}0_{d_1}&0\\
 0&I_{d_2}
 \end{matrix}\right),\quad
 D_2=\left(\begin{matrix}I_{d_1}&0\\
 0&0_{d_2}
 \end{matrix}\right),\quad
 D_3=\left(\begin{matrix}-2q_{11}&0\\
 0&-2q_{22}
 \end{matrix}\right)\in\mathbb R^{d\times d},
\end{align*}
where $q_{11}\in\mathbb R^{d_1\times d_1},\ q_{22}\in\mathbb R^{d_2\times d_2}$, they are used to denote the entrances of $Q^{(1)}_{12}$ as follows
\begin{align*}
 Q^{(1)}_{12}=\left(\begin{matrix}q_{11}&q_{12}\\
 q_{21}&q_{22}
 \end{matrix}\right)\in\mathbb R^{d\times d},\quad q_{ij}\in\mathbb R^{d_i\times d_j},\ 1\leq i,j\leq 2.
\end{align*}
Consider
\begin{align*}
 \left\{\begin{aligned}
  \frac d{dt}a^+_1(t)=-CD_2+D_3a^+_1(t)+a^+_1(t)^\top D^\top_3-Ca^+_1(t)^\top D_1a^+_1(t),\quad t\in[0,T),\\
  \frac d{dt}a^-_1(t)=CD_1+D_3a^-_1(t)+a^-_1(t)^\top D^\top_3+Ca^-_1(t)^\top D_2a^-_1(t),\quad t\in[0,T),\\
  a^+_1(T)=CD_1,\quad a^-_1(T)=-CD_2.
 \end{aligned}\right.
\end{align*}
Notice that $D_1,\ D_2,\ D_3$ are all of diagonal shape, it is then easy to show that $\big(a^+_1(t)\big)_{t\in[0,T]}$ and $\big(a^-_1(t)\big)_{t\in[0,T]}$ are also of diagonal shape and well-posed. In view of Lemma \ref{lem-comparison}, 
 \begin{align*}
  a^-_1(t)\leq a^{(1)}(t)\leq a^+_1(t),\quad t\in[0,T].
 \end{align*}
Hence $\big(a^{(1)}(t)\big)_{t\in[0,T]}$ is well-posed. Defining
\begin{align*}
 a^{(5)}(t)=a^{(1)}(t)+a^{(3)}(t),\quad t\in[0,T],
\end{align*}
and adding the first and third equalities in \eqref{LQ-ham}, we obtain
\begin{align*}
  \left\{\begin{aligned}&\frac d{dt}a_5(t)+Q^{(1)}_{11}+Q^{(2)}_{11}+4a_5(t)^\top\big(Q^{(1)}_{22}+Q^{(2)}_{22}\big)a_5(t)\\
  &\quad+2a_5(t)^\top\big(Q^{(1)}_{21}+Q^{(2)}_{21}\big)+2\big(Q^{(1)}_{21}+Q^{(2)}_{21}\big)a_5(t)=0,\quad t\in[0,T],\\
  &a_5(T)=Q^{(1)}_T+Q^{(2)}_T.\end{aligned}\right.
\end{align*}
Similar to the well-posedness of $\big(a^{(1)}(t)\big)_{t\in[0,T]}$, we may use the second inequality in \eqref{generalized-displacement-convexity-LQ} to show that $\big(a_5(t)\big)_{t\in[0,T]}$ is well-posed, hence the well-posedness of $\big(a_3(t)\big)_{t\in[0,T]}$ and \eqref{LQ-Riccati}. According to \eqref{LQ-solution},
\begin{align*}
 \partial_\mu V(t,\mu,x)=2a_1(t)x+2a_3(t)\int_{\mathbb R^d}\tilde x\mu(d\tilde x)+a_2(t),\\
 \partial^2_{\mu\mu}V(t,\mu,x,\tilde x)=2a_3(t),\quad D_x\partial_\mu V(t,\mu,x)=2a_1(t),\\
 (t,\mu,x,\tilde x)\in[0,T]\times\mathcal P_2(\mathbb R^d)\times\mathbb R^d\times\mathbb R^d.
\end{align*}
Using the well-posedness of \eqref{LQ-Riccati}, we may plug \eqref{LQ-solution} and the above into \eqref{HJB}, \eqref{LQ-ham} and show that $V$ solves \eqref{HJB}.
\end{proof}
Next we will analyze the a priori estimates and limit in Theorem \ref{limit-V-lambda}. More specifically, we will study the following limit
\begin{align}\label{penalize-constraint}
    V(t,\mu)=\lim_{\lambda\to+\infty}V^\lambda(t,\mu),\quad(t,\mu)\in[0,T)\times\mathcal P_2\big(\mathbb R^d\big),
\end{align}
where $(V^\lambda)_{\lambda>0}$ are the solutions to \eqref{HJB} with the Hamiltonian and penalized terminal condition $(\mathcal H^\varepsilon,U^\lambda)$ in \eqref{LQ-ham-penalized}.

\begin{proof}[Proof of Theorem \ref{limit-V-lambda}]
    The expression of $V^\lambda$ in \eqref{lambda-value} is directly from Theorem \ref{LQ-global-wp} and \eqref{LQ-ham-penalized}. According to \eqref{generalized-displacement-convexity-LQ},
    \begin{align*}
 & \left(\begin{matrix}-C\cdot D_2&D_3\\
 D^\top_3&-C\cdot D_1
 \end{matrix}\right)+4\varepsilon\cdot \left(\begin{matrix}0&0\\
 0&D_2
 \end{matrix}\right)\leq\left(\begin{matrix}-Q^{(1)}_{11}&-2Q^{(1)}_{21}\\
 -2Q^{(1)}_{12}&-4Q^{(1)}_{22}+4\varepsilon D_2
 \end{matrix}\right)\\
 &\qquad\leq\left(\begin{matrix}C\cdot D_1&D_3\\
 D^\top_3&C\cdot D_2
 \end{matrix}\right)+4\varepsilon\cdot \left(\begin{matrix}0&0\\
 0&D_2
 \end{matrix}\right),
 \end{align*}
    In view of Lemma  \ref{lem-comparison}, we obtain
    \begin{align}\label{a1-lambda-bound-1}
        a^{\lambda,-}(t)\leq a^\lambda_1(t)\leq a^{\lambda,+}(t),\quad t\in[0,T),
    \end{align}
    where $\big(a^{\lambda,-}(t),a^{\lambda,+}(t)\big)_{t\in[0,T)}$ solves
\begin{align}\label{a1-lambda}
 \left\{\begin{aligned}
  \frac d{dt}a^{\lambda,+}(t)=-CD_2+D_3a^{\lambda,+}(t)+a^{\lambda,+}(t)D^\top_3+a^{\lambda,+}(t)^\top(4\varepsilon D_2-CD_1)a^{\lambda,+}(t),\quad t\in[0,T),\\
  \frac d{dt}a^{\lambda,-}(t)=CD_1+D_3a^{\lambda,-}(t)+a^{\lambda,-}(t)D^\top_3+(4\varepsilon+C)a^{\lambda,-}(t)^\top D_2a^{\lambda,-}(t),\quad t\in[0,T),\\
  a^{\lambda,+}(T)=\lambda D_2+q^T_{22}D_1,\quad a^{\lambda,-}(T)=\lambda D_2+q^T_{22}D_1.
 \end{aligned}\right.
\end{align}
In view of \eqref{a1-lambda}, it is easy to see that
\begin{align*}
 a^{\lambda,+}(t)=\left(\begin{matrix}a^{\lambda,+}_1(t)&0\\
 0&\hat a^{\lambda,+}_1(t)
 \end{matrix}\right),
 \end{align*}
where $\big(a^{\lambda,+}_1(t),\hat a^{\lambda,+}_1(t)\big)_{t\in[0,T]}$ satisfies
\begin{align}\label{upper-1}
 \left\{\begin{aligned}
  \frac d{dt}a^{\lambda,+}_1(t)=-CI_{d_1}+q_{11}a^{\lambda,+}_1(t)+a^{\lambda,+}_1(t)q^\top_{11}+4\varepsilon a^{\lambda,+}_1(t)^\top a^{\lambda,+}_1(t),\quad t\in[0,T),\\
  \frac d{dt}\hat a^{\lambda,+}_1(t)=CI_{d_2}+q_{22}\hat a^{\lambda,+}_1(t)+\hat a^{\lambda,+}_1(t)q^\top_{22}-C\hat a^{\lambda,+}_1(t)^\top\hat a^{\lambda,+}_1(t),\quad t\in[0,T),\\
  a^{\lambda,+}_1(T)=\lambda I_{d_1},\quad\hat a^{\lambda,+}_1(T)=q^T_{22}.
 \end{aligned}\right.
\end{align}
The well-posedness of \eqref{upper-1} is standard and we may use Lemma \ref{lem-comparison} again to obtain
\begin{align*}
a^{\lambda,+}_1(t)\leq a^{\lambda,+}_2(t),\quad t\in[0,T],
\end{align*}
where
\begin{align}\label{upper-2}
\left\{\begin{aligned}  \frac d{dt}a^{\lambda,+}_2(t)=-CI_{d_1}+4\varepsilon C^{-1} a^{\lambda,+}_2(t)^\top a^{\lambda,+}_2(t),\quad t\in[0,T),\\
  a^{\lambda,+}_2(T)=\lambda CI_{d_1},\end{aligned}\right.
\end{align}
for a possibly different and sufficiently large $C>0$. Now it is straightforward to solve the above and get
\begin{align*}
 a^{\lambda,+}_1(t)&\leq\bigg[\frac{\sqrt C}{2\sqrt{\varepsilon}}\bigg(e^{4\sqrt{\varepsilon}(T-t)}\frac{2\sqrt{\varepsilon C}\lambda+\sqrt C}{2\sqrt{\varepsilon C}\lambda-\sqrt C}-1\bigg)^{-1}+\frac14\sqrt{\frac C\varepsilon}\bigg]I_{d_1}\\
 &\leq\bigg[\frac{\sqrt C}{2\sqrt{\varepsilon}}\big(e^{4\sqrt{\varepsilon}(T-t)}-1\big)^{-1}+\frac14\sqrt{\frac C\varepsilon}\bigg]I_{d_1},\quad t\in[0,T).
\end{align*}
At the meant time,
\begin{align*}
\hat a^{\lambda,+}_1(t)\leq\hat a^{\lambda,+}_2(t),\quad t\in[0,T],
\end{align*}
where
\begin{align*}
 \left\{\begin{aligned}\frac d{dt}\hat a^{\lambda,+}_2(t)=q_{22}\hat a^{\lambda,+}_2(t)+\hat a^{\lambda,+}_2(t)q^\top_{22},\quad t\in[0,T),\\
 \hat a^{\lambda,+}_2(T)=q^T_{22}.
 \end{aligned}\right.
\end{align*}
Hence $\hat a^{\lambda,+}_2(t)\leq e^{q_{22}(T-t)}q^T_{22}e^{q_{22}(T-t)},\ t\in[0,T]$. Similarly, for $\big(a^{\lambda,-}(t)\big)_{t\in[0,T]}$ it admits the following formulation:
\begin{align*}
 a^{\lambda,-}(t)=\left(\begin{matrix}a^{\lambda,-}_1(t)&0\\
 0&\hat a^{\lambda,-}_1(t)
 \end{matrix}\right),\quad t\in[0,T],
\end{align*}
for some $a^{\lambda,-}_1(t)\in\mathbb R^{d_1\times d_1}$, $\hat a^{\lambda,-}_1(t)\in\mathbb R^{d_2\times d_2}$. We may repeat the procedure in \eqref{upper-1}$\sim$\eqref{upper-2} towards $\big(a^{\lambda,-}_1(t),\hat a^{\lambda,-}_1(t)\big)_{t\in[0,T]}$ and show
\begin{align*}
 a^{\lambda,-}(t)\geq\left(\begin{matrix}\frac1{C\big(\lambda^{-1}+(4\varepsilon+C)(T-t)\big)}I_{d_2}&0\\
 0&e^{q_{22}(T-t)}q^T_{22}e^{q_{22}(T-t)}-C(T-t)I_{d_2}
 \end{matrix}\right).
\end{align*}
We may also apply the above analysis to $\big(a^\lambda_1(t)+a^\lambda_3(t)\big)_{t\in[0,T)}$ and obtain
\begin{align}\label{a1-lambda-bound-2}
        a^{\lambda,-}(t)\leq a^\lambda_1(t)+a^\lambda_3(t)\leq a^{\lambda,+}(t),\quad t\in[0,T).
\end{align}
In view of \eqref{a1-lambda-bound-1} and \eqref{a1-lambda-bound-2}, for each $\kappa>0$, $\big(a^\lambda_1(t)\big)_{t\in[0,T-\kappa]}$ and $\big(a^\lambda_1(t)+a^\lambda_3(t)\big)_{t\in[0,T-\kappa]}$ are bounded uniformly in $\lambda$. We can also see from Lemma \ref{lem-comparison} that $\big(a^\lambda_1(t)\big)_{t\in[0,T-\kappa]}$ and $\big(a^\lambda_1(t)+a^\lambda_3(t)\big)_{t\in[0,T-\kappa]}$ are both increasing in $\lambda$. Hence, as $\lambda$ goes to infinity, we have the uniform convergence of $\big(a^\lambda_1(t),a^\lambda_1(t)+a^\lambda_3(t)\big)_{t\in[0,T-\kappa]}$ to $\big(a^{+\infty}_1(t),a^{+\infty}_1(t)+a^{+\infty}_3(t)\big)_{t\in[0,T-\kappa]}$, as well as \eqref{uniform-estimates}.
\end{proof}
\subsection{Proofs for Theorem \ref{lqd-general-HJB-N-wp-1} and Proposition \ref{xi-implicit}}
Before delving into the proof of Theorem \ref{lqd-general-HJB-N-wp-1} and Proposition \ref{xi-implicit}, we establish several auxiliary results, starting with the approximation in Lemma \ref{dis-con-app}.
\begin{proof}[Proof of Lemma \ref{dis-con-app}]
In view of the law of large numbers (see e.g. \cite{Carmona2018-I}),
\begin{align*}
 \lim_{m\to+\infty}\mathbb E\bigg[\mathcal H\bigg(\frac1m\sum_{i=1}^m\delta_{\xi_i}\bigg)-\mathcal H(\tilde\mu)\bigg]\leq\lim_{m\to+\infty}\mathbb E\bigg[\mathcal W_2\bigg(\frac1m\sum_{i=1}^m\delta_{\xi_i},\tilde\mu\bigg)^2\bigg]=0.
\end{align*}
Since $D^2_{zz}\mathcal H^R$ is supported on a compact set, direct calculation shows that $\mathcal H_{R,N}$ has bounded derivatives. Denote by 
Then
\begin{align*}
 \lim_{m\to+\infty}\mathcal H_m(\tilde\mu)=\lim_{m\to+\infty}\mathbb E\bigg[\mathcal H^m\bigg(\frac1m\sum_{i=1}^m\delta_{\xi_i}\bigg)\bigg]=\lim_{m\to+\infty}\mathbb E\bigg[\mathcal H\bigg(\frac1m\sum_{i=1}^m\delta_{\xi_i}\bigg)\bigg]=\mathcal H(\tilde\mu).
\end{align*}
\end{proof}
We also need the following lemma to apply the results in Theorem \ref{existence-of-decoupling-field} and Theorem \ref{prop-eigen}, as well as the construction of optimal robust strategy.
\begin{lemma}\label{emp-ham}
For $\tilde\mu\in\mathcal P_2(\mathbb R^d\times\mathbb R^d)$, let $(X,P)\in L^2(\Omega,\mathcal F,\mathbb P;\mathbb R^d\times\mathbb R^d)$ be its lifting over an appropriate probability space $(\Omega,\mathcal F,\mathbb P)$, and denote by
\begin{align}\label{cal K}
 \mathcal K(\tilde\mu):=\inf_{\xi\in L^2(\Omega,\mathcal F,\mathbb P;\mathbb R^d)}\mathbb E[\psi(\xi,\mu_\xi)\cdot P+\mathcal G(\mu_{\xi,X})],
\end{align}
where $\psi\in\mathcal C^2\big(\mathcal P_2(\mathbb R^d)\times\mathbb R^d;\mathbb R^d\big)$ with bounded derivatives, $\mathcal G\in\mathcal C^2\big(\mathcal P_2(\mathbb R^d)\big)$.
Suppose that  for certain empirical measure
\begin{align*}
 \tilde\mu_N=\frac1N\sum_{i=1}^N\delta_{(x_i,p_i)},\quad\ x_1,\ldots,x_N,\ p_1,\ldots,p_N\in\mathbb R^d,
\end{align*}
with lifting $(X_N,P_N)$, it holds for some $\delta>0$ that
\begin{align}\label{cal G}
&\mathbb E\big[\mathbb{\breve E}[(\partial_\mu\psi(\xi_1,\mu_{\xi_1},\breve\xi_1)-\partial_\mu\psi(\xi_2,\mu_{\xi_2},\breve\xi_2))^\top\breve P_N]+(D_\xi\psi(\xi_1,\mu_{\xi_1})-D_\xi\psi(\xi_2,\mu_{\xi_2}))^\top P_N\big]\notag\\
 &\qquad+\mathbb E\big[\big\langle\partial_\mu\mathcal G(\mu_{(\xi_1,X_N)},\xi_1,X_N)^{(\xi)}-\partial_\mu\mathcal G(\mu_{(\xi_2,X_N)},\xi_2,X_N)^{(\xi)},\xi_1-\xi_2\big\rangle\big]\geq\delta\mathbb E|\xi_1-\xi_2|^2,\notag\\
 &\xi_1,\xi_2\in L^2(\Omega,\mathcal F,\mathbb P;\mathbb R^d).
\end{align}
Then the infimum in \eqref{cal K} is attained at $\xi^*_N\in L^2(\Omega,\mathcal F,\mathbb P;\mathbb R^d)$ satisfying
\begin{align*}
 {\rm Law}(\xi^*_N,X_N,P_N)=\frac1N\sum_{i=1}^N\delta_{(y^*_i,x_i,p_i)},
 \end{align*}
 where $(y^*_1,\ldots,y^*_N)$ uniquely solves
\begin{align*}
 \partial_\mu\mathcal G(\mu_{(\xi^*_N,X_N)},y^*_i,x_i)^{(\xi)}+\frac1N\sum_{j=1}^N\partial_\mu\psi(y^*_j,\mu_{\xi^*_N},y^*_i)^\top p_j+D_\xi\psi(y^*_i,\mu_{\xi^*_N})^\top p_i=0,\quad1\leq i\leq N,
 \end{align*}
 in other words,
 \begin{align}\label{FOC-0}
  \partial_\mu\mathcal G(\mu_{(\xi^*_N,X_N)},\xi^*_N,X_N)^{(\xi)}+\breve{\mathbb E}[\partial_\mu\psi(\breve\xi^*_N,\mu_{\xi^*_N},\xi^*_N)^\top\breve P_N]+D_\xi\psi(\xi^*_N,\mu_{\xi^*_N})^\top P_N=0\ {\rm a.s.}.
 \end{align}
Moreover, suppose that $\tilde\mu_N\to\tilde\mu$ in $\mathcal P_2(\mathbb R^d)$ for some $\tilde\mu\in\mathcal P_2(\mathbb R^d)$ with lifting $(X,P)\in L^2(\Omega,\mathcal F,\mathbb P;\mathbb R^d)$, and all the derivatives of $\psi,\ \mathcal G$ are bounded, then there exists a unique $\xi^*\in L^2(\Omega,\mathcal F,\mathbb P;\mathbb R^d)$ such that
\begin{align}\label{FOC-1}
\left\{\begin{aligned}&\mathcal K(\mu)=\mathbb E[\psi(\xi^*,\mu_{\xi^*})\cdot P+\mathcal G(\mu_{(\xi^*,X)})],\\
&\big(D_\xi\psi(\xi^*,\mu_{\xi^*})^\top+\breve{\mathbb E}[\partial_\mu\psi(\xi^*,\mu_{\xi^*},\breve\xi^*)^\top]\big)P+\partial_\mu\mathcal G(\mu_{(\xi^*,X)},\xi^*,X)^{(\xi)}=0\ {\rm a.s.}.\end{aligned}\right.
\end{align}
\end{lemma}
\begin{proof}
 Since $\psi\in\mathcal C^2\big(\mathcal P_2(\mathbb R^d)\times\mathbb R^d;\mathbb R^d\big)$, $\mathcal G\in\mathcal C^2\big(\mathcal P_2(\mathbb R^d)\big)$, in view of the law of large numbers,
 \begin{align}\label{inf-1}
  \mathcal K(\mu_N)=\inf_{\substack{\xi\in L^2(\Omega,\mathcal F,\mathbb P;\mathbb R^d)\\ \mu_{(\xi,\eta)}=\frac1{kN}\sum_{i=1}^N\sum_{j=1}^N\delta_{(y_{(i-1)k+j},x_i)}\\k\in\mathbb N}}\mathbb E[\psi(\xi,\mu_\xi)\cdot\eta+\mathcal G(\mu_\xi)].
 \end{align}
For $k\geq1$, consider $\xi,\ \tilde\xi\in L^2(\Omega,\mathcal F,\mathbb P;\mathbb R^d)$ in \eqref{cal G} satisfying
\begin{align*}
 {\rm Law}(\xi,\tilde\xi,\eta)=\frac1{kN}\sum_{i=1}^N\sum_{j=1}^N\delta_{(y_{(i-1)k+j},\tilde y_{(i-1)k+j},x_i)},
\end{align*}
and denote by
\begin{align*}
 &G_k(y_1,\ldots,y_{kN}):=\frac1{kN}\sum_{i=1}^N\sum_{j=1}^k\psi\bigg(y_{(i-1)k+j},\frac1{kN}\sum_{i=1}^{kN}\delta_{y_i}\bigg)\cdot x_i+\mathcal G\bigg(\frac1{kN}\sum_{i=1}^{kN}\delta_{y_i}\bigg),\\
 &(y_1,\ldots,y_{kN})\in(\mathbb R^d)^{kN}.
\end{align*}
Then for $\mu_N$ and each $k\geq1$,
\begin{align}\label{inf-2}
 \inf_{\substack{\xi\in L^2(\Omega,\mathcal F,\mathbb P;\mathbb R^d)\\ \mu_{(\xi,\eta)}=\frac1{kN}\sum_{i=1}^N\sum_{j=1}^N\delta_{(y_{(i-1)k+j},x_i)}}}\mathbb E[\psi(\xi,\mu_\xi)\cdot\eta+\mathcal G(\mu_\xi)]=\inf_{(y_1,\ldots,y_{kN})\in(\mathbb R^d)^{kN}}G_k(y_1,\ldots,y_{kN}).
\end{align}
Moreover, for $y=(y_1,\ldots,y_{kN})$, $\tilde y=(\tilde y_1,\ldots,\tilde y_{kN})$, \eqref{cal G} implies that
\begin{align*}
 \langle\nabla G_k(y)-\nabla G_k(\tilde y),y-\tilde y\rangle\geq\delta|y-\tilde y|^2.
\end{align*}
Hence the strict convexity and super linear growth of $G_k$, which implies a unique minima. In particular, the following has a unique solution $(y^*_1,\ldots,y^*_N)$:
\begin{align}\label{FOC-G1}
 \nabla G_1(y^*_1,\ldots,y^*_N)=0,
\end{align}
which yields the unique solution to \eqref{FOC-0}. Moreover, $G_k$ also admits a unique minima $y^{(k)}$. According to \eqref{FOC-G1}, it is straight forward to check that
\begin{align}\label{inf-3}
 \hat y^*_{(i-1)k+j}:=y^*_i,\quad 1\leq i\leq N,\ 1\leq j\leq k,
\end{align}
satisfies the following first order condition related to $y^{(k)}$:
\begin{align*}
 \nabla G_k(\hat y^*_1,\ldots,\hat y^*_{kN})=0.
\end{align*}
Combining with the uniqueness of minima, we may plug \eqref{inf-3} into \eqref{inf-2} and obtain
\begin{align*}
 \inf_{\substack{\xi\in L^2(\Omega,\mathcal F,\mathbb P;\mathbb R^d)\\ \xi=\frac1{kN}\sum_{i=1}^{kN}\delta_{\xi_i}}}\mathbb E[\psi(\xi,\mu_\xi)\cdot P+\mathcal G(\mu_{\xi,X})]=\frac1N\sum_{j=1}^N\psi\bigg(y^*_j,\frac1N\sum_{i=1}^N\delta_{y^*_i}\bigg)\cdot p_j+\mathcal G\bigg(\frac1N\sum_{i=1}^N\delta_{(y^*_i,x_i)}\bigg).
\end{align*}
Noting that the right hand side of the above is independent of $k$, by \eqref{inf-1} we have \eqref{FOC-0} and 
\begin{align*}
 \mathcal K(\tilde\mu_N)=\mathbb E\big[\psi(\xi^*_N,\mu_{\xi^*_N})\cdot P_N+\mathcal G(\mu_{(\xi^*_N,X_N)})\big].
\end{align*}
Let $(\mu_N)_{N\geq1}$ be empirical measures and $\mu_N\to\mu\in\mathcal P_2(\mathbb R^d)$. Since all the derivatives of $\psi$ and $\mathcal G$ are bounded, then by \eqref{FOC-0} and \eqref{cal G},
\begin{align*}
&\quad C\big(\mathbb E|X_N-X_M|^2+\mathbb E|P_N-P_M|^2\big)^\frac12\cdot\big(\mathbb E|\xi^*_N-\xi^*_M|^2\big)^\frac12\\
 &\geq\mathbb E\big[\mathbb{\breve E}[\partial_\mu\psi(\xi^*_N,\mu_{\xi^*_N},\breve\xi^*_N)-\partial_\mu\psi(\xi^*_M,\mu_{\xi^*_M},\breve\xi^*_M)]^\top P_N+(D_\xi\psi(\xi^*_N,\mu_{\xi^*_N})-D_\xi\psi(\xi^*_M,\mu_{\xi^*_M}))^\top P_N\big]\notag\\
 &\qquad+\mathbb E\big[\big\langle\partial_\mu G(\mu_{(\xi^*_N,X_N)},\xi^*_N,X_N)^{(\xi)}-\partial_\mu G(\mu_{(\xi^*_M,X_N)},\xi^*_M,X_N)^{(\xi)},\xi^*_N-\xi^*_M\big\rangle\big]\\&\geq\delta\mathbb E|\xi^*_N-\xi^*_M|^2.
\end{align*}
The above and $\tilde\mu_N\to\tilde\mu\in\mathcal P_2(\mathbb R^d\times\mathbb R^d)$ implies that $(\xi^*_N)_{N\geq1}$ is convergent. Therefore we have \eqref{FOC-1} by Lemma 10.1.3 in \cite{Ambrosio2008}. Define
\begin{align*}
 g(\lambda)=\psi(\lambda\xi_1+(1-\lambda)\xi_2,\mu_{\lambda\xi_1+(1-\lambda)\xi_2})\cdot P+\mathcal G(\mu_{\lambda\xi_1+(1-\lambda)\xi_2,X}),\quad\lambda\in[0,1].
\end{align*}
For $\lambda_1>\lambda_2$, denote by
\begin{align*}
 \xi^\lambda_i=\lambda_i\xi_1+(1-\lambda_i)\xi_2,\ i=1,2.
\end{align*}
We may further send $N$ to infinity in \eqref{cal G} and see that 
\begin{align*}
 &\quad g'(\lambda_1)-g'(\lambda_2)\\
 &=(\lambda_1-\lambda_2)\mathbb E\big[\mathbb{\breve E}[(\partial_\mu\psi(\breve\xi^\lambda_1,\mu_{\xi^\lambda_1},\xi^\lambda_1)-\partial_\mu\psi(\breve\xi^\lambda_2,\mu_{\xi^\lambda_2},\xi^\lambda_2))^\top\breve P]+(D_\xi\psi(\xi^\lambda_1,\mu_{\xi^\lambda_1})-D_\xi\psi(\xi^\lambda_2,\mu_{\xi^\lambda_2}))^\top P\big]\notag\\
 &\qquad+\mathbb E\big[\big\langle\partial_\mu G(\mu_{(\xi^\lambda_1,X)},\xi^\lambda_1,X)^{(\xi)}-\partial_\mu G(\mu_{(\xi^\lambda_2,X)},\xi^\lambda_2,X)^{(\xi)},\xi^\lambda_1-\xi^\lambda_2\big\rangle\big]\geq\delta(\lambda_1-\lambda_2)\mathbb E|\xi^\lambda_1-\xi^\lambda_2|^2.
\end{align*}
Hence $g$ is strictly convex, therefore $\xi^*\in L^2(\Omega,\mathcal F,\mathbb P;\mathbb R^d)$ in \eqref{FOC-1} is unique.
\end{proof}
 Having obtained the above auxiliary results, we may now consider \eqref{lqd-general-HJB-m}. Similar to \eqref{HJB-N}$\sim$\eqref{HJB-N-1}, \eqref{lqd-general-HJB-m} corresponds to the following controlled particle system:
\begin{align}\label{lqd-general-HJB-N-m}
  \left\{\begin{aligned}
    &\partial_tV_{N,n}(t,x)+\frac12\sum_{i=1}^N{\rm tr}\big[\sigma^\top\sigma D^2_{x_ix_i}V_{N,n}(t,x)\big]+H_{N,n}\big(x,\nabla_x V_{N,n}(t,x)\big)=0,\\
    &\qquad\qquad(t,x)\in[0,T)\times\mathbb(R^d)^N,\\
    &V_{N,n}(T,x)=U_{N,n}(x),\quad x\in(\mathbb R^d)^N,
  \end{aligned}\right.
\end{align}
where for $x=(x_1,\ldots,x_N),\ p=(p_1,\ldots,p_N)\in(\mathbb R^d)^N$ and $(\mathcal H_n,\mathcal U_n)$ from \eqref{lqd-general-HJB-m},
\begin{align*}
 &H_{N,n}(x,p):=\mathcal H_n\bigg(\frac1N\sum_{i=1}^N\delta_{(x_i,Np_i)}\bigg),\quad U_{N,n}(x):=\mathcal U_n\bigg(\frac1N\sum_{i=1}^N\delta_{x_i}\bigg).
\end{align*}
\begin{lemma}\label{uniform-derivative-bound}
For $n\geq1$, $\mathcal H_n$ satisfies
\begin{align*}
    |\partial_\mu\mathcal H_n(\tilde\mu,z)|\leq C\bigg(1+|z|+\int_{\mathbb R^d}|\tilde z|\mu(d\tilde z)\bigg),\quad(\mu,z)\in\mathcal P_2(\mathbb R^d)\times\mathbb R^d,
\end{align*}
where the constant $C$ is independent of $n$.
\end{lemma}
\begin{proof}
Let us consider $\mathcal H$ from \eqref{lqd-general-ham} and show that
    \begin{align}\label{cal H growth}
     |\partial_\mu\mathcal H(\tilde\mu,z)|\leq C\bigg(1+|z|+\int_{\mathbb R^d}|\tilde z|\tilde\mu(d\tilde z)\bigg).
    \end{align}
It suffices to consider $\tilde\mu=\frac1N\sum_{i=1}^N\delta_{z_i}$, $N\geq1$, $z_i=(m_i,q_i,p_{1,i},p_{2,i})\in(\mathbb R^{2d})$, $1\leq i\leq N$. In view of Lemma \ref{emp-ham} and \eqref{lqd-general-ham},
\begin{align}\label{lqd-general-ham-1}
 &\mathcal H\bigg(\frac1N\sum_{i=1}^N\delta_{z_i}\bigg)=f\bigg(\frac1N\sum_{i=1}^N\delta_{m_i}\bigg)+\frac{\kappa^2_2}{4\lambda N}\sum_{i=1}^N|p_{2,i}|^2+\Phi^*\bigg(\frac1N\sum_{i=1}^Np_{2,i}\bigg)\\
 &\qquad\qquad+\frac1N\sum_{i=1}^N\bigg[-p_{1,i}\cdot\xi^*_i+p_{2,i}\cdot\bigg(\gamma_1\xi^*_i+\frac{\gamma_2}N\sum_{j=1}^N\xi^*_j+\kappa_2\theta_i+\kappa_1\alpha_i\bigg)+U\big(\xi^*_i,q_i,\frac1N\sum_{j=1}^N\delta_{\xi^*_j}\big)\bigg].\notag
\end{align}
Here $(\xi^*_1,\ldots,\xi^*_N)$ is implicitly determined by $(z_1,\ldots,z_N)$ via the following first order condition
\begin{align*}
 \Psi(z_1,\ldots,z_N,\xi^*_1,\ldots,\xi^*_N)=0,
\end{align*}
where
\begin{align}\label{FOC}
    &\quad\Psi_i(z_1,\ldots,z_N,\xi^*_1,\ldots,\xi^*_N)\\
    &=-p_{i,1}+\gamma^\top_1p_{i,2}+\frac{\gamma^\top_2}N\sum_{k=1}^Np_{k,2}+\partial_\xi U\big(\xi^*_i,q_i,\frac1N\sum_{j=1}^N\delta_{\xi^*_j}\big)+\frac1N\sum_{k=1}^N\partial_\mu U\big(\xi^*_k,q_i,\frac1N\sum_{j=1}^N\delta_{\xi^*_j},\xi^*_i\big),\notag\\
    &1\leq i\leq N.\notag
\end{align}
According to \eqref{ham-convexity-1} and the implicit function theorem, some algebra gives
\begin{align*}
 \bigg|D_{z_1}\mathcal H\big(\frac1N\sum_{i=1}^N\delta_{z_i}\big)\bigg|\leq C\bigg(1+|z_1|+\frac1N\sum_{i=1}^N|z_i|\bigg),
\end{align*}
which implies \eqref{cal H growth}. In view of \eqref{N-poly}$\sim$\eqref{convexity}, as well as the growth condition \eqref{cal H growth},
    \begin{align*}
     |\partial_\mu\mathcal H_n(\tilde\mu,z_1)|&\leq\int_{\mathbb R^d}\cdots\int_{\mathbb R^d}\bigg|\partial_\mu\mathcal H^n\bigg(\frac1n\sum_{i=1}^n\delta_{z_i},z_1\bigg)\bigg|\tilde\mu(dz_2)\cdots\tilde\mu(dz_n)\\
     &\leq C\int_{\mathbb R^d}\cdots\int_{\mathbb R^d}\bigg(\frac1n\sum_{i=1}^n|z_i|+|z_1|+1\bigg)\tilde\mu(dz_2)\cdots\tilde\mu(dz_n),
\end{align*}
which completes the proof.
\end{proof}
Next we analyze the auxiliary equation \eqref{lqd-general-HJB-N-m}.
 \begin{lemma}\label{lqd-general-HJB-N-wp}
It holds that
  \begin{enumerate}
  \item[(i)]The HJBI equation \eqref{lqd-general-HJB-N-m} has a unique classical solution $V_{N,n}$;
  \item[(ii)]$\big(V_{N,n}\big)_{N,n\geq1}$ satisfy
   \begin{align}\label{DxiVNm-negative}
 \left\{\begin{aligned}
 &N|D_{x_i}V_{N,n}(t,x)|\leq C(|x_i|+1)+\bigg(\frac CN\sum_{j=1}^N|x_j|^2\bigg)^{\frac12},\\
 &\qquad 1\leq i\leq N,\ (t,x)\in[0,T]\times (\mathbb R^d)^N,
 \end{aligned}\right.
\end{align}
 where the constant $C$ is independent of $N,n$, and $x=(x_1,\ldots,x_N)^\top$, $x_i=(m_i,q_i)^\top\in\mathbb R\times\mathbb R$, $1\leq i\leq N;$
   \item[(iii)] The HJBI equation \eqref{lqd-general-HJB-m} has a unique classical solution $V_n$ satisfying
   \begin{align}\label{higher-order-goal-1}
 \left\{\begin{aligned}
 &|\partial_\mu V_n(t,\mu,x)|\leq C(|x|+1)+C\bigg(\int_{\mathbb R^2}|\tilde x|^2\mu(d\tilde x)\bigg)^{\frac12},\quad(t,\mu,x)\in[0,T]\times\mathcal P_2(\mathbb R^2)\times(\mathbb R^d)^N,\\
 &\sum_{\substack{0\leq k\leq 4,0\leq l_1,\ldots,l_k\leq4\\2\leq l_1+\cdots+l_k+k\leq 4}}\big|D^{l_1}_{x_1}\cdots D^{l_k}_{x_k}\partial^k_\mu V_n(t,\cdot)\big|_\infty<C,
 \end{aligned}\right.
\end{align}
 where the constant depends only on $\mathcal H,\ U$;   
   \item[(iv)] $(\partial_tV_n)_{n\geq1}$ is locally bounded uniformly in $n$ by a constant depending only on $\mathcal H,\ U$.
  \end{enumerate}
 \end{lemma}
\begin{proof}
(i). It is the consequence of Lemma \ref{differentiability} because $\mathcal H_n$ satisfies Assumption \ref{assumption-displacement-convex}.

(ii). Consider
\begin{align*}
 \left\{\begin{aligned}
 &dX^i_t=\sigma W^i_t+\sigma_0 W^0_t,\quad Y^i_t=D_{x_i}V_{N,n}(t,X_t),\\
 &\qquad t\in[0,T],\ i=1,\ldots,N,\ X(0)=x=(x_1,\ldots,x_N)^\top.
 \end{aligned}\right.
 \end{align*}
 Taking $D_{x_i}$ in \eqref{lqd-general-HJB-N-m} yields
 \begin{align}\label{dynamic-Y}
 Y_t=\mathbb E_t\bigg[\nabla_xU_{N,n}(\mu_{X_T})+\int_t^T\big(\nabla_xH_{N,n}(t,X_t,Y_t)+D^2V_{N,n}(t,X_t)\nabla_pH_{N,n}(t,X_t,Y_t)\big)dt\bigg].
\end{align}
According to Lemma \ref{uniform-derivative-bound},
 \begin{align*}
  \big|ND_{x_i}H_{N,n}(t,X_t,Y_t)\big|+\big|D_{p_i}H_{N,n}(t,X_t,Y_t)\big|\leq C\big(|X^i_t|+N|Y^i_t|+\sqrt{N}^{-1}|X_t|+\sqrt{N}^{-1}|NY_t|+1\big).
 \end{align*}
In view of Theorem \ref{prop-eigen} and the above, 
\begin{align*}
  &\quad\big|(NY_t)^\top D^2V_{N,n}(t,X_t)N\nabla_pH_{N,n}(t,X_t,Y_t)\big|\leq\frac CN|NY_t|\cdot|N\nabla_pH_{N,n}(t,X_t,Y_t)|\\
&\leq C\big(|NY_t|^2+\sqrt{N}|NY_t|+|NY_t|\cdot|X_t|\big)\leq C(|NY_t|^2+|X_t|^2+N).
\end{align*}
Denoting by $D(t):=\max_{s\in[t,T]}\mathbb E[|NY_t|^2],\ t\in[0,T]$, and applying It\^o's formula to $|NY_t|^2$, we obtain from the above estimate that
\begin{align*}
D(t)\leq NC+C\int_t^T\big(\mathbb E|X_s|^2+N+D(s)\big)ds\leq NC+C\int_t^T\big(|x|^2+N+D(s)\big)ds.
\end{align*}
Therefore $D(t)\leq C_T(|x|^2+N),\ t\in[0,T]$. The estimates on $D(t)$ and \eqref{dynamic-Y} then yield
\begin{align*}
 |Y^i_0|&\leq\mathbb E\bigg[D_{x_i}U_{N,n}(\mu_{X_T})+\int_t^T\big(D_{x_i}H_{N,n}(t,X_t,Y_t)+(e^i_N)^\top D^2V_{N,n}(t,X_t)\nabla_pH_{N,n}(t,X_t,Y_t)\big)dt\bigg]\\
 &\leq\frac1N(|x^i|+N^{-1}|x|).
\end{align*} 

(iii). The well-posedness is the result Theorem \ref{existence-of-decoupling-field}. Let ${\rm Law}(\xi,\eta)\in\mathcal P_2(\mathbb R^d\times\mathbb R^d)$ and consider the sequence
\begin{align*}
 \frac1N\sum_{i=1}^N\delta_{(x_i,y_i)}\to{\rm Law}(\xi,\eta)\ \text{in}\ \mathcal P_2(\mathbb R^d\times\mathbb R^d).
\end{align*}
According to \eqref{DxiVNm-negative},
\begin{align*}
 |V_{N,n}(t,x+\varepsilon y)-V_{N,n}(t,x)|\leq\frac{\varepsilon C}N\sum_{i=1}^N\bigg[|x_i|+1+\bigg(\frac1N\sum_{j=1}^N|x_j|^2\bigg)^{\frac12}\bigg]|y_i|+\varepsilon^2C\bigg(\frac1N\sum_{j=1}^N|y_j|^2\bigg)^{\frac12}.
\end{align*}
Sending $N$ to infinity yields
\begin{align*}
 |V_n\big(t,{\rm Law}(\xi+\varepsilon\eta)\big)-V_n\big(t,{\rm Law}(\xi)\big)|\leq\varepsilon C\mathbb E\big[(|\xi|+\mathbb E[|\xi|^2]^\frac12)|\eta|\big]+\varepsilon^2C\mathbb E[|\eta|^2]^\frac12.
\end{align*}
Notice that the above holds for arbitrary ${\rm Law}(\xi,\eta)\in\mathcal P_2(\mathbb R^d\times\mathbb R^d)$, hence the first estimate in \eqref{higher-order-goal-1}. 

Consider \eqref{1st-step} and apply it to \eqref{lqd-general-HJB-m}. Let $(\nu_1,\nu_2)\in L^2(\Omega,\mathcal F,\mathbb P;\mathbb R^{d_1}\times\mathbb R^{d_2})$ be the lifting of $\mu$. Then
 \begin{align}\label{1st-step-1}
  \left\{\begin{aligned}
    dM_s&=\partial_{\tilde\mu}\mathcal H^{(Y_1)}_n\big(\mathbb P_{(M_s,Q_s,Y^1_s,Y^2_s)},M_s,Q_s,Y^1_s,Y^2_s\big)ds,\\
    dQ_s&=\partial_{\tilde\mu}\mathcal H^{(Y_2)}_n\big(\mathbb P_{(M_s,Q_s,Y^1_s,Y^2_s)},M_s,Q_s,Y^1_s,Y^2_s\big)ds+dW_s,\ (M_t,Q_t)=(\nu_1,\nu_2),\\
    dY^1_s&=-\partial_{\tilde\mu}\mathcal H^{(M)}_n\big(\mathbb P_{(M_s,Q_s,Y^1_s,Y^2_s)},M_s,Q_s,Y^1_s,Y^2_s\big)ds+Z^1_sdW_s,\ Y^1_T=\partial_\mu U^{(M)}(\mathbb P_{(M_T,Q_T)},M_T,Q_T),\\
    dY^2_s&=-\partial_{\tilde\mu}\mathcal H^{(Q)}_n\big(\mathbb P_{(M_s,Q_s,Y^1_s,Y^2_s)},M_s,Q_s,Y^1_s,Y^2_s\big)ds+Z^2_sdW_s,\ Y^2_T=\partial_\mu U^{(Q)}(\mathbb P_{(M_T,Q_T)},M_T,Q_T),\\
    V_s=&\mathbb E\bigg[U(\mathbb P_{(M_T,Q_T)})+\int_s^T\bigg(\mathcal H_n\big(\mathbb P_{(M_u,Q_u,Y^1_u,Y^2_u)}\big)-Y^1_u\cdot\partial_{\tilde\mu}\mathcal H^{(Y_1)}_n\big(\mathbb P_{(M_u,Q_u,Y^1_u,Y^2_u)},M_u,Q_u,Y^1_u,Y^2_u\big)\\
    &-Y^2_u\cdot\partial_{\tilde\mu}\mathcal H^{(Y_2)}_n\big(\mathbb P_{(M_u,Q_u,Y^1_u,Y^2_u)},M_u,Q_u,Y^1_u,Y^2_u\big)\bigg)du\bigg|\mathcal F^W_s\bigg].
\end{aligned}\right.
\end{align}
As is mentioned in Theorem \ref{existence-of-decoupling-field} and its proof, \eqref{lqd-general-HJB-m} is globally well-posed with bounded derivatives. In addition,
\begin{align}\label{VN-Fey-Kac}
 V_t=V_n(t,\mu),\quad\partial_\mu V_n(t,\mu)=(Y^1_t,Y^2_t).
\end{align}
According to \eqref{lqd-general-ham} and \eqref{convexity}, direct calculation yields that both $\mathcal H^{(X)}_n$ and $\mathcal H^{(P)}_n$ are bounded. Furthermore, for $0\leq\hat k\leq 4,0\leq \hat l_1,\ldots,\hat l_{\hat k}\leq4,\ \hat l_1+\cdots+\hat l_{\hat k}+\hat k=\hat d$, $1\leq\hat d\leq 4$ we may show that
\begin{align*}
 \big|\partial^{\hat l_1}_{x_1}\cdots\partial^{\hat l_{\hat k}}_{x_{\hat k}}\partial^{\hat k}_\mu\mathcal H_N(t,\cdot)\big|_\infty\leq C\sum_{\substack{0\leq k\leq 4,0\leq l_1,\ldots,l_k\leq4\\2\leq l_1+\cdots+l_k+k\leq d}}\big|\partial^{l_1}_{x_1}\cdots\partial^{l_k}_{x_k}\partial^k_\mu \mathcal H(t,\cdot)\big|_\infty.
\end{align*}
Therefore, following the same idea as in the proof of Theorem \ref{existence-of-decoupling-field}, we may show the second estimates in \eqref{higher-order-goal-1}.

(iv). In view of the uniform estimates in \eqref{higher-order-goal-1} and the equation \eqref{lqd-general-HJB-m}, we obtain the uniform estimates on $\partial_tV_n$.
 \end{proof}
With the above estimates at hand, we are now able to prove Theorem \ref{lqd-general-HJB-N-wp-1}.
\begin{proof}[Proof of Theorem \ref{lqd-general-HJB-N-wp-1}]
 Let us first show the convergence of $(V_n)_{n\geq1}$ to certain $V$. In view of the uniform estimates in Lemma \ref{lqd-general-HJB-N-wp} as well as the Arzela-Ascoli theorem, it suffices to show that $(V_n)_{n\geq1}$ are uniformly bounded on any compact set. Consider \eqref{1st-step-1}. In view of \eqref{higher-order-goal-1}, \eqref{VN-Fey-Kac}, Lemma \ref{uniform-derivative-bound} and Gronwall's inequality,
 \begin{align*}
  \mathbb E[|M_s|+|Q_s|]\leq C\mathbb E[|\nu_1|+|\nu_2|+1],\quad s\in[t,T].
 \end{align*}
Plugging the above into the last equality in \eqref{1st-step-1}, we obtain
\begin{align*}
 |V_n(t,\mu)|=|\mathbb EV_t|\leq C\mathbb E\big[|\nu_1|+|\nu_2|+1\big],
\end{align*}
which implies that $(V_n)_{n\geq1}$ are uniformly bounded on any compact set. Hence the convergence of $(V_n)_{n\geq1}$ to certain $V$. Furthermore, for $0\leq\hat k\leq 4,0\leq \hat l_1,\ldots,\hat l_{\hat k}\leq4,\ 1\leq\hat l_1+\cdots+\hat l_{\hat k}+\hat k\leq 3$, \eqref{higher-order-goal} also implies the uniform convergence of $\big(\partial^{\hat l_1}_{x_1}\cdots\partial^{\hat l_{\hat k}}_{x_{\hat k}}\partial^{\hat k}_\mu V_n\big)_{n\geq1}$ to $\partial^{\hat l_1}_{x_1}\cdots\partial^{\hat l_{\hat k}}_{x_{\hat k}}\partial^{\hat k}_\mu V$. Therefore we may pass $n$ to infinity in \eqref{lqd-general-HJB-m} and show that $V$ solves \eqref{lqd-general-HJB} in the classical sense. The uniqueness follows by the verification results in Section \ref{robust-mfc-problem}.
\end{proof}
We finish this subsection with the construction of an optimal robust strategy.

\begin{proof}[Proof of Proposition \ref{xi-implicit}]
In view of Theorem \ref{verification} and the setting of the current section, it suffices to verify Assumption \ref{assumption-saddle-points} for $(Id,\partial_\mu V(t,\mu,\cdot))\sharp\mu$, $\mu\in\mathcal P_2(\mathbb R^d)$, and show that $\big(\xi^*(t,\cdot),\eta^*(t,\cdot),\alpha^*(t,\cdot)\big)$ is Lipschitz and admissible. For arbitrary $N\geq1$ and $(m_i,q_i)_{1\leq i\leq N}\in(\mathbb R^d)^N$, define:
\begin{align*}
 \mu_N:=\frac1N\sum_{i=1}^N\delta_{(m_i,q_i)},\quad\tilde\mu_N:=\big(Id,\partial_\mu V(t,\mu_N,\cdot)\big)\sharp\mu_N,
\end{align*}
as well as
\begin{align*}
 \psi(\xi,\mu)=-\phi\xi-\tilde\phi\int_{\mathbb R^{d_1}}\tilde\xi\mu(d\tilde\xi),\ \mathcal G(\tilde\mu)=\int_{\mathbb R^{d_1}}\int_{\mathbb R^{d_1}}U(\xi,q,\pi_2\sharp\tilde\mu)\tilde\mu(d\xi dq),\\
 \mu\in\mathcal P_2(\mathbb R^{d_1}),\ \tilde\mu\in\mathcal P_2(\mathbb R^{d_1}\times\mathbb R^{d_2}),
\end{align*}
where $\pi_2:(x,\xi)\mapsto\xi$ is the projection mapping. In view of \eqref{ham-convexity-1}, we may plug the above into Lemma \ref{emp-ham} and verify that \eqref{cal G} holds. According to Lemma \ref{emp-ham}, there exists a unique $(\xi^*_{N,i})_{1\leq i\leq N}$ such that
\begin{align}\label{xi-implicit-1}
  &-\tilde\phi^\top\int_{\mathbb R^d}\partial_\mu V^{(q)}(t,\mu_N,m,q)\mu_N(dmdq)+D_\xi U(\xi^*_{N,i},m_i,\nu_N)+\frac1N\sum_{j=1}^N\partial_\mu U(\xi^*_{N,j},m_j,\nu_N,\xi^*_{N,i})\\
  &\quad-\phi^\top\partial_\mu V^{(q)}(t,\mu_N,m_i,q_i)-\partial_\mu V^{(m)}(t,\mu_N,m_i,q_i)=0,\quad1\leq i\leq N,\notag
\end{align}
where $\nu_N=\frac1N\sum_{i=1}^N\delta_{\xi^*_{N,i}}$, the joint distribution
\begin{align*}
 {\rm Law}(\xi^*_N,Q_N,M_N)=\frac1N\sum_{i=1}^N\delta_{(\xi^*_{N,i},q_i,m_i)}.
\end{align*}
For empirical measure $\mu_N$ and $(m,q),\ (\tilde m,\tilde q)\in{\rm spt}(\mu_N)$, define 
\begin{align*}
& \xi^*(t,\mu_N,m_i,p_i)=:\xi^*_{N,i},\quad \xi^*(t,\mu_N,\tilde m_i,\tilde p_i)=:\tilde\xi^*_{N,i},\\
 &D_\xi U(\xi_i,m_i,\mu_N)+\frac1N\sum_{j=1}^N\partial_\mu U(\xi_j,m_j,\mu_N,\xi^*_i)=:U^*_{N,i},\quad i=1,\ldots,N.
\end{align*}
Then \eqref{ham-convexity-1} and \eqref{xi-implicit-1} imply
\begin{align}\label{xi-Lipschitz}
\delta|\xi^*_N-\tilde\xi^*_N|^2&\leq\big\langle\xi^*_N-\tilde\xi^*_N,U^*_N-\tilde U^*_N\big\rangle\notag\\
&\leq C\big(|\partial_\mu V(t,\mu_N,m,q)-\partial_\mu V(t,\mu_N,\tilde m,\tilde q)|+|m-\tilde m|+|q-\tilde q|\big)\cdot|\xi^*_N-\tilde\xi^*_N|.
\end{align}
Since $\partial_\mu V$ has bounded derivatives, we get
\begin{align*}
 |\xi^*(t,\mu_N,m,q)-\xi^*(t,\mu_N,\tilde m,\tilde q)|\leq C(|m-\tilde m|+|q-\tilde q|).
\end{align*}
A similar analysis to \eqref{xi-Lipschitz} gives
\begin{align*}
 |\xi^*(t,\mu_N,m,q)-\xi^*(t,\hat\mu_N,m,q)|\leq C\mathcal W_2(\mu_N,\hat\mu_M).
\end{align*}
Since empirical measures are dense in $\mathcal P_2(\mathbb R^d)$, it is easy to extend the above $\xi^*(t,\cdot)$ to a continuous function that is Lipschitz on $\mathcal P_2(\mathbb R^d)\times\mathbb R^d$. Moreover, according to Lemma \ref{emp-ham} and \eqref{xi-implicit-1}, for each $\mu\in\mathcal P_2(\mathbb R^d)$ and its lifting $(M,Q)$, it holds a.s. that
\begin{align*}
 &\quad\inf_{\xi\in L^2(\Omega;\mathbb R^{d_1})}\mathbb E\big[-\partial_\mu V^{(m)}(t,\mu,M,Q)\cdot\xi+\partial_\mu V^{(q)}(t,\mu,M,Q)\cdot\big(-\phi\xi-\tilde\phi\mathbb E\xi\big)+U(\xi,Q,\mu_\xi)\big]\\
 &=\mathbb E\big[-\partial_\mu V^{(m)}(t,\mu,M,Q)\cdot\xi^*+\partial_\mu V^{(q)}(t,\mu,M,Q)\cdot\big(-\phi\xi-\tilde\phi\mathbb E\xi\big)+U(\xi,Q,\mu_\xi)\big],
 \end{align*}
 where $\xi^*=\xi^*(t,\mu,M,Q).$ The above together with $(\theta^*,\alpha^*)$ in \eqref{optimal-strategy-1} verifies Assumption \ref{assumption-saddle-points} with $\mathcal H$ in \eqref{lqd-general-ham} and $\tilde\mu:=\big(Id,\partial_\mu V(t,\mu,\cdot)\big)\sharp\mu$. Now that $\xi^*$ is Lipschitz, for any Lipschitz feedback function $(\theta,\alpha)$ and the corresponding $(M,Q)$ from \eqref{model-1}$\sim$\eqref{lqd-dyn-0}, it is now standard to verify
\begin{align*}
 \sup_{\tau\leq T}\mathbb E\bigg[\int_\tau^T\big|\eta\big(t,\mu_{(M_t,Q_t)},M_t,Q_t\big)\big|^2dt\bigg|\mathcal F_\tau\bigg]<+\infty.
\end{align*}
Thus for $\eta_s:=\eta\big(s,\mu_{(M_s,Q_s)},M_s,Q_s\big)$, $\bigg(\int_0^t\eta_sdW_s\bigg)_{t\in[0,T]}$ is a BMO martingale and $\xi^*$ is admissible.
\end{proof}

\subsection{Proofs for Theorem \ref{application-LQ-1} and \ref{verification-LQ-constrained}}
\begin{proof}[Proof for Theorem \ref{application-LQ-1}]
 It is the direct consequence of Theorems \ref{verification} and \ref{LQ-global-wp}.
\end{proof}
Next we turn to the proofs for Theorem \ref{verification-LQ-constrained}.
\begin{lemma}\label{verification-1}
 For $V^{+\infty}$ in \eqref{constraint-value}, it holds that 
 \begin{align*}
  \inf_{\xi\in\tilde{\mathcal U}^{Con}_0}J(0,m,q,\xi)\geq V^{+\infty}(0,\mu),\quad{\rm Law}(m,q)=\mu\in\mathcal P_2(\mathbb R^d).
 \end{align*}
\end{lemma}
\begin{proof}
For varying $\lambda$ in the cost function \eqref{lqd-dyn-4}, denote the value function and optimal feedback strategy of the unconstrained problem by $V^\lambda$ and $(\xi^{*,\lambda},\theta^{*,\lambda},\alpha^{*,\lambda})$ as in \eqref{lambda-value} and \eqref{lambda-feedback}. Let $\xi\in\tilde{\mathcal U}^{Con}_0$ and consider
\begin{align*}
 \left\{\begin{aligned}
  dM_s&=-\xi_sds,\\
  dQ_s&=\big[\alpha^{*,\lambda}_s+\eta^{*,\lambda}_s-\gamma_1\xi_s-\gamma_2\mathbb E\xi_s+\gamma_3M_s+\gamma_4\mathbb EM_s)\big]ds+dW_s,\ s\in[0,T],\\
  M_0=&m,\ Q_0=q.
 \end{aligned}\right.
\end{align*}
where for $\varphi=\xi,\ \alpha^{*,\lambda},\ \eta^{*,\lambda}$,
\begin{align*}
 \varphi_s:=\varphi\big(s,M_s,Q_s,\mu_{(M_s,Q_s)}\big),\ s\in[0,T].
\end{align*}
Since $(\eta^{*,\lambda},\alpha^{*,\lambda})$ is Lipschitz in its variable, by definition, $(\eta^{*,\lambda},\alpha^{*,\lambda})\in\mathcal V^\xi_0$ and $M_T=0.$ Next we may follow the same idea as in \eqref{veri-0} and show that
\begin{align*}
J(0,m,\xi)\geq\tilde J(0,m,\xi,\theta^{*,\lambda},\alpha^{*,\lambda})\geq V^\lambda\big(0,\mu_{(M_t,Q_t)}\big).
\end{align*}
In view of Theorem \ref{limit-V-lambda}, we may pass $\lambda$ to infinity in the above and get
\begin{align*}
J(0,m,q,\xi)\geq V^{+\infty}\big(0,\mu_{(M_0,Q_0)}\big).
\end{align*}
Noticing $\xi\in\tilde{\mathcal U}^{Con}_0$ is arbitrary, we have completed the proof.
\end{proof}

\begin{lemma}\label{verification-2}
 The feedback control $(\xi^*_s)_{s\in[0,T]}$ in \eqref{MS-path} is admissible.
\end{lemma}
\begin{proof}
Let $\big(M_s,Q_s\big)_{s\in[0,T]}$ solve
\begin{align}\label{MS-path-0-1}
  \left\{\begin{aligned}dM_s&=-\xi^*_sds,\\
  dQ_s&=[\alpha_s+\eta_s-\gamma_1\xi^*_s-\gamma_2\mathbb E\xi^*_s+\gamma_3M_s+\gamma_4\mathbb EM_s]ds+dW_s,\\
  M_0&=m,\quad Q_0=q.
\end{aligned}\right.
\end{align}
In \eqref{MS-path-0-1}, $(\xi^*_s,\eta_s,\alpha_s)_{s\in[0,T]}$ is a feedback strategy where $(\eta_s,\alpha_s)$ admits Lipschitz feedback function $(\eta,\alpha)$ and the feedback function of $\xi^*_s$ is from \eqref{MS-path}, i.e.,
\begin{align*}
 \xi^*_s=\frac1{2\lambda_1}\Bigg[&\left(\begin{matrix}I_{d_1} & \gamma^\top_1 \end{matrix}\right)\bigg(2a^{+\infty}_1(s)\left(\begin{matrix}M_s\\
 Q_s
 \end{matrix}\right)+2a^{+\infty}_3(s)\left(\begin{matrix}\mathbb EM_s\\
 \mathbb EQ_s
 \end{matrix}\right)\bigg)+2K_s\big(\gamma^\top_1Q_s+\gamma_2\mathbb EQ_s\big)\notag\\
 &+4\left(\begin{matrix}0&\gamma_2
 \end{matrix}\right)\bigg(\big(a^{+\infty}_1(s)+a^{+\infty}_3(s)\big)\left(\begin{matrix}\mathbb EM_s\\
 \mathbb EQ_s
 \end{matrix}\right)\bigg)+\lambda_5\gamma_5Q_s\Bigg],\ s\in[t,T].
\end{align*}
According to the definition, $\xi^*\in\tilde{\mathcal U}^{Con}_0$ if $M_T=0$ and  $\big(\int_0^s\theta_udW_u\big)_{s\in[0,T]}$ is a BMO martingale. Let us apply the mean value theorem and turn $(\eta_s)_{s\in[t,T]}$ into
\begin{align}\label{linear-theta}
\eta_s=\eta^0_s+A_sM_s+B_sQ_s,\quad s\in[t,T],
\end{align}
where for $s\in[0,T]$, $A_s\in\mathbb R^{d_2\times d_1},\ B_s\in\mathbb R^{d_2\times d_2}$ are bounded matrix value processes and
\begin{align*}
 \eta^0_s=\eta\big(s,0,0,\mu_{(M_s,Q_s)}\big).
\end{align*}
 Hence
\begin{align}\label{MS-path-1}
  \left\{\begin{aligned}dM_s&=-\xi^*_sds,\\
  dQ_s&=[\alpha_s+\eta^0_s+A_sM_s+B_sQ_s-\gamma_1\xi^*_s-\gamma_2\mathbb E\xi^*_s+\gamma_3M_s+\gamma_4\mathbb EM_s]ds+dW_s.
\end{aligned}\right.
\end{align}
In view of \eqref{uniform-estimates} and \eqref{MS-path-1}, it is easy to show that
\begin{align*}
 \frac d{ds}\mathbb E\big[|M_s|^2+|Q_s|^2\big]\leq C\mathbb E\big[|M_s|^2+|Q_s|^2\big]+C.
\end{align*}
Therefore
\begin{align*}
 \mathbb E\big[|M_s|^2+|Q_s|^2\big]\leq C,\quad s\in[t,T].
\end{align*}
Next we show $\mathbb E\big[|M_T|^2]=0.$ We first claim that for sufficiently small $\alpha\in(0,1)$, there exists a constant $C_\alpha$ such that
\begin{align}\label{MT-1}
 \big|\mathbb EM_s\big|^2\leq C_\alpha(T-s)^\alpha,\quad s\in[t,T).
\end{align}
Towards that end, define
\begin{align*}
 g(s)=\frac1{(T-s)^\alpha}\big|\mathbb EM_s\big|^2,\quad s\in[t,T),\ \alpha\in(0,1).
\end{align*}
According to \eqref{uniform-estimates} and \eqref{MS-path-1}, taking expectation in \eqref{MS-path-1} and applying It\^o's formula to $|\mathbb EM_s|^2$ yields a positive constant $C$ such that
\begin{align*}
 \frac d{ds}g(s)&\leq\frac C{(T-s)^\alpha}\bigg(\varepsilon\big((T-s)^{-1}+(e^{4\sqrt{\varepsilon}(T-s)}-1)^{-1}\big)|\mathbb EM_s|^2+\varepsilon^{-1}|\mathbb EQ_s|^2\\
 &\quad\qquad+|a^{(2)}(s)|\cdot|\mathbb EM_s|+|\mathbb EM_s|^2\bigg)+\frac\alpha{(T-s)^{1+\alpha}}\big|\mathbb EM_s\big|^2\\
 &\quad-\frac{C^{-1}}{(T-s)^\alpha}\mathbb EM^\top_sa^{(5)}_{11}(s)\mathbb EM_s,
\end{align*}
where $\delta>0$ is a constant to be determined, and
\begin{align*}
 a^{(5)}(t)=a^{(1)}(t)+a^{(3)}(t)=\left(\begin{matrix}a^{(5)}_{11}(t)&a^{(5)}_{12}(t)\\
 a^{(5)}_{21}(t)&a^{(5)}_{22}(t)
 \end{matrix}\right).
\end{align*}
In view of \eqref{uniform-estimates}, there exist positive constants $C_1,\ C_2,\ \delta$ such that
\begin{align}\label{MT-2}
C_1I_{d_1}\leq\frac{a^{(1)}_{11}(t)+a^{(3)}_{11}(t)}{(T-t)^{-1}}\leq C_2I_{d_1},\ C_1I_{d_1}\leq\frac{a^{(1)}_{11}(t)}{(T-t)^{-1}}\leq C_2I_{d_1},\quad t\in[0,T-\delta).
\end{align}
Then by \eqref{LQ-Riccati}, for any $\tilde\alpha\in(0,1)$,
\begin{align*}
 \limsup_{t\to T-}\frac{|a^{(2)}(t)|}{(T-t)^{-\tilde\alpha}}<+\infty.
\end{align*}
It is now easy to see from the above and \eqref{uniform-estimates} that there exist sufficiently small $\alpha\in(0,1)$ and $\varepsilon,\ \delta>0$ such that
\begin{align*}
 \frac d{ds}g(s)<\frac C{(T-s)^{\alpha+\tilde\alpha}}+\frac C{(T-s)^\alpha},\quad s\in[T-\delta,T).
\end{align*}
We may now choose $\alpha,\ \tilde\alpha$ sufficiently small such that $g(s)\leq C_{\alpha,\tilde\alpha}$ for some constant $C_{\alpha,\tilde\alpha}$ and thus \eqref{MT-1}. Now we claim that for sufficiently small $\beta\in(0,1)$,
\begin{align}\label{MT-3}
 \mathbb E\big[|M_t|^2]\leq C_\beta(T-t)^\beta,
\end{align}
which implies $\mathbb E\big[|M_T|^2]=0$. Define
\begin{align*}
 \tilde g(s)=\frac1{(T-s)^\beta}\mathbb E|M_s|^2,\quad s\in[t,T),\ \beta\in(0,1).
\end{align*}
In view of \eqref{uniform-estimates}, \eqref{MS-path-1} as well as \eqref{MT-1}, there exists a positive constant $C$ such that
\begin{align*}
 \frac d{ds}\tilde g(s)&\leq\frac C{(T-s)^\beta}\bigg(\varepsilon\big((T-s)^{-1}+(e^{4\sqrt{\varepsilon}(T-s)}-1)^{-1}\big)\mathbb E|M_s|^2+(\varepsilon^{-1}+1)\mathbb E|Q_s|^2+|a^{(2)}(s)|\cdot\mathbb E|M_s|\\
 &\quad\qquad+\mathbb E|M_s|^2+\big((T-s)^{-1}+(e^{4\sqrt{\varepsilon}(T-s)}-1)^{-1}\big)(T-t)^{2\alpha}\bigg)+\frac\beta{(T-s)^{1+\beta}}\mathbb E|M_s|^2\\
 &\quad-\frac{C^{-1}}{(T-s)^\beta}\mathbb E\big[M^\top_sa^{(1)}_{11}(s)M_s\big].
\end{align*}
For sufficiently small $\alpha\in(0,1)$ in \eqref{MT-1}. Let $\beta-2\alpha<0$ and $\beta,\ \varepsilon,\ \tilde\alpha$ be sufficiently small, then the above estimates yield
\begin{align*}
 \tilde g(s)<C\big((T-s)^{2\alpha-\beta}+(T-s)^{1-\beta-\tilde\alpha}+1\big),\quad s\in[0,T).
\end{align*}
Hence \eqref{MT-3}. According to \eqref{MS-path-1} and \eqref{MT-2}, 
\begin{align*}
 d\left(\begin{matrix}M_s\\
 Q_s
 \end{matrix}\right)=\tilde A_s\left(\begin{matrix}M_s\\
 Q_s
 \end{matrix}\right)+\left(\begin{matrix}0\\
 dW_s
 \end{matrix}\right),
\end{align*}
where $(\tilde A_s)_{s\in[0,T]}$ is matrix value process satisfying $\tilde A_s\in\mathbb R^{d\times d},\quad|\tilde A_s|\leq\frac C{T-s},\ s\in[0,T]$ for some constant $C$. Then
\begin{align*}
 \left(\begin{matrix}M_s\\
 Q_s
 \end{matrix}\right)=\Phi^-_s \left(\begin{matrix}M_0\\
 Q_0
 \end{matrix}\right)+\int_0^s\Phi^-_s\Phi^+_u\left(\begin{matrix}0\\
 dW_u
 \end{matrix}\right),
\end{align*}
where
\begin{align*}
 \frac d{ds}\Phi^+_s=\Phi^+_s\tilde A_s,\quad\frac d{ds}\Phi^-_s=-\Phi^-_s\tilde A_s,\quad\Phi_0=I_d.
\end{align*}
It is easy to see from Gronwall's inequality that the existence of $C>0$ such that
\begin{align*}
 |\Phi^-_s\Phi^+_u|\leq(T-s)^C,\quad u,s\in[0,T],\ u\leq s.
\end{align*}
Therefore, it is straightforward to check that $\big(\int_0^sM_udW_u\big)_{s\in[t,T]}$ and $\big(\int_0^sS_udW_u\big)_{s\in[t,T]}$ are BMO martingales. Hence, by \eqref{linear-theta} $\big(\int_0^s\eta_udW_u\big)_{s\in[t,T]}$ is also a BMO martingale.
\end{proof}

\begin{proof}[Proof of Theorem \ref{verification-LQ-constrained}]
After Lemma \ref{verification-1} and Lemma \ref{verification-2}, it remains to show
\begin{align*}
J(0,m,q,\xi^*)=V(0,\mu).
\end{align*}
Set $\mu_t:={\rm Law}(M_t,Q_t),\ t\in[0,T]$. Applying It\^o's formula to $(\mu_t)_{t\in[0,T]}$ in \eqref{MS-path-0-1} and repeating the similar reasoning to that in \eqref{veri-0-1}, we get
\begin{align*}
 &V^{+\infty}(T-\delta,\mu_{T-\delta})+\mathbb E\bigg[\lambda_1\int_0^{T-\delta}|\xi_s|^2ds+\lambda_{21}\int_0^{T-\delta}|M_s|^2ds+\lambda_{22}\int_0^{T-\delta}|M_s-\mathbb EM_s|^2ds\\
 &\qquad\qquad\qquad-\int_0^{T-\delta}\big(\lambda_3|\eta_s|^2+\lambda_4|\alpha_s|^2\big)ds-\lambda_5\int_0^{T-\delta}\xi_s\cdot Q_sds\bigg]\leq V(0,\mu).
\end{align*}
Denote by $\zeta_s:=(M_s,S_s)^\top,\ s\in[0,T]$. In view of Theorem \ref{limit-V-lambda},
{\small\begin{align*}
 &\quad V(T-\delta,\mu_{T-\delta})=\mathbb E[\zeta^\top_{T-\delta}]\big(a^{+\infty}_1(T-\delta)+a^{+\infty}_3(T-\delta)\big)\mathbb E[\zeta_{T-\delta}]+\mathbb E\big[(\zeta_{T-\delta}-\mathbb E\zeta_{T-\delta})^\top a^{+\infty}_1(T-\delta)(\zeta_{T-\delta}-\mathbb E\zeta_{T-\delta})\big]\\
 &\geq\mathbb E[Q^\top_{T-\delta}]e^{-\delta q^T_{22}}q^T_{22}e^{-\delta q^T_{22}}\mathbb E[Q_{T-\delta}]+\mathbb E\big[(Q_{T-\delta}-\mathbb EQ_{T-\delta})^\top e^{-\delta q^T_{22}}q^T_{22}e^{-\delta q^T_{22}}(Q_{T-\delta}-\mathbb EQ_{T-\delta})\big]-\delta C\\
 &=\mathbb E[Q_{T-\delta}e^{-\delta q^T_{22}}q^T_{22}e^{-\delta q^T_{22}}Q_{T-\delta}]-\delta C.
\end{align*}}
Sending $\delta$ to $0$, we get
\begin{align*}
 &\mathbb E\bigg[Q_Tq^T_{22}Q_T+\lambda_1\int_0^T|\xi_t|^2dt+\lambda_{21}\int_0^T|M_t|^2dt+\lambda_{22}\int_0^T|M_t-\mathbb EM_t|^2dt\\
 &\qquad\qquad\qquad-\int_0^T\big(\lambda_3|\eta_t|^2+\lambda_4|\alpha_t|^2\big)dt-\lambda_5\int_0^T\xi_tQ_tdt\bigg]\leq V(0,\mu).
\end{align*}
Since $(\eta,\alpha)$ is arbitrary, we obtain
\begin{align*}
J(0,m,q,\xi^*)=V(0,\mu).
\end{align*}
\end{proof}

{\bf Acknowledgements.} HL is partially supported by NSFC grant 12522122 and Hong Kong RGC Grant ECS 21302521. SL is partially supported by the NSFC/RGC Joint Research Scheme under grant P0031382/S-ZG9U and the Postdoc Matching Fund 1-W32B. CM is supported in part by NSFC grant 12522122, NSFC/RGC JRS N\_CityU165/25, GRF 11311422, and GRF 11303223. DS is supported in part  by the Research Center for Intelligent Operations Research and Hong Kong RGC Senior Research Fellow Scheme No. SRFS2223-5S02. 

\bibliographystyle{plain}

\end{document}